\documentclass[12pt]{amsart}
\usepackage{amsmath, amsthm, amssymb}
\usepackage{fullpage}
\usepackage{color}
\usepackage{hyperref}

\newtheorem{theorem}{Theorem}
\newtheorem{lemma}{Lemma}
\newtheorem{proposition}[lemma]{Proposition}

\newtheorem{remark}[lemma]{Remark}

\numberwithin{lemma}{section}

\newcommand{\Ez}{E_0}
\newcommand{\E}{E}
\newcommand{\Elind}{E^{(2)}_{lin}}
\newcommand{\Elint}{E^{(3)}_{lin}}

\newcommand{\End}{E^{n,(2)}}
\newcommand{\Ent}{E^{n,(3)}}
\newcommand{\Enthigh}{E^{n,(3)}_{high}}
\newcommand{\Ennfz}{E^{n}_{NF,0}}
\newcommand{\Ennf}{E^{n}_{NF}}
\newcommand{\Ennfhigh}{E^{n}_{NF,high}}
\newcommand{\Ennflow}{E^{n}_{NF,low}}
\newcommand{\M}{\mathfrak M}
\newcommand{\cM}{{\mathcal M}}
\newcommand{\cK}{{\mathcal K}}

\numberwithin{equation}{section}

\newcommand{\R}{{\mathbb R}}

\newcommand{\Z}{{\mathbb Z}}
\newcommand{\tW}{{\tilde W}}
\newcommand{\tQ}{{\tilde Q}}

\renewcommand{\H}{{\mathcal H }}
\renewcommand{\AA}{\mathbf A}
\newcommand{\WH}{{\mathcal{WH} }}

\renewcommand{\R}{\mathbf R}
\newcommand{\XX}{X}

\newcommand{\Trunc}{{{\Lambda}}}

\newcommand{\tG}{\tilde{G}}
\newcommand{\tK}{\tilde{K}}

\newcommand{\Cx}{\mathbb{C}}
\newcommand{\Ha}{\mathbb{H}}
\newcommand{\con}{\mathcal{F}}
\newcommand{\pot}{\Psi}
\newcommand{\hpot}{\Theta}
\newcommand{\Hi}{H}
\newcommand{\Id}{I}
\newcommand{\sgn}{\mathop{\mathrm{sgn}}}

\newcommand{\tw}{{\tilde w}}
\newcommand{\tr}{{\tilde r}}
\newcommand{\tR}{{\tilde {\mathbf R}}}
\newcommand{\dH}{{\dot{\mathcal H} }}

\newcommand{\W}{{\mathbf W}}

\renewcommand{\S}{{\mathbf S}}
\newcommand{\err}{\text{\bf err}}
\newcommand{\errw}{\text{\bf err}(L^2)}
\newcommand{\errr}{\text{\bf err}(\dot H^\frac12)}
\newcommand{\norm}{{\mathbf N}}
\renewcommand\R{\mathbf R}

\begin{document}

\title{Two dimensional water waves in holomorphic coordinates}

\author{John K. Hunter}
\address{ Department of Mathematics, University of California at Davis}
\thanks{The first author was partially supported by the NSF under grant number  DMS-1312342.}
\email{hunter@math.davis.edu}
\author{Mihaela Ifrim}
 \address{Department of Mathematics, University of California at Berkeley}.
\email{ifrim@math.berkeley.edu}
\thanks{The second author was supported by the National Science Foundation under Grant No. 0932078 000, while the author was in residence at the Mathematical Science Research Institute in Berkeley, California, during the Fall semester 2013.}
\author{ Daniel Tataru}
\address{Department of Mathematics, University of California at Berkeley}
 \thanks{The third author was partially supported by the NSF grant DMS-1266182
as well as by the Simons Foundation}
\email{tataru@math.berkeley.edu}

\begin{abstract}
  This article is concerned with the infinite depth water wave
  equation in two space dimensions.  We consider this problem
  expressed in position-velocity potential holomorphic coordinates.
  Viewing this problem as a quasilinear dispersive equation, we
  establish two results: (i) local well-posedness in Sobolev spaces,
  and (ii) almost global solutions for small localized data. Neither
  of these results are new; they have been recently obtained by
  Alazard-Burq-Zuily~\cite{abz}, respectively by Wu~\cite{wu} using different coordinates
  and methods. Instead our goal is improve the understanding of this
  problem by providing a single setting for both problems, 
by proving sharper versions of the above results,  as well as
presenting new,  simpler proofs. This article is self contained.

\end{abstract}

\maketitle

\section{Introduction}
We consider the two dimensional water wave equations with infinite
depth with gravity but without surface tension.  This is governed by
the incompressible Euler's equations with boundary conditions on the
water surface. Under the additional assumption that the flow is
irrotational the fluid dynamics can be expressed in terms of a
one-dimensional evolution of the water surface coupled with the trace
of the velocity potential on the surface.  

This problem was previously considered by several other authors. The
local in time existence and uniqueness of solutions was proved in
\cite{n,y, wu2}, both for finite and infinite depth.  Later, Wu~\cite{wu} proved 
almost global existence for small localized data.
Very recently, global results for small localized data were independently obtained by
Alazard $\&$ Delort \cite{ad} and by Ionescu $\&$ Pusateri
\cite{ip}. Extensive work was also done on the same problem in three
or higher space dimensions, and also on related problems with surface
tension, vorticity, finite bottom, etc. Without being exhaustive, we
list some of the more recent references \cite{abz, abz1, chs, cl,cs,
  DL, HL, o, sz, zz}.

Our goal here is to revisit this problem and to provide a new,
self-contained approach which, we hope, considerably simplifies and
improves on  many of the results mentioned above. Our analysis is based
on the use of holomorphic coordinates, which are described below.
Our results include:

(i) local well-posedness in Sobolev spaces, improving on previous regularity thresholds,
e.g. in \cite{abz},
up to the point where the transport vector field is no longer Lipschitz, and
 has merely a $BMO$ derivative.  

(ii) cubic life-span bounds for small data. These are related to the normal form 
method, but are instead proved by a modified energy method, inspired  from the authors' 
previous article \cite{BH}. 

(iii) almost global well-posedness for small localized data, refining and 
simplifying Wu's approach in \cite{wu}.

We consider both the case of the real line $\mathbb{R}$ and the periodic case
$\mathbb{S}^1$.  Our equations are expressed in coordinates $(t,\alpha)$ where
$\alpha$ corresponds to the holomorphic parametrization of the water
domain by the lower half-plane restricted to the real line.  To write
the equations we use the Hilbert transform $H$, as well as the
operator
\begin{equation*}
P= \frac12(I-iH).
\end{equation*}
Note that $P$ is a projector in $\mathbb{R}$ but not on $\mathbb{S}^1$. 

Our variables $(Z,Q)$ represent the position of the water surface,
respectively the holomorphic extension of the velocity
potential. These will be restricted to the closed subspace of
holomorphic functions within various Sobolev spaces. Here we define
holomorphic functions on $\mathbb{R}$ or on $\mathbb{S}^1$ as those
whose Fourier transform is supported in $(-\infty,0]$; equivalently,
they admit a bounded holomorphic extension into the lower
half-space. On $\mathbb{R}$ this can be described by the relation $Pf
= f$, but on $\mathbb{S}^1$ we also need to make some adjustments for
the constants.

 There is a one dimensional degree of freedom in the choice of $\alpha$, namely the horizontal translations. To fix this, in the real case we are considering waves which either decay at infinity,
\[
\lim_{|\alpha \vert \to \infty} Z(\alpha) - \alpha = 0.
\]
In the  periodic case we instead assume that $ Z(\alpha) - \alpha$  has period $2\pi$ and purely imaginary average. We can also harmlessly assume that $Q$ has real average.

In position-velocity potential holomorphic coordinates the equations have the form
\begin{equation*}
\left\{
\begin{aligned}
& Z_t + F Z_\alpha = 0, \\
& Q_t + F Q_\alpha -i (Z-\alpha) + P\left[ \frac{|Q_\alpha|^2}{J}\right]  = 0, \\
\end{aligned}
\right.
\end{equation*}
where
\begin{equation*}
 F = P\left[ \frac{Q_\alpha - \bar Q_\alpha}{J}\right] , \qquad J = |Z_\alpha|^2.
\end{equation*}

For the derivation of the above equations, we refer the reader to \emph{Appendix}~\ref{holom-eq}. In the real case these equations originate in \cite{ov}.
The changes needed for the periodic case are also described in the same \emph{Appendix}~\ref{holom-eq}. There are also other ways of expressing the equations, for instance in Cartesian coordinates using the Dirichlet to Neumann map associated to the water domain, see e.g. \cite{abz} . Here we prefer the holomorphic coordinates due to the simpler form of the equations; in particular, in these coordinates the Dirichlet to Neumann map is given in terms of the standard Hilbert transform.

It is convenient to work with a new variable, namely
\[
W = Z-\alpha .
\]
The equations become
\begin{equation}
\label{ww2d1}
\left\{
\begin{aligned}
& W_t + F (1+W_\alpha) = 0, \\
& Q_t + F Q_\alpha -i W + P\left[ \frac{|Q_\alpha|^2}{J}\right]  = 0, \\
\end{aligned}
\right.
\end{equation}
where
\begin{equation*}
 F = P\left[\frac{Q_\alpha - \bar Q_\alpha}{J}\right], \qquad J = |1+W_\alpha|^2.
 \end{equation*}
These equations are considered either in $\mathbb{R} \times \mathbb{R}$ or in $\mathbb{R} \times \mathbb{S}^1$.

As the system \eqref{ww2d1} is fully nonlinear, a standard procedure is to convert it into a quasilinear system by differentiating it. Observing that almost no undifferentiated functions appear in  \eqref{ww2d1}, one sees that by differentiation we get a self-contained  first order quasilinear system for $(W_\alpha,Q_\alpha)$. To write this system we introduce the auxiliary real function  $b$, which we call the {\em  advection velocity}, and is given by
\begin{equation*}
b = P \left[\frac{{Q}_\alpha}{J}\right] +  \bar P\left[\frac{\bar{Q}_\alpha}{J}\right].
\end{equation*}
The reason for this will be immediately apparent. Using $b$, the system \eqref{ww2d1} is written in the form
\begin{equation*}
\left\{
\begin{aligned}
&W_t + b (1+ W_\alpha) =  \frac{\bar Q_\alpha}{1+\bar W_\alpha},
\\
&Q_t +  b Q_\alpha - iW =  \bar P\left[ \frac{|Q_\alpha|^2}{J}\right],
\end{aligned}
\right.
\end{equation*}
where the terms on the right are antiholomorphic and disappear when the equations are projected onto the holomorphic space. Differentiating with respect to $\alpha$ yields a system for $(W_\alpha,Q_\alpha)$, namely
\begin{equation*}
\left\{
\begin{aligned}
  &W_{\alpha t} + b W_{\alpha \alpha} +
    \frac{1}{1+\bar W_\alpha}\left(Q_{\alpha\alpha} - \frac{Q_\alpha}{1+W_\alpha}
      W_{\alpha \alpha}\right) = - (1+W_\alpha) \bar F_\alpha
- \left[\frac{\bar Q_\alpha}{1+\bar W_\alpha}\right]_\alpha,
 \\
  &Q_{t\alpha} + b Q_{\alpha\alpha} - iW_\alpha +
    \frac{1}{1+\bar W_\alpha} \frac{Q_\alpha}{1+W_\alpha}\!\!\left(\!\!Q_{\alpha\alpha} - \frac{Q_\alpha}{1+W_\alpha}  W_{\alpha \alpha}\!\! \right)\!\! = - Q_\alpha \bar F_\alpha+ \bar P\left[ \frac{|Q_\alpha|^2}{J}\right]_\alpha.
\end{aligned}
\right.
\end{equation*}
The terms on the right are mostly antiholomorphic and can be viewed as lower order when projected on the holomorphic functions. Examining the expression on the left one easily sees that the above first order system is degenerate, and has a double speed $b$. Then it is natural to diagonalize it. This is done using the operator
\begin{equation}
\AA(w,q) := (w,q - Rw), \qquad R := \frac{Q_\alpha}{1+W_\alpha}.
\label{defR}
\end{equation}
The factor $R$ above has an intrinsic meaning, namely it is the complex velocity on the water surface. We also remark that
\[
\AA(W_\alpha,Q_\alpha) = (\W,R), \qquad \W : = W_\alpha.
\]
Thus, the pair $(\W,R)$ diagonalizes the differentiated system. Indeed, a direct computation yields the self-contained system
\begin{equation} \label{ww2d-diff}
\left\{
\begin{aligned}
 & \W_{ t} + b \W_{ \alpha} + \frac{(1+\W) R_\alpha}{1+\bar \W}   =  (1+\W)M,
\\
& R_t + bR_\alpha = i\left(\frac{\W - a}{1+\W}\right),
\end{aligned}
\right.
\end{equation}
where the real {\em frequency-shift} $a$ is given by
\begin{equation}
a := i\left(\bar P \left[\bar{R} R_\alpha\right]- P\left[R\bar{R}_\alpha\right]\right),
\label{defa}
\end{equation}
and  the auxiliary function $M$ is given by
\begin{equation}\label{M-def}
M :=  \frac{R_\alpha}{1+\bar \W}  + \frac{\bar R_\alpha}{1+ \W} -  b_\alpha =
\bar P [\bar R Y_\alpha- R_\alpha \bar Y]  + P[R \bar Y_\alpha - \bar R_\alpha Y].
\end{equation}
The function $Y$ above, given by
\[
Y := \frac{\W}{1+\W},
\]
is introduced in order to avoid rational expressions above and in many places in the sequel.  The system \eqref{ww2d-diff} governs an evolution in the space of holomorphic functions, and will be used both directly and in its projected version.

Incidentally, we note that when expressed in terms of $(Y,R)$ the water wave system becomes purely polynomial, see also \cite{zakharov2},
\begin{equation*}
\left\{
\begin{aligned}
& Y_t + b Y_\alpha + |1-Y|^2 R_\alpha = (1-Y)  M,
  \\ & R_t + b R_\alpha - i(1+a) Y  = - ia,
\end{aligned}
\right.
\end{equation*}
where $M$ is as above, and
\[
b = 2 \Re ( R - P(R \bar Y)), \qquad a = 2 \Re P(R \bar R_\alpha).
\]
However, we do not take advantage of this formulation in the present article. 

The functions $b$ and $a$ also play a fundamental role in the linearized equation which  is computed in the next section, Section~\ref{s:linearized}. The linearized variables are denoted by $(w,q)$ and, after the diagonalization, $(w, r:=q-Rw)$. The linearized equation, see \eqref{lin(wr)0}, has the form
\begin{equation}
\label{lin(wr)00}
\left\{
\begin{aligned}
& (\partial_t + b \partial_\alpha) w  +  \frac{1}{1+\bar \W} r_\alpha
+  \frac{R_{\alpha} }{1+\bar \W} w  =\ (1+\W) (P \bar m + \bar P  m),
 \\
&(\partial_t + b \partial_\alpha)  r  - i  \frac{1+a}{1+\W} w  = \  \bar P n - P \bar n,
\end{aligned}
\right.
\end{equation}
 where 
 \[ 
 m := \frac{r_\alpha +R_\alpha w}{J} + \frac{\bar R w_\alpha}{(1+\W)^2}, \qquad n := \frac{ \bar R(r_{\alpha}+R_\alpha w)}{1+\W}.
\]

In particular, we remark that the linearization of the system \eqref{ww2d-diff} around the zero solution is 
\begin{equation} \label{ww2d-0} \left\{
\begin{aligned}
 & w_{ t} +  r_\alpha   =  0,
\\
& r_t - i w = 0.
\end{aligned}
\right.
\end{equation}
The analysis of the linearized equation, carried out in Section~\ref{s:linearized}, is a key component of this paper. 

It is also useful to  further differentiate \eqref{ww2d-diff}, in order to obtain a system for $(\W_{\alpha}, R_{\alpha})$:
\[
\left\{
\begin{aligned}
 & \W_{ \alpha t} + b \W_{ \alpha \alpha}  + \frac{[(1+\W) R_\alpha]_\alpha}{1+\bar \W}   =
- b_\alpha \W_\alpha +  (1+\W) R_\alpha \bar Y_\alpha +
 \W_\alpha  M
 + (1+\W) M_\alpha,
 \\
 & R_{t\alpha} + bR_{\alpha\alpha}  =- b_\alpha R_\alpha+  i\left(\frac{(1+a)\W_\alpha}{(1+\W)^2}
-  \frac{ a_\alpha}{1+\W}\right).
\end{aligned}
\right.
\]
In order to better compare this with the linearized system we introduce the modified variable $\R := R_\alpha(1+\W)$ to get the system
\[
\left\{
\begin{aligned}
 & \W_{ \alpha t} + b \W_{ \alpha \alpha}  + \frac{\R_\alpha}{1+\bar \W}   =
- b_\alpha \W_\alpha  +   \R \bar Y_\alpha +
 \W_\alpha  M
 + (1+\W) M_\alpha ,
\\
& \R_{t} + b\R_{\alpha}  =-\left(b_\alpha+ \frac{R_\alpha}{1+\bar \W}\right)  \R +  i\left(\frac{(1+a)\W_\alpha}{1+\W}
-  a_\alpha\right) + \R M .
\end{aligned}
\right.
\]
Expanding the $b_\alpha$  terms via \eqref{M-def} this yields
\begin{equation}
\left\{
\begin{aligned}
&  \W_{ \alpha t} + b \W_{ \alpha \alpha}  + \frac{\R_\alpha}{1+\bar \W} + \frac{R_\alpha}{1+\bar \W} \W_\alpha  = G_2,
\\
&  \R_{t} + b\R_{\alpha} -   i\frac{(1+a)\W_\alpha}{1+\W}  =K_2,
\end{aligned}
\right.
\label{WR-diff}
\end{equation}
where
\begin{equation*}
\left\{
\begin{aligned}
&  G_2  = \R \bar Y_\alpha - \frac{\bar R_\alpha}{1+\W} \W_\alpha + 2M   \W_\alpha  +  (1+\W)  M_\alpha ,
\\
&  K_{2}=-2\left(\frac{\bar R_\alpha}{1+\W}+ \frac{R_\alpha}{1+\bar \W}\right)  \R  + 2 M \R + ( R_\alpha \bar R_\alpha -i a_\alpha).
\end{aligned}
\right.
\end{equation*}

Next, we define our function spaces. The system \eqref{ww2d-0} is a well-posed linear evolution in the space $\dH_0$ of holomorphic functions endowed with the $L^2 \times \dot H^{\frac12}$ norm. A conserved energy for this system is
\begin{equation}\label{E0}
\Ez (w,r) = \int \frac12 |w|^2 + \frac{1}{2i} (r \bar r_\alpha - \bar r r_\alpha) d\alpha.
\end{equation}
The nonlinear system \eqref{ww2d1} also admits a conserved energy, which has the form
\begin{equation}\label{ww-energy}
\E(W,Q) = \int \frac12 |W|^2 + \frac1{2i} (Q \bar Q_\alpha - \bar Q Q_\alpha)
- \frac{1}{4} (\bar W^2 W_\alpha + W^2 \bar W_\alpha)\, d\alpha.
\end{equation}
 As suggested by the above energy, our main function spaces for the 
 differentiated water wave system \eqref{ww2d-diff} are the
spaces $\dH_n$ endowed with the norm
\[
\| (\W,R) \|_{\dH_n}^2 := \sum_{k=0}^n
\| \partial^k_\alpha (\W,R)\|_{ L^2 \times \dot H^\frac12}^2,
\]
where $n \geq 1$. As an auxiliary step, we will also consider solutions $(\W,R)$ in the smaller space
\[
\H_n := H^n \times H^{n+\frac12},
\]
with $n \geq 2$. 

To describe the lifespan of the solutions we define the control norms
\begin{equation}\label{A-def}
A := \|\W\|_{L^\infty}+\| Y\|_{L^\infty} + \||D|^\frac12 R\|_{L^\infty \cap B^{0,\infty}_{2}},
\end{equation}
respectively
\begin{equation}\label{B-def}
B :=\||D|^\frac12 \W\|_{BMO} + \| R_\alpha\|_{BMO}.
\end{equation}
where $|D|$ represents the multiplier with symbol $|\xi|$.
Here $A$ is a scale invariant quantity, while $B$ corresponds to the
homogeneous $\dH_1$ norm of $(\W, R)$. We note that $B$ and all but
the $Y$ component of $A$ are controlled by the $\dH_1$ norm of the
solution. 


Now we are ready to state our main local well-posedness result:
\begin{theorem}
\label{baiatul}
 Let  $ n \geq 1$.  The system \eqref{ww2d-diff} is locally well-posed
for data in $\dH_n(\mathbb{R})$ so that $|\W+1| > c > 0$ .
Further, the solution can be continued for as long as $A$ and $B$ 
remain bounded. The same result holds in the periodic setting.
\end{theorem}
In terms of Sobolev regularity of the data, this result improves the
thresholds in earlier results of Wu~\cite{wu2, wu} and
Alazard-Burq-Zuily~\cite{abz}. However, a direct comparison is nontrivial
due to the fact that the above two papers use different coordinate
frames, namely Lagrangian, respectively Eulerian.

As an interesting side remark, the above result makes no requirement that the 
curve $\{ Z(\alpha); \alpha \in \mathbb{R}\}$ determined by $\W$ be nonself-intersecting.
If self-intersections occur then the  physical interpretation is lost, but 
the well-posedness of the system \eqref{ww2d-diff} is not affected.

Our second goal in this article is to consider the question of
obtaining improved lifespan bounds for the small data problem. Since
the nonlinearities in our equations contain quadratic terms, the
standard result is to obtain an $O(\epsilon^{-1})$ lifespan for smooth
initial data of size $\epsilon$. However, this problem has the
additional feature that there exists a quadratic normal form
transformation which eliminates the quadratic terms in the
equation. In the setting of holomorphic coordinates considered in this
paper, this is most readily seen at the level of the system
\eqref{ww2d1}. There, the quadratically nonlinear terms may be removed
from the water-wave equations by the near-identity, normal form
transformation
\begin{equation}
\tilde W = W - 2 \M_{\Re W} W_\alpha, \qquad \tilde Q = Q - 2 \M_{\Re W} R,
\label{nft1}
\end{equation}
where the holomorphic multiplication operator $\M_f$ is given by $\M_f
g = P\left[ fg\right] $.  For a more symmetric form of this
transformation, one can replace $R$ by $Q_{\alpha}$. However, it is
more convenient to use the diagonal variable $R$. For $(\tW,\tQ)$ we
have
\begin{proposition}\label{p:normal}
The normal form variables (\ref{nft1}) satisfy equations of the form
\begin{equation}
\left\{
\begin{aligned}
&\tW_t + \tQ_\alpha = \tG,
\\
&\tQ_t - i \tW =  \tK,
\end{aligned}
\right.
\label{nft1eq}
\end{equation}
where $\tG$, $\tK$ are cubic (and higher order) functions of
$(W, \W,R, \W_{\alpha}, R_{\alpha})$, given by
\begin{equation}
\label{gk-tilde}
\left\{
\begin{aligned}
\tilde G = & \  2P[ (F - R)_\alpha \Re W + \W_{ \alpha} F\Re W + \W \Re (\W F)+F_{\alpha }\W\Re W] \\
& \ -  P[\bar \W R \bar Y - \W(P[\bar R Y]+\bar P[R \bar Y])],
\\
\tilde K = & \  P\left[ (\bar{F}(1+\bar{\W})-\bar{R})R+2iP\left[ \frac{\W^2+a}{1+\W}\right]\cdot \Re W +2P\left[ bR_{\alpha}\right]\cdot\Re W  \right]. 
\end{aligned}
\right.
\end{equation}
\end{proposition}
  The proof is straightforward; one rewrites the system \eqref{ww2d1}
  in terms of the normal form variables $(\tilde{W}, \tilde{Q})$,
  \eqref{nft}. The original variables are $(W,Q)$, but the derivatives
  of $Q$ from the perturbative terms $G$ and $K$ are expressed in terms of $R$
 and eliminated. We also make use of the identity $P+\bar{P}=I$. The details are 
left for the reader. We note that the difference $R-F$ is quadratic,
\[
R-F = P[R \bar Y- \bar R Y].
\]

Heuristically, having cubic nonlinearities yields an improved
$O(\epsilon^{-2})$ lifespan for initial data of size $\epsilon$.
However, implementing this idea directly is fraught with
difficulties. To start with, while $\tG$, $\tK$ are cubic and higher
order terms they also depend on higher-order derivatives of $(W,Q)$;
thus it is not possible to directly close energy estimates for the
normal form variables $(\tW,\tQ)$.  This is related to the fact that
the normal form transformation \eqref{nft1} is not invertible, and
further to the fact that the system \eqref{ww2d1} is fully nonlinear,
as opposed to semilinear.

There are at least two existing methods in the literature which
attempt to address this difficulty.  One such method, introduced by
Wu~\cite{wu}, is based on the idea that any transformation which agrees
quadratically with the above normal form transform will have the same
effect as the normal form transform, but perhaps one can also choose
such a transformation such that it is invertible. In Wu's work this
transformation is an implicit change of coordinates, which is further
followed by a secondary normal form transformation. A related example
where an implicit change of coordinates is fully sufficient appears in
the work \cite{hi} of the first two authors for the related
Burgers-Hilbert problem.

 A second method, which appears in the work of Shatah etc \cite{s}, is based on
 a mix of quadratic energy estimates for high derivatives of the
 solutions, combined with a normal form method for low
 derivatives. This works well for water waves in dimension three, but
 is not precise enough for the two dimensional problem.

In the present paper we propose an alternative approach for two
dimensional water waves, which seems to be both simpler and more
accurate. Precisely, rather than  attempting to modify the equations
using a normal form transform, we instead construct modified energy
functionals which have cubic accuracy. A significant advantage of this idea is
that it applies even for the leading order energy functionals, which to our knowledge is
new.  In a simpler setting, this method was first introduced by the authors
in \cite{BH} in the context of the Burgers-Hilbert problem.

Our first result is translation invariant, and yields a cubic lifespan bound.

\begin{theorem}
\label{t:cubic}
  Let $\epsilon \ll 1$.  Assume that the initial data for the equation
  \eqref{ww2d-diff} on either $\mathbb{R}$ or $\mathbb{S}^1$ satisfies
\begin{equation}
\|(\W(0), R(0))\|_{\dH_1} \leq \epsilon.
\end{equation}
Then the solution exists on an $\epsilon^{-2}$ sized time interval
$I_\epsilon = [0,T_\epsilon]$ , and satisfies a similar bound. In
addition, the estimates
\[
\sup_{t \in I_\epsilon} \| (\W(t), R(t))\|_{\dH_n} \lesssim \| (\W(0), R(0))\|_{\dH_n},
\qquad n \geq 2,
\]
hold whenever the right hand side is finite.
\end{theorem}

Our second result assumes some additional localization for the initial
data, and establishes almost global existence of solutions.  This
applies only for the problem on $\mathbb{R}$, and relies on the dispersive
properties of the linear equation \eqref{ww2d-0}, whose solutions with
localized data have $t^{-\frac12}$ dispersive decay. To state the
result we need to return to the original set of variables $(W,Q)$.  We
also take advantage of the scale invariance of the water wave
equations. Precisely, it is invariant with respect to the scaling law
\[
(W(t,\alpha), Q(t,\alpha)) \to (\lambda^{-2} W(\lambda t,\lambda^2 \alpha),
 \lambda^{-3} Q(\lambda t,\lambda^2 \alpha)).
\]
This suggests that we should use the scaling vector field
\[
S = t \partial_t + 2 \alpha \partial_\alpha,
\]
and its action on the pair $(W,Q)$, namely
\[
\S(W,Q) = ((S-2)W,(S-3)Q).
\]

However, these are not the correct diagonal variables; to
diagonalize we use the notations 
\[
(w,r) =: \AA\S(W,Q).
\]
Then $(\W, R)$ solve the linearized equations \ref{lin(wr)00}
and  define the weighted energy
\begin{equation}\label{WH}
\|(W,Q)(t)\|_{\WH}^2 :=  \|(W,Q)(t)\|_{\dH_0}^2 + \|(\W,R)(t)\|_{\dH_5}^2 + \|(w,r)(t)\|_{\dH_0}^2.
\end{equation}

Then we have
\begin{theorem} \label{t:almost}
There exists $ c > 0$ so that   for each initial data
$(W(0),Q(0))$ for the system \eqref{ww2d1} satisfying
\begin{equation}\label{data}
\|(W,Q)(0)\|_{\WH}^2  \leq \epsilon \ll 1,
\end{equation}
the solution exists up to time $T_\epsilon = e^{c\epsilon ^{-2}}$
and satisfies
\begin{equation}\label{almost-e}
\|(W,Q)(t)\|_{\WH}^2  \lesssim \epsilon, \qquad |t| < T_\epsilon.
\end{equation}
as well as 
\begin{equation}\label{almost-e-point}
|W|+|W_\alpha| + ||D|^\frac12 W_{\alpha}| + |R| + |R_\alpha|    \lesssim 
\frac{\epsilon}{\langle t \rangle^\frac12}, \qquad |t| < T_\epsilon.
\end{equation}
\end{theorem}

This lifespan bound was originally established by Wu~\cite{wu}. Here, we
prove the same result under less restrictive assumptions, and,
hopefully, with a simpler proof. We should also mention here the recent
work of Ionescu-Pusateri~\cite{ip0},\cite{ip} and Alazard-Delort~\cite{ad}, where
global well-posedness is proved for small localized data. In a follow-up paper
we provide a simplified proof of this result as well.

While our research for this paper was largely complete by the time \cite{ip}
and \cite{ad} appeared, there is one idea from Ionescu and Pusateri's article \cite{ip}
which we adopted here in order to shorten the exposition; this is the fact 
that in order to close the estimates it suffices to use a single iteration of the scaling
vector field $S$. However, our implementation of this idea is different from 
\cite{ip}, and also more efficient, in the sense that we use no higher derivatives 
of $\S(W,Q)$. 

For the reminder of the introduction we provide a brief outline of the
paper.  The first step of the analysis is to study the linearization
of the equation \eqref{ww2d1}; this is done in
Section~\ref{s:linearized}. We begin with the diagonalisation of the
linearized equations; this in turn leads to energy estimates, which
are crucial in the proof of the local well-posedness result. The
linearized energy functional is then refined so that cubically
nonlinear estimates can be proved; this is essential in the proof of
the improved lifespan result.  We make no use of dispersive decay in
this normal form analysis, so it works also for spatially periodic
solutions. The low regularity threshold is reached by using various
bilinear Coifman-Meyer type estimates, as well as multilinear versions
thereof.


In Section~\ref{s:ee} we consider the equations for higher order
derivatives of the solution. The principal part of these equations is
closely related to the linearized equations studied in the previous
section. After some normalization, the quadratic bounds follow
directly from the ones for the linearized equation. The emphasis there
is again on obtaining cubically nonlinear estimates. The essential
idea is to construct a modified energy functional with better
estimates. Our modified energy essentially combines the
linearized energy, for the leading part, with the cubic normal form energy
for the lower order terms. This is similar to the approach  in the paper \cite{BH}
devoted to the Burgers-Hilbert problem.

Section~\ref{s:lwp} contains the proof of the local well-posedness
result.  We begin with more regular data, both in terms of low
frequencies and in terms of high frequencies. For such data, a
standard mollifier technique suffices in order to establish
well-posedness. The rough $\dH_1$ solutions are obtained as uniform
limits of smooth solutions by using the estimates for the linearized
equation.  The same construction yields their continuous dependence on
data.

In Section~\ref{s:cubic} we prove the cubic lifespan bounds for small initial data in Theorem~\ref{baiatul}.

In Section~\ref{s:decay} we provide the proof of the long time results.
The cubic lifespan result is a straightforward consequence of the
cubic energy estimates. The proof of the almost global result is slightly
more involved, as it requires, as an intermediate step, to prove
the $t^{-\frac12}$ dispersive decay for a limited number of derivatives
of $(\W,R)$. These bounds are obtained from the vector field
energy estimates,  essentially in an elliptic fashion via Sobolev type
embeddings.

Appendix~\ref{holom-eq} includes, for reader's convenience, a
complete derivation of the holomorphic water wave equations.  Finally,
Appendix~\ref{s:multilinear} contains a collection of bilinear,
multilinear and commutator estimates which are used at various places
in the paper. We are grateful to Camil Muscalu for useful conversations
pointing us in the right direction for this last section.


\section{ The linearized equation}
\label{s:linearized}

In this section we derive the linearized water wave equations,  and
prove energy estimates for  them.  We do this in three stages. First we 
prove quadratic energy estimates in $\dH_0$, which apply for the large data
problem.  Then we prove cubic energy estimates in $\dH_0$ for the small 
data problem. 
Various bilinear, multilinear and commutator estimates
which are used in this section are collected in Appendix~\ref{s:multilinear}.

\subsection{ Computing the linearization} The solutions for the
linearized water wave equation around a solution $\left( W,Q\right) $
are denoted by $(w,q)$. However, it will be more convenient to
immediately switch to diagonal variables $(w,r)$, where
\[
r := q - Rw.
\]
The linearization of $R$ is
\[
\delta R =
\dfrac{q_{\alpha}- Rw_{\alpha}}{1+\W}
= \dfrac{r_{\alpha}+ R_\alpha w}{1+\W},
\]
while the linearization of $F$ can be expressed in the form
\[
\delta F = P[ m - \bar m],
\]
where the auxiliary variable $m$ corresponds to differentiating
$F$ with respect to the holomorphic variables,
\[
m := \frac{q_\alpha - R w_\alpha}{J} + \frac{\bar R w_\alpha}{(1+\W)^2} =
 \frac{r_\alpha +R_\alpha w}{J} + \frac{\bar R w_\alpha}{(1+\W)^2}.
\]
Denoting also
\[
n := \bar R \delta R = \frac{ \bar R(r_{\alpha}+R_\alpha w)}{1+\W},
\]
the linearized water wave equations take the form
\begin{equation*}
\left\{
\begin{aligned}
&w_{t}+ F w_\alpha + (1+ \W) P[ m-\bar m] = 0, \\
&q_{t}+ F q_\alpha + Q_\alpha P[m-\bar m]  -i w +P\left[n+\bar n\right] =0.
\end{aligned}
\right.
\end{equation*}
Recalling that $b = F + \dfrac{\bar R}{1+\W}$, this becomes
\begin{equation*}
\left\{
\begin{aligned}
&(\partial_t + b \partial_\alpha) w + (1+ \W) P[ m-\bar m] =   \dfrac{\bar R w_\alpha}{1+\W} ,\\
& (\partial_t + b \partial_\alpha) q + Q_\alpha P[m-\bar m]  -i w +P\left[n+\bar n\right] = \dfrac{\bar R q_\alpha}{1+\W}.
\end{aligned}
\right.
\end{equation*}
Now, we can use the second equation in \eqref{ww2d-diff} to switch from $q$ to $r$ and obtain
\begin{equation*}
\left\{
\begin{aligned}
&(\partial_t + b \partial_\alpha) w + (1+ \W) P[ m-\bar m] =   \dfrac{\bar R w_\alpha}{1+\W},\\
& (\partial_t + b \partial_\alpha) r  -i \frac{1+a}{1+\W} w +P\left[n+\bar n\right] = \dfrac{\bar R (r_\alpha+ R_\alpha w)}{1+\W}.
\end{aligned}
\right.
\end{equation*}
Terms like $\bar P m$, $\bar Pn$ are lower order since the differentiated holomorphic variables have to be lower frequency. The
same applies to their conjugates. Moving those terms to the right
and taking advantage of algebraic cancellations we are left with
\begin{equation}\label{lin(wr)0}
\left\{
\begin{aligned}
& (\partial_t + b \partial_\alpha) w  +  \frac{1}{1+\bar \W} r_\alpha
+  \frac{R_{\alpha} }{1+\bar \W} w  = \mathcal{G}(w,r),
 \\
&(\partial_t + b \partial_\alpha)  r  - i  \frac{1+a}{1+\W} w  = \mathcal{K}(w,r),
\end{aligned}
\right.
\end{equation}
where
\begin{equation*}
\begin{aligned}
\mathcal{G}(w,r) = \ (1+\W) (P \bar m + \bar P  m), \quad \mathcal{K}(w,r) =  \  \bar P n - P \bar n.
\end{aligned}
\end{equation*}

We remark that while $(w,r)$ are holomorphic, it is not directly obvious
that the above evolution preserves the space of holomorphic states. To
remedy this one can also project the linearized equations onto the
space of holomorphic functions via the projection $P$.  Then we obtain
the equations
\begin{equation}\label{lin(wr)}
\left\{
\begin{aligned}
& (\partial_t + \M_b \partial_\alpha) w  + P \left[ \frac{1}{1+\bar \W} r_\alpha\right]
+  P \left[ \frac{R_{\alpha} }{1+\bar \W} w \right] = P \mathcal{G}(  w, r),
 \\
&(\partial_t + \M_b \partial_\alpha)  r  - i P\left[ \frac{1+a}{1+\W} w\right]  =
 P \mathcal{K}( w,r).
\end{aligned}
\right.
\end{equation}
Since the original set of equations \eqref{ww2d1} is fully
holomorphic, it follows that the two sets of equations,
\eqref{lin(wr)0} and \eqref{lin(wr)}, are algebraically equivalent.

In order to obtain cubic linearized energy  estimates it is also of interest to separate
the quadratic parts $ \mathcal{G}^{2}$ and $\mathcal{K}^{2}$
of $\mathcal{G}$ and $\mathcal{K}$. These are split into quadratic and higher
terms as shown below
\[
\begin{split}
   \mathcal{G} = &  \,  \mathcal{G}^{(2)}+  \mathcal{G}^{(3+)},
  \qquad   \mathcal{K} = \  \mathcal{K}^{(2)}+  \mathcal{K}^{(3+)}.
\end{split}
\]
For the quadratic parts we have
\[
\begin{split}
 P \mathcal{G}^{(2)}(w,r) = & \ -P \left[ \W \bar r_\alpha
  \right] + P\left[ R \bar w_\alpha\right], \qquad   P \mathcal{K}^{(2)}(w,r) =  - P \left[ R \bar r_\alpha\right],
\end{split}
\]
with $\bar P \mathcal{G}^{(2)}(w,r) = \overline{P \mathcal{G}^{(2)}(w,r)}$ and $\bar P \mathcal{K}^{(2)}(w,r) = - \overline{P \mathcal{K}^{(2)}(w,r)}$.
We can also rewrite the above expressions  in a commutator form
\begin{equation}\label{gk2}
\begin{aligned}
P \mathcal{G}^{(2)}(w,r)=-\left[ P,\W\right] \bar{r}_{\alpha}+ \left[ P,R\right] \bar{w}_{\alpha},\quad P \mathcal{K}^{(2)}(w,r)=- \left[P, R\right]  \bar r_\alpha.
\end{aligned}
\end{equation}

The cubic terms have the form
\[
\begin{split}
  \mathcal{G}^{(3+)}(w,r) =  \ P\bar m^{(3+)} + \bar P m^{(3+)} +\W( P \bar m + \bar P  m),\quad \mathcal{K}^{(3+)}(w,r) =  \   \bar P n^{(3+)} - P \bar n^{(3+)}.
\end{split}
\]
For the purpose of simplifying nonlinear estimates, it is convenient
to express $\mathcal G^{(3)}$ and $\mathcal K^{(3)}$ in a polynomial
fashion. This is done using the variable $Y=\dfrac{\W}{1+\W}$. Then we have
\[
\begin{split}
  \bar P m = & \ \bar P [w_\alpha (1-Y)^2 \bar R- (r_\alpha + R_\alpha
  w)(1-Y) \bar Y],
  \\
  \bar P m^{(3+)} = & \ \bar P [ r_\alpha (\bar \W + Y) \bar Y  -
  R_\alpha w (1-Y) \bar Y - w_\alpha (2Y-Y^2) \bar R],
  \\
  \bar P n^{(3+)} = & \ \bar P [- r_{\alpha} Y \bar R + R_\alpha w
  (1-Y) \bar R].
\end{split}
\]

\subsection{Quadratic estimates for large data.}  Our goal here is to
study the well-posedness of the system \eqref{lin(wr)} in $L^2 \times
\dot H^\frac12$.  We begin with a more general version of the system
\eqref{lin(wr)}, namely
\begin{equation}\label{lin(wr)inhom}
\left\{
\begin{aligned}
& (\partial_t + \M_b \partial_\alpha) w  + P \left[ \frac{1}{1+\bar \W} r_\alpha\right]
+  P \left[ \frac{R_{\alpha} }{1+\bar \W} w \right] = G,
 \\
&(\partial_t + \M_b \partial_\alpha)  r  - i P\left[ \frac{1+a}{1+\W} w\right]  = K,
\end{aligned}
\right.
\end{equation}
and  define the associated positive definite linear energy
\[
\Elind(w,r) = \int_{\R} (1+a) |w|^2 + \Im ( r \bar  r_\alpha)
d\alpha .
\]
We remark that, by Proposition~\ref{regularity for a}, $a$ is nonnegative and bounded,
therefore 
\[
\Elind(w,r) \approx_A \Ez(w,r)
\]
Our first result uses the control parameters $A$ and $B$ defined in
\eqref{A-def}, \eqref{B-def}:

\begin{proposition}\label{plin-short}
a) The linear equation \eqref{lin(wr)inhom} is well-posed in $\dH_0$,
and the following estimate holds:
\begin{equation}\label{lin-gen2}
  \frac{d}{dt} \Elind(w,r)  =  2 \Re \int_{\R} (1+a) \bar w \,  G -
i \bar r_\alpha \,  K \ d\alpha +  O_A(A B)  \Elind(w,r).
\end{equation}

b) The linearized equation \eqref{lin(wr)} is well-posed in $L^2 \times \dot H^\frac12$,
and the following estimate holds:
\begin{equation}\label{lin2}
\frac{d}{dt} \Elind(w,r)  \lesssim_A
 B  \Elind(w,r).
\end{equation}
\end{proposition}

\begin{proof}
a)  A direct computation yields
\[
\begin{split}
\frac{d}{dt} \int (1+a)|w|^2 d \alpha = & \ 2 \Re \!\! \int  (1+a) \bar w (\partial_t
+ \M_b \partial_\alpha  ) w +
a \bar w   [b,P] w_\alpha\,  d\alpha ,\\ & \ +  \int \left[ a_t+((1+a)b)_\alpha\right]
 |w|^2 \, d \alpha .
\end{split}
\]
A similar computation shows that
\[
\frac{d}{dt} \int \Im (  r \partial_\alpha \bar r)  \, d \alpha = 2 \Im  \int    (\partial_t
+ \M_b \partial_\alpha) r  \, \partial_\alpha \bar r \, d \alpha. \ \
\]

Adding the two and using the equations \eqref{lin(wr)inhom}, the quadratic
$\Re (w \bar r_\alpha)$ term cancels modulo another commutator term,
and we obtain
\begin{equation}\label{dE2}
\frac{d}{dt}  \Elind(w,r) = 2\Re \int (1+a) \bar w\,   G- i \bar r_\alpha \,  K\, d\alpha + err_1,
\end{equation}
where
\[
\begin{aligned}
  err_1 = & \int \left[ a_t+((1+a)b)_\alpha\right] |w|^2 d\alpha -
  2\Re \int (1+a)\frac{R_{\alpha}}{1+\bar{\W}}\vert w\vert ^2\,  d\alpha
  \\ & -2 \Re \int + a \bar w \, (\left[\bar Y, P\right] (
  r_\alpha+R_\alpha w) + [P,b] w_\alpha)\, d\alpha.
\end{aligned}
\]
Using the auxiliary function $M$ in \eqref{M-def},
we  rewrite it as
\[
\begin{aligned}
err_1  =  \int \left( a_t+ba_{\alpha}\right)  |w|^2 + M (1+a)\vert w\vert ^2 \, d\alpha
- 2 \Re  \int a \bar w \, (\left[\bar Y,P \right]  (r_\alpha+R_\alpha w)
 +   [P,b] w_\alpha) \, d\alpha.
\end{aligned}
\]
The error term is at least quartic. To conclude the proof of
\eqref{lin-gen2} it suffices to show that
\begin{equation}\label{err12}
 |err_1| \lesssim   AB   \Elind(w,r) .
\end{equation}

For the first term, by  Proposition~\ref{regularity for a} in the \emph{Appendix}~\ref{s:multilinear}, we have $\vert  a_t+ba_{\alpha}\vert \lesssim AB$.
For the second term we combine the pointwise bounds $\vert a\vert
\lesssim A^2$ in Lemma~\ref{regularity for a} together with $\|M\|_{L^\infty}
\lesssim AB$ in Lemma~\ref{l:M}.

For the last  term  it remains to estimate the
commutators in $L^2$. Two of them are obtained  using Lemma~\ref{l:com},
\[
\|\left[\bar Y ,P\right]  r_\alpha\|_{L^2} \lesssim \| |D|^\frac12 Y\|_{BMO}
\| r \|_{\dot H^\frac12}, \qquad    \| [P,b] w_\alpha\|_{L^2} \lesssim 
\|b_\alpha\|_{BMO} \|w\|_{L^2},
\]
and suffice due to the bounds for $b$ and $Y$
in Lemmas~\ref{l:b},\ref{l:Y}. For the remaining piece we write
$[\bar Y, P] (R_\alpha w)= [\bar P, \bar P [\bar Y R_\alpha]] w]$ and 
use \eqref{CM} to estimate
\[
\|  \bar P[\bar P [\bar Y R_\alpha] w]\|_{L^2} \lesssim 
\|w\|_{L^2} \|\bar P [\bar Y R_\alpha]\|_{BMO} \lesssim 
\|w\|_{L^2} \||D|^\frac12 Y\|_{BMO} \| |D|^\frac12 R\|_{BMO},
\]
where the bilinear bound in the second step follows after a bilinear
Littlewood-Paley decomposition from \eqref{bmo-bmo} and
\eqref{bmo>infty}.

b) To estimate the terms involving $ \mathcal{G}$ and $ \mathcal{K}$ we
separate the quadratic and cubic parts. It suffices to show that  the quadratic terms
satisfy
\begin{equation}\label{gk2-est}
\| \mathcal{G}^{(2)}(w, r)\|_{L^2} + \| \mathcal{K}^{(2)}(w, r)\|_{\dot H^\frac12}
\lesssim_A B (\|w\|_{L^2} + \|r\|_{\dot H^\frac12}),
\end{equation}
while  the cubic and higher terms satisfy
\begin{equation}\label{gk3-est}
\| \mathcal{G}^{(3+)}(w, r)\|_{L^2} + \| \mathcal{K}^{(3+)}(w, r)\|_{\dot H^\frac12}
\lesssim_A AB(\|w\|_{L^2} + \|r\|_{\dot H^\frac12}).
\end{equation}

In order to obtain the estimates claimed in
\eqref{gk2-est},\eqref{gk3-est} we use the Coifman-Meyer \cite{cm}
type commutator estimates described in the \emph{Appendix}~\ref{s:multilinear},
Lemma~\ref{l:com}. Precisely, for the first term in $
P\mathcal{G}^{(2)}(w,r)$ we use \eqref{first-com} with $s=0$, and
$\sigma = \frac12$ to write
\[
\| \left[ P,\W\right] \bar{r}_{\alpha}\|_{L^2} \lesssim \||D|^\frac12 \W\|_{BMO} \|r\|_{\dot H^\frac12}.
\]
For the second term in $ P\mathcal{G}^{(2)}(w,r)$ we use
\eqref{first-com} with $s=0$ and $\sigma = 1$ to obtain
\[
\|  \left[ P,R\right] \bar{w}_{\alpha}\|_{L^2} \lesssim \| R_\alpha\|_{BMO} \|w\|_{L^2},
\]
and for $P\mathcal{K}^{(2)}(w,r)$ we use \eqref{first-com} with $s = \frac12$,
and $\sigma = \frac12$, and conclude that
\[
\| \left[P, R\right]  \bar r_\alpha\|_{\dot H^\frac12} \lesssim \|R_\alpha\|_{BMO}  \|r\|_{\dot H^\frac12}.
\]
 The same estimate applies to the antiholomorphic parts of
$\mathcal{G}^{(2)}$ and $ \mathcal{K}^{(2)}$, and \eqref{gk2}
follows.

For the cubic and higher parts of $\mathcal{G}$ and
$\mathcal{K}$ we apply  the same type of commutator
estimates, as well as the $BMO$ bounds in Proposition~\ref{p:bmo},  as follows:
\[
\| \bar P [r_\alpha (1-Y) \bar Y \bar \W]\|_{L^2} \lesssim
\| r\|_{\dot H^\frac12} \| (1-Y) \bar Y \bar \W\|_{BMO^\frac12}
\lesssim_A \|Y\|_{L^\infty}  \| r\|_{\dot H^\frac12} ,
\]
using \eqref{bmo-alg} at the last step.
\[
\| \bar P[w(1-Y) R_\alpha \bar Y]\|_{L^2} 
\lesssim  \|w(1-Y)\|_{L^2} \| \bar P[R_\alpha \bar Y]\|_{BMO} 
\lesssim_A \|w\|_{L^2} \|R\|_{BMO^\frac12} \|Y\|_{BMO^\frac12}
\]
using \eqref{bmo-bmo} and \eqref{bmo>infty} at the last step.
\[
\| \bar P[ w_\alpha(2Y-Y^2) \bar R]\|_{L^2}
\lesssim \|w\|_{L^2} \| \partial_\alpha \bar P[(2Y-Y^2) \bar R]\|_{BMO},
\lesssim_A \|w\|_{L^2} \|Y\|_{L^\infty} \| R\|_{BMO}
\]
using \eqref{bmo-bmo}, and \eqref{bmo-infty} at the last step.
\[
\| |D|^\frac12 \bar P[r_\alpha Y \bar R]\|_{L^2}
\lesssim \| r\|_{\dot H^\frac12} \|  \partial_\alpha \bar P[ Y \bar R]\|_{L^2}
\lesssim_A \| r\|_{\dot H^\frac12} \|Y\|_{L^\infty} \| R\|_{BMO},
\]
again by  \eqref{bmo-bmo} and \eqref{bmo-infty}. Finally,
\[
\begin{split}
\| |D|^\frac12 \bar P[ w(1-Y) R_\alpha \bar R]\|_{L^2}
\lesssim & \ \| w(1-Y)\|_{L^2} \| |D|^\frac12 \bar P[R_\alpha \bar R]\|_{BMO}
\\
\lesssim_A & \  \|w\|_{L^2} \| |D|^\frac12 R\|_{BMO} \|R_\alpha\|_{BMO}
\end{split}
\]
follows using \eqref{bmo-bmo} and \eqref{bmo>infty}.

\end{proof}

\subsection{Cubic estimates for small data.}
For the small data problem it is of further interest to track the solution
on larger time scales.  For this  we add to the equations the
holomorphic quadratic parts $P \mathcal G^{(2)}$
and $P \mathcal K^{(2)}$ of $\mathcal G$
and $\mathcal K$ and consider the linear equations
\begin{equation}\label{lin(wr)inhom3}
\left\{
\begin{aligned}
  & (\partial_t + \M_b \partial_\alpha) w + P \left[ \frac{1}{1+\bar
      \W} r_\alpha\right] + P \left[ \frac{R_{\alpha} }{1+\bar \W} w
  \right] = -P \left[ \W \bar r_\alpha - R \bar w_\alpha\right] + G,
  \\
  &(\partial_t + \M_b \partial_\alpha) r - i P\left[ \frac{1+a}{1+\W}
    w\right] =- P \left[ R \bar r_\alpha\right]+ K.
\end{aligned}
\right.
\end{equation}
For this problem we add appropriate cubic terms and define the modified energy
\[
\Elint(w,r) = \int_{\R} (1+a) |w|^2 + \Im (r \bar  r_\alpha)
+ 2 \Im (\bar R w r_\alpha) -2\Re(\bar{\W} w^2)\, d\alpha.
\]

 Then we have:
\begin{proposition}\label{plin-long}
Assume that $A \ll1$. Then
\begin{equation}\label{elin3-eq}
 \Elint(w,r) = (1+O(A)) \Ez(w,r).
\end{equation}
In addition, the following properties hold:

a) The solutions to \eqref{lin(wr)inhom3}
 satisfy
\begin{equation}\label{elin3-dinhom}
\begin{split}
  \frac{d}{dt} \Elint(w,r) = & \ 2 \Re \int \left((1+a) \bar w - i\bar R r_\alpha - 2\bar \W w\right) \, G + i (\bar r - \bar R w) \, K_\alpha \, d\alpha \\ & + O_A(AB) \Elint(w,r).
\end{split}
\end{equation}

b) For solutions to the linearized equation \eqref{lin(wr)} we have:
\begin{equation}\label{elin3-diff}
\frac{d}{dt} \Elint(w,r)  \lesssim_A AB    \Elint(w,r).
\end{equation}
\end{proposition}

\begin{proof}
For \eqref{elin3-eq} we need to estimate the added cubic terms in $\Elint(w,r)$.
The second is trivially bounded, while  the first is rewritten as
\[
\Im \int w \bar P[\bar R r_\alpha] \, d\alpha.
\]
By Lemma~\ref{l:com} we have $\| P[\bar R r_\alpha]\|_{L^2} \lesssim
\||D|^\frac12 R\|_{BMO} \|r\|_{{\dot H^\frac12}}$,  hence \eqref{elin3-eq}
follows.

a) To prove the estimate \eqref{elin3-dinhom} we compute the time
derivative of the cubic component of the energy $E^3_{lin}(w,r)$,
using the projected equations for $w$ and $r$ and the unprojected
equations for $R$ and $\W$:
\begin{equation*}
\begin{aligned}
\frac{d}{dt}  \left( \Im \int   \bar{R} w r_\alpha \, d\alpha -\Re\int \bar{\W}w^2d\alpha\right) =&
 \Im \! \int \!   - i \bar \W w r_\alpha     -  \bar{R}  r_\alpha r_\alpha+ i \bar{R} w  w_\alpha
+ \bar R r_\alpha G + \bar R w K_\alpha \, d\alpha \\
 & +\Re\int \bar{R}_{\alpha}w^2+2\bar{\W} wr_{\alpha}+ 2 \bar \W w F\, d\alpha+ err_2,
 \end{aligned}
\end{equation*}
where
\begin{equation}
\label{err3}
\begin{aligned}
err_2 =  & \ \Im \int\!\! \left\lbrace   \left( i\left( \frac{\bar{\W}^2+a}{1+\bar{\W}} \right)-b\bar{R}_{\alpha} \right) wr_{\alpha}
-  \bar{R}w\partial_{\alpha}\left( \M_{b}r_{\alpha}-iP\left[ \frac{a-\W}{1+\W}w\right] + P[R\bar{r}_\alpha]\right)\right.
  \\
   &\left.   \qquad  -\bar{R}r_{\alpha} \left( \M_{b}w_{\alpha}-P\left[ \frac{\bar{\W}}{1+\bar{\W}}r_{\alpha}\right]
   +P\left[ \frac{R_{\alpha}}{1+\bar{\W}}w\right] + P[\W \bar{r}_\alpha - R\bar{w}_\alpha] \right) \right\rbrace \, d\alpha \\
&+ \Re\int \left\{ w^2\left( b \bar \W_{\alpha}+ \frac{\bar{\W}-\W}{1+\W}\bar{R}_{\alpha} - (1+\bar{\W})\bar{M} \right) \right.
\\ &  \qquad \left. + 2\bar{\W} w \left( \M_{b}w_{\alpha}-
P\left[ \frac{\bar{\W}}{1+\bar{\W}}r_{\alpha}\right] +P\left[ \frac{R_{\alpha}}{1+\bar{\W}}w \right]
+ P[\W \bar{r}_\alpha - R\bar{w}_\alpha]\right) \right\} \, d\alpha .
\end{aligned}
\end{equation}

Adding this to \eqref{lin-gen2} (but applied to solutions to \eqref{lin(wr)inhom3}) we  obtain
\begin{equation}\label{dE3}
\frac{d}{dt}  \Elint(w,r) =    \ 2 \Re \int \left((1+a) \bar w - i
    \bar R r_\alpha - 2\bar \W w\right) \, G + i (\bar r - \bar R 
  w) \, K_\alpha \, d\alpha+
 err_1+  err_3,
\end{equation}
where
\[
err_3 = 2 err_2 - 2 \Re\int a\bar{w} P\left[\W\bar{r}_\alpha - R \bar{w}_\alpha\right]\, d\alpha.
\]
Given the bound \eqref{err12} for $err_1$,
the proof of \eqref{elin3-dinhom} is concluded if we
show that
\begin{equation}\label{error3}
|err_3|  \lesssim AB \Ez(w,r) .
\end{equation}
Further, recalling the estimate \eqref{gk2-est}, which in expanded form reads
\begin{equation}\label{GK2}
 \|P\left[\W\bar{r}_\alpha - R \bar{w}_\alpha\right]\|_{L^2}+
 \|P[R \bar r_\alpha]\|_{\dot H^\frac12} \lesssim B \|(w,r)\|_{L^2 \times \dot H^\frac12} ,
\end{equation}
it suffices to estimate $err_2$,
\begin{equation}\label{error2}
|err_2|  \lesssim AB  \Ez(w,r) .
\end{equation}
For the remainder of the proof we separately estimate  several types of terms in $err_2$:
\bigskip 

{\bf A. Terms involving $b$.}  Here, we use the bounds for $b$ in
Lemma~\ref{l:b}, which give
\[
\| b_\alpha\|_{BMO} \lesssim_A B, \qquad \||D|^\frac12 b\|_{BMO} \lesssim_A A.
\]
We first collect all the terms that are
contained in the first integral in $err_2$ and include $b$,
\[
I_1 = \int  -b\bar{R}_{\alpha}  wr_{\alpha}-\bar{R}r_{\alpha}  \M_{b}w_{\alpha}-  
\bar{R}w\partial_{\alpha} (\M_{b}r_{\alpha})\, d\alpha.
\]
We claim that 
\begin{equation}\label{i1-est}
|I_1|
\lesssim (\||D|^\frac12 R\|_{BMO} \|b_\alpha\|_{BMO} + \|R_\alpha\|_{BMO}
\||D|^\frac12 b\|_{BMO}) \|w\|_{L^2} \|r\|_{\dot H^\frac12}.
\end{equation}
Integrating by parts we get $I_1 = I_2+I_3$, where
\[
I_2 =   \ \int  \bar{R}_{\alpha} w \bar P[b  r_{\alpha}] \, d\alpha, \qquad
I_3 =  \  \int -\bar{R}r_{\alpha}  \M_{b}w_{\alpha}-  \bar{R}w\partial_{\alpha} (\M_{b}r_{\alpha})\, d\alpha.
\]
The first term on the right has a commutator structure and will be estimated separately
later, see $I_5$ below.  The bound for $I_3$ is proved in the appendix, see \eqref{i1}.

 We next collect all the terms that are contained in the second integral in  $err_2$ and include $b$,
and rewrite them as
\begin{equation*}
\label{b-terms}
\begin{split}
I_4= \! \!\int w^2\partial_{\alpha}\overline{\M_{b}\W}+2\bar{\W}w\M_{b}w_{\alpha}\, d\alpha= \! \!\int
\!  -2ww_{\alpha}b\bar{\W}+2\bar{\W}w\M_{b}w_{\alpha}\, d\alpha =
\! \!\int \! -2\bar{\W}w\bar P[b w_{\alpha}] \, d\alpha.
\end{split}
\end{equation*}
The expression $\bar P[b w_{\alpha}]$ is bounded in $L^2$ using
Lemma~\ref{l:com} to obtain
\[
|I_4| \lesssim \|\W\|_{L^\infty} \|b_\alpha\|_{BMO} \|w\|_{L^2}^2.
\]

{\bf B. Quadrilinear terms bounded via both $L^2 \cdot L^2$ and $\dot H^\frac12 \cdot \dot H^{-\frac12}$ pairings. } This includes the following expressions:
\begin{equation*}
\begin{aligned}
I_5 &=  \int  \bar{R}_{\alpha} w \bar P[b  r_{\alpha}] \, d\alpha =  \int \bar P[b  r_{\alpha}] P[ \bar{R}_{\alpha} w] \, d\alpha,\\
I_6 &= \int \bar{R}r_{\alpha}P\left[ \frac{\bar{\W}}{1+\bar{\W}}r_{\alpha}\right] \, d\alpha = \int\bar P[ \bar{R}r_{\alpha}] P\left[ \bar Y r_{\alpha}\right] \, d\alpha,\\
I_7 &=  \int \bar{R}r_{\alpha}P\left[ \frac{R_{\alpha}}{1+\bar{\W}}w\right] \, d\alpha =\int \bar P [\bar{R}r_{\alpha}]  P\left[ R_{\alpha} w(1-\bar Y)\right] \, d\alpha,\\
I_8 &= \int \bar{\W} w P\left[ \frac{R_{\alpha}}{1+\bar{\W}}w\right] \, d\alpha = \int \bar P[\bar{\W} w]  P\left[ R_{\alpha}(1-\bar Y) w\right] \, d\alpha,\\
I_{9}& = \int \bar{\W} wP\left[ \frac{\bar{\W}}{1+\bar{\W}}r_{\alpha}\right] \, d\alpha = \int \bar P[\bar{\W}w] P\left[ \bar Y r_{\alpha}\right] \, d\alpha.
\end{aligned}
\end{equation*}

The strategy here is to bound the first factor in both $L^2$ and $\dot
H^\frac12$, and the second, partially in $L^2$ and partially in $\dot H^{-\frac12}$.
For the first factor we have by Lemma~\ref{l:com}:
\begin{equation*}
\begin{aligned}
&\| \bar P [b r_\alpha]\|_{L^2}+ \|\bar P[\bar R r_\alpha]\|_{L^2} \lesssim
(\||D|^\frac12 b\|_{BMO} + \||D|^\frac12 R\|_{BMO}) \|r\|_{\dot H^\frac12}
\lesssim A \|r\|_{\dot H^\frac12},\\
&\| \bar P [b r_\alpha]\|_{\dot H^\frac12}+ \|\bar P[\bar R r_\alpha]\|_{\dot H^\frac12} \lesssim
(\|b_\alpha\|_{BMO} + \| R_\alpha\|_{BMO}) \|r\|_{\dot H^\frac12}
\lesssim B \|r\|_{\dot H^\frac12},
\end{aligned}
\end{equation*}
as well as
\begin{equation*}
\begin{aligned}
&\|\bar P[\bar{\W}w]\|_{L^2}+ \||D|^\frac12 \bar P[\bar{R}w] \|_{L^2} \lesssim(
 \| {\W}\|_{BMO}+  \| |D|^\frac12 R\|_{BMO})  \|w\|_{L^2}
\lesssim A  \|w\|_{L^2},\\
&\|\bar P[\bar{\W}w]\| _{\dot H^\frac12} + \||D|^\frac12 \bar P[\bar{R}w] \| _{\dot H^\frac12}
\lesssim (\| |D|^\frac12 {\W}\|_{BMO}+ \|R_\alpha\|_{BMO}) \|w\|_{L^2}
\lesssim B  \|w\|_{L^2}.
\end{aligned}
\end{equation*}

We now consider the second factor in the above integrals. For $P[\bar{R}_{\alpha} w]$ we have
\[
\| \sum_k P[ \bar{R}_{k,\alpha} w_k]\|_{L^2} \lesssim \|R_\alpha\|_{BMO}
\|w\|_{L^2}, \quad
\| \sum_k P[ \bar{R}_{< k,\alpha} w_k]\|_{\dot H^{-\frac12}} \lesssim \||D|^\frac12 R\|_{BMO}
\|w\|_{L^2}.
\]
The same argument applies to $P\left[ R_{\alpha}(1-\bar Y) w\right]$ once we use the decomposition
\begin{equation*}
P\left[ (1-\bar{Y})R_{\alpha}w\right] = P\left[ (1-\bar{Y}) \sum_{k\in \mathbf{Z}} R_{\alpha, \geq k}w_{k}\right]  - P\left[ \sum_{k\in \mathbf{Z}} \bar{Y} _{k}R_{\alpha, <k}w_k\right] + \sum_{k\in \mathbf{Z}}(1-\bar{Y}) _{<k}R_{\alpha, <k}w_k.
\end{equation*}
The first term is easily bounded in $L^2$ by Lemma~\ref{l:com}. The
second is also in $L^2$ using \eqref{bmo-infty} for the product of the
first two factors. Finally, the third is bounded in $\dot H^{-\frac12}$ by estimating
$\| R_{\alpha,<k}\|_{L^\infty} \lesssim 2^{\frac{k}2} A$.

It remains to consider the expression
\[
P\left[ \bar Y r_{\alpha}\right] = P\left[ \sum_k \bar Y_k r_{k,\alpha}\right]
+   \sum_k \bar Y_{<k} r_{k,\alpha}.
\]
Here, the first term is estimated in $L^2$ using Lemma~\ref{l:com}, while the
second goes into $\dot H^{-\frac12}$.

{\bf C. Quadrilinear terms bounded via an  $L^2 \cdot L^2$   pairing.}
This includes the following expressions:
\begin{equation*}
\begin{aligned}
&I_9 = \int \bar{R}w\partial_{\alpha}P\left[ \frac{a-\W}{1+\W} w \right] \,d\alpha =- \int \partial_\alpha  \bar P[\bar{R}w]  P\left[ (a(1-Y) - Y) w \right]\, d\alpha,\\
&I_{10} = \int \bar R w \partial_\alpha P[R \bar r_\alpha] \, d\alpha = -  \int \partial_\alpha \bar P[\bar R w] \partial_\alpha P[R \bar r_\alpha] \, d\alpha ,\\
&I_{11} = \int \bar R r_\alpha P[\W \bar r_\alpha - R \bar w_\alpha]\, d\alpha = \int \bar P[\bar R r_\alpha] P[\W \bar r_\alpha - R \bar w_\alpha] d\alpha, \\
&I_{12} = \int \bar \W w P[\W \bar r_\alpha - R \bar w_\alpha] \, d\alpha = \int \bar P[\bar \W w] P[W \bar r_\alpha - R \bar w_\alpha] \, d\alpha .
\end{aligned}
\end{equation*}
 In all cases both factors are estimated directly in $L^2$, using Lemma~\ref{l:com}, see also \eqref{GK2}.

{\bf D. Trilinear estimates.}
This includes the terms:
\begin{equation*}
\begin{aligned}
&I_{13} = \int \frac{\bar{\W}^2+a}{1+\bar{\W}}wr_{\alpha}\, d\alpha=  \int w \bar P[\bar P f r_{\alpha}]\, d\alpha, \ \ \ \ \ \ \, \qquad f=\frac{\bar{\W}^2+a}{1+\bar{\W}},\\
&I_{14} = \int w^2\bar{P}\left[ \frac{\bar{\W}-\W}{1+\W} \bar{R}_{\alpha}\right]\, d \alpha = \int w \bar P[ \bar P g w] \, d\alpha, \quad  \ \,  g=\frac{\bar{\W}-\W}{1+\W} \bar{R}_{\alpha},\\
&I_{15} = \int w^2\bar{P}\left[ (1+\bar{\W}) M\right]\,d \alpha = \int w \bar P[ \bar P h w] \, d\alpha,\qquad h=(1+\bar{\W})M.
\end{aligned}
\end{equation*}
Using Lemma~\ref{l:com} we have
\begin{equation}\label{est-tri}
|I_{13}| \lesssim \| |D|^\frac12 \bar P f\|_{BMO} \|w\|_{L^2} \|r\|_{\dot H^\frac12},
\ 
|I_{14}| \lesssim \|  \bar P g\|_{BMO} \|w\|_{L^2}^2, \ 
|I_{15}| \lesssim \|  \bar P h\|_{BMO} \|w\|_{L^2}^2,
\end{equation}
so it suffices to show that
\[
\| |D|^\frac12  f\|_{BMO} + \|   g\|_{BMO}+ \|   h\|_{BMO}  \lesssim AB.
\]
The $f$ bound follows from the algebra property of $BMO^\frac12 \cap
L^\infty$ in \eqref{bmo-alg} in view of \eqref{a-point} and
\eqref{a-bmo+}. The $g$ bound is obtained by writing
\[
\frac{\bar{\W}-\W}{1+\W} \bar{R}_{\alpha}
= \sum_{k} P_{\leq k} \left(\frac{\bar{\W}-\W}{1+\W}\right) R_{k,\alpha}
+  \sum_{k} P_{ k} \left(\frac{\bar{\W}-\W}{1+\W}\right) R_{< k,\alpha}.
\]
For the first term we use \eqref{bmo-infty}, while for the second, \eqref{bmo>infty}.
Finally, the $h$ bound is trivial due to \eqref{M-infty}.
The proof of \eqref{elin3-dinhom} is concluded.

b) To prove the bound \eqref{elin3-diff} it suffices to  apply the estimate in
\eqref{elin3-dinhom} with 
\[
F = P \mathcal F^{3+}(w,r), \qquad G = P \mathcal G^{3+}(w,r).
\]
Given the estimate \eqref{gk3-est} for the cubic components
of $\mathcal F$ and $\mathcal G$ and  the pointwise bound \eqref{a-point}
for $a$, it remains to consider the terms
\[
  \int \bar R r_\alpha P \mathcal F^{(3+)}  \, d\alpha,
\qquad \int  \bar \W w P \mathcal F^{(3+)} \,  d\alpha,   \qquad \int   \bar R
  w) \, P\mathcal G^{(3+)} \, d\alpha.
\]
 For the first one we use  the second part of \eqref{est-tri} to get
\begin{equation}
\label{term4}
\left \vert \int  \bar{R}r_{\alpha}P  \mathcal F^{(3+)}  \, d\alpha\right \vert\lesssim \Vert |D|^{\frac{1}{2}} R\Vert _{L^{\infty}}\Vert r\Vert_{\dot{H}^{\frac{1}{2}}} \Vert  \mathcal F^{(3+)} \Vert_{L^2}\lesssim
AB \Vert (w,r)\Vert_{\dH_0}^2
\end{equation}
The second one is directly estimated as
\begin{equation}
\label{term5}
\left\vert \int \bar{R}w\partial_{\alpha}\mathcal{K}\, d\alpha\right\vert \lesssim \Vert |D|^{\frac{1}{2}}R\Vert_{L^{\infty}}\Vert w\Vert_{L^2}\Vert \mathcal{K}\Vert_{\dot{H}^{\frac{1}{2}}} \lesssim AB \Vert (w,r)\Vert_{\dH_0}^2
\end{equation}
On the last term, using the first part of \eqref{est-tri}, we get
\begin{equation}
\label{term6}
\left\vert \int \bar{R}w\partial_{\alpha}\mathcal{K}\, d\alpha\right\vert \lesssim \Vert |D|^{\frac{1}{2}}R\Vert_{L^{\infty}}\Vert w\Vert_{L^2}\Vert \mathcal{K}\Vert_{\dot{H}^{\frac{1}{2}}} \lesssim AB \Vert (w,r)\Vert_{\dH_0}^2
\end{equation}
The proof of the proposition is concluded.

\end{proof}

\section{ Higher order energy estimates}
\label{s:ee}

The main goal of this section is to establish two energy bounds for
$(\W,R)$ and their higher derivatives. The first one is a quadratic
bound which applies for all solutions. The second one is a cubic bound
which only applies for small solutions. The  large data result is as follows:

\begin{proposition}\label{t:en=large}
 For any $n \geq 1$ there exists an energy functional $\End$ with  the
following properties:
(i) Norm equivalence:
\begin{equation*}
\End(\W,R) \approx_A   \Ez(\partial^{n-1} \W, \partial^{n-1} R),
\end{equation*}

(ii) Quadratic energy estimates for solutions to \eqref{ww2d-diff}:
\begin{equation*}
\frac{d}{dt} \End(\W,R)  \lesssim_A B \End(\W,R).
\end{equation*}
\end{proposition}

The small data result is as follows:

\begin{proposition}\label{t:en=small}
  For any $n \geq 1$ there exists an energy functional $\Ent$ which
  has the following properties as long as $A \ll 1$:

(i) Norm equivalence:
\begin{equation*}
\Ent (\W,R)= (1+ O(A)) \Ez (\partial^{n-1} \W, \partial^{n-1} R),
\end{equation*}

(ii) Cubic energy estimates:
\begin{equation*}
\frac{d}{dt} \Ent (\W,R)  \lesssim_A AB \Ent (\W,R).
\end{equation*}
\end{proposition}

We remark that the case $n=1$ corresponds to bounds for
$(\W,R)$. But these solve the linearized system \eqref{lin(wr)},
so the above results are consequences of Proposition~\ref{plin-short}
and Proposition~\ref{plin-long}. In the sequel we consider separately the cases
$n = 2$ and $n \geq 3$.

\subsection{The case \texorpdfstring{$n=2$}.}
We  use the system \eqref{WR-diff} for
$(\W_\alpha, \R:=R_{\alpha}(1+\W))$, which for convenience we recall here:
\[
\left\{
\begin{aligned}
&  \W_{ \alpha t} + b \W_{ \alpha \alpha}  + \frac{\R_\alpha}{1+\bar \W} + \frac{R_\alpha}{1+\bar \W} \W_\alpha  = \R \bar Y_\alpha - \frac{\bar R_\alpha}{1+\W} \W_\alpha
+ 2M   \W_\alpha  +  (1+\W)  M_\alpha ,
\\
&  \R_{t} + b\R_{\alpha} -   i\frac{(1+a)\W_\alpha}{1+\W}  =-2\left(\frac{\bar R_\alpha}{1+\W}+ \frac{R_\alpha}{1+\bar \W}\right)  \R  + 2 M \R + ( R_\alpha \bar R_\alpha -i a_\alpha).
\end{aligned}
\right.
\]
Here we have isolated on the left the leading part of the linearized
equation.  We want to interpret the terms on the right as mostly
perturbative, but also pay attention to the quadratic part.  Thus, for
bookkeeping purposes, we introduce two types of error terms, denoted
$\errw$ and $\errr$, which correspond to the two equations. The
bounds for these errors are in terms of the control variables $A,B$, as well
as the $L^2$ type norm
\begin{equation*}
\norm_2 = \| (\W_\alpha,R_\alpha)\|_{L^2 \times \dot H^\frac12}.
\end{equation*}

By $\errw$ we denote terms $G$, which  satisfy the estimates
\begin{equation*}
\|P G\|_{L^2} \lesssim_A AB \norm_2,
\end{equation*}
and
\begin{equation*}
\text{either} \quad \|\bar P G\|_{L^2} \lesssim_A B \norm_2 \quad \text{or} \quad
 \|\bar P G\|_{\dot H^{-\frac12}} \lesssim_A A \norm_2.
\end{equation*}
By $\errr$ we denote terms $K$, which are at least cubic and which satisfy the estimates
\begin{equation*}
\|P K\|_{\dot H^\frac12} \lesssim_A AB  \norm_2, \qquad \| P K\|_{L^2} \lesssim_A A^2  \norm_2,
\end{equation*}
and
\begin{equation*}
 \|\bar P K\|_{L^2}  \lesssim_A A  \norm_2.
\end{equation*}

The use of the more relaxed quadratic control on the antiholomorphic terms, as opposed to the cubic control on the holomorphic terms, is motivated by
the fact that the equations will eventually get projected on the
holomorphic space, so the antiholomorphic components will have less of
an impact. A key property of the space of errors is contained in the
following
\begin{lemma}\label{l:err}
Let $\Phi$ be a function which satisfies
\begin{equation}\label{Phi-est}
\| \Phi\|_{L^\infty} \lesssim A, \qquad \||D|^\frac12 \Phi\|_{BMO} \lesssim B.
\end{equation}
Then, we have the multiplicative bounds
\begin{equation}
 \Phi \cdot  \errw = \errw, \qquad \Phi \cdot \errr = \errr,
\end{equation}
\begin{equation}
 \Phi \cdot  P \errw = A\, \errw, \qquad \Phi \cdot P \errr = A\, \errr.
\end{equation}
\end{lemma}
The proof of the lemma, based on Lemma~\ref{l:com}, is relatively straightforward
 and is left for the reader.
We will apply this lemma for $\Phi$ which are arbitrary smooth
functions of $\W$ and $\bar \W$. Then the estimates \eqref{Phi-est}
are consequences of our Moser estimates in \eqref{bmo-moser}.

We now expand some of the terms in the above system. For this we
will use the following bounds for $M$, see \eqref{M-infty} and \eqref{M-L2}:
\begin{equation}\label{M-bd}
\| M\|_{L^\infty} \lesssim AB, \qquad \|M\|_{\dot H^\frac12} \lesssim A \norm_2.
\end{equation}
First we note that
\begin{equation}\label{Mwr}
M\W_\alpha = \errw, \qquad M \R = \errr.
\end{equation}
The first is straightforward in view of pointwise bound for $M$.  For
the second, by Lemma~\ref{l:err} we can replace $M\R$ by $M R_\alpha$.
After a Littlewood-Paley decomposition, the $\dot H^\frac12$ estimate
for $M R_\alpha$ is a consequence of the pointwise bound in
\eqref{M-bd} for low-high and balanced interactions, and of the $\dot
H^\frac12$ bound in \eqref{M-bd} combined with Lemma~\ref{l:com}
for the high-low interactions.

It remains to estimate $M R_\alpha$ in $L^2$. If the frequency of $M$ is larger than or equal to the frequency of $R_\alpha$,
then we can use the $\dot H^\frac12$ bound for $M$. We are left with
\[
\sum_k R_{k,\alpha} M_{<k} = \sum_k R_{k,\alpha} M(R_{<k}, Y_{<k}) +
\sum_k \sum_{j \geq  k} R_{k,\alpha} P_{<k} M(R_j,Y_j).
\]
For the first sum we use
\[
\| M(R_{<k}, Y_{<k})\|_{L^\infty} \lesssim 2^{\frac{k}2} A^2.
\]
For the second we bound
\[
\| \sum_k \sum_{j \geq  k} R_{k,\alpha} P_{<k} M(R_j,Y_j) \|^2_{L^2}
\lesssim \sum_{j \geq k} 2^{k-j} \||D|^\frac12 R\|_{L^\infty}^4 \|Y_{j,\alpha}\|^2 \lesssim A^2 \norm_2.
\]

Next we consider $(1+\W) M_\alpha$, for which we claim that
\begin{equation}\label{Malpha}
\begin{split}
& M_\alpha =
 R_\alpha \bar Y_\alpha - \bar R_\alpha Y_\alpha +
P[ R \bar \W_{\alpha \alpha}- \bar R_{\alpha \alpha} \W] + \errw,
\\ &  P[ R \bar \W_{\alpha \alpha}- \bar R_{\alpha \alpha} \W]  = A^{-1} \errw.
\end{split}
\end{equation}
By Lemma~\ref{l:err}, this shows that
\[
(1+\W) M_\alpha =  R  \bar Y_\alpha - \frac{\bar R_\alpha}{1+\W} W_\alpha +
P[ R \bar \W_{\alpha \alpha}- \bar R_{\alpha \alpha} \W] + \errw .
\]
To prove \eqref{Malpha} we write
\[
\begin{split}
M_\alpha = & \ R_\alpha \bar Y_\alpha - \bar R_\alpha Y_\alpha
+ P[ R \bar Y_{\alpha \alpha}- \bar R_{\alpha \alpha} Y]  + \bar P (g_1+2 g_2)
\\
= & \  R_\alpha \bar Y_\alpha - \bar R_\alpha Y_\alpha +
P[ R \bar \W_{\alpha \alpha}- \bar R_{\alpha \alpha} \W] - P f +\bar P (g_1+2 g_2),
\end{split}
\]
where
\[
f = R (\bar \W \bar Y)_{\alpha \alpha}- \bar R_{\alpha \alpha} ( \W  Y),
\quad g_1 =
\bar R Y_{\alpha\alpha}- R_{\alpha\alpha} \bar Y, \quad   g_2 = \bar R_\alpha
Y_\alpha- R_\alpha \bar Y_\alpha.
\]
For $f$ and $g_1$ we have $L^2$ bounds
\[
\| Pf \|_{L^2} \lesssim AB \norm_2, \qquad \|\bar Pg_1 \|_{L^2} \lesssim B \norm_2,
\]
which follow from  commutator type bounds 
\begin{equation}\label{com2w}
\| P[R \bar \Phi_{\alpha \alpha}]\|_{L^2} \lesssim \|R_\alpha\|_{BMO} \|\Phi_\alpha\|_{L^2},
\quad
\| P[\bar R_{\alpha \alpha}  \Phi] \|_{L^2}
\lesssim \| R_\alpha\|_{\dot H^\frac12} \||D|^\frac12 \Phi\|_{BMO},
\end{equation}
derived from Lemma~\ref{l:com}. For the first term in $g_2$ we have a  similar
$L^2$ bound, but for the second we split
\[
\bar P[R_\alpha \bar Y_\alpha] = \bar P[\sum_k R_{k,\alpha} \bar Y_{k,\alpha}]
+  \sum_k R_{<k,\alpha} \bar Y_{k,\alpha}.
\]
The first sum is bounded in $L^2$ using Lemma~\ref{l:com}, but for the second
we only get a $\dot H^{-\frac12}$ bound,
\[
\|\sum_k R_{<k,\alpha} \bar Y_{k,\alpha} \|_{\dot H^{-\frac12}} \lesssim A \norm_2.
\]

Finally, we also claim that
\[
i a_\alpha = R_\alpha \bar R_\alpha + P[R \bar R_{\alpha \alpha}] + \errr,
\qquad P[R \bar R_{\alpha \alpha}] = A^{-1} \errr,
\]
which is again a consequence of commutator type estimates for holomorphic $V$:
\begin{equation}
\label{com2r}
\| P[R \bar V_{\alpha}]\|_{\dot H^\frac12} \lesssim \|R_\alpha\|_{BMO}
\|V\|_{\dot H^\frac12}, \qquad \| P[R \bar V_{\alpha}]\|_{L^2} \lesssim \||D|^\frac12 R\|_{L^{\infty}} \|V\|_{\dot H^\frac12}.
\end{equation}

Taking into account all of the above expansions,
it follows that our system can be rewritten in the form 
\[
\left\{
\begin{aligned}
& (\partial_{t}+b\partial_{\alpha}) \W_{ \alpha } + \frac{\R_\alpha}{1+\bar \W} + \frac{R_\alpha \W_\alpha }{1+\bar \W}  = 2 \R \bar Y_\alpha - \frac{2 \bar R_\alpha\W_\alpha}{1+\W} 
 + P[ R \bar \W_{\alpha \alpha}- \bar R_{\alpha \alpha}\W]+ \errw,
\\
&   (\partial_{t}+b\partial_{\alpha}) \R -   i\frac{(1+a)\W_\alpha}{1+\W}  =-2\left(\frac{\bar R_\alpha}{1+\W}+ \frac{R_\alpha}{1+\bar \W}\right)  \R  - P[\bar R_{\alpha \alpha} R]  + \errr.
\end{aligned}
\right.
\]
One might wish to compare this system with the linearized system which
was studied before. However, the terms on the right cannot be all bounded
in $L^2 \times \dot H^\frac12$, even after applying the projection operator $P$.
Precisely,   the terms on the right which
cannot be bounded directly in $L^2 \times \dot H^\frac12$ are $
\displaystyle - 2\frac{\bar R_\alpha}{1+\W} \W_\alpha$, respectively $\displaystyle -2
\left(\frac{\bar R_\alpha}{1+\W}+ \frac{R_\alpha}{1+\bar \W}\right)
\R$.

But these terms can be eliminated by conjugation with respect to
a real exponential weight $e^{2\phi}$,  where $\phi = - 2\Re \log(1+\W)$.  Then
\[
\phi_\alpha = - 2 \Re \frac{\W_\alpha}{1+\W}, \qquad
(\partial_t + b \partial_\alpha) \phi =  2 \Re \frac{R_\alpha}{1+\bar \W} - 2 M.
 \]
We denote the weighted variables by
\[
w = e^{2\phi} \W_\alpha, \qquad r = e^{2\phi} \R.
\]
Using \eqref{Mwr} and Lemma~\ref{l:err} it follows that
 $Mw = \errw$, $Mr = \errr$. Then  we get the equations
\[
\left\{
\begin{aligned}
&  w_{ t} + b w_{\alpha}  + \frac{r_\alpha}{1+\bar \W}  + \frac{R_\alpha}{1+\bar \W} w =
   P[ R \bar \W_{\alpha \alpha}- \bar R_{\alpha \alpha} \W]+  \errw,
\\
&  r_{t} + br_{\alpha}  - i \frac{(1+a) w}{1+\W} =  -  P[\bar R_{\alpha \alpha} R]  + \errr.
\end{aligned}
\right.
\]
We are not yet in a position to use our bounds for the linearized
equation since $w$ and $r$ are not exactly holomorphic. We project
onto the holomorphic space to write a system for the variables $(Pw,Pr)$.
At this point one may legitimately be concerned that restricting to the holomorphic part
might remove a good portion of our variables. However, this is not the case:
\begin{lemma}\label{en:n=2}
 The energy of $(Pw,Pr)$ above is equivalent to the energy of $(\W_\alpha,R_\alpha)$
\begin{equation}
\| (Pw,Pr)\|_{L^2 \times \dot H^\frac12} \approx_A\| (w,r)\|_{L^2 \times \dot H^\frac12}
\approx_A \|(\W_\alpha,R_\alpha)\|_{L^2 \times \dot H^\frac12}= \norm_2.
\end{equation}
\end{lemma}
\begin{proof}
 The estimate for $w$ is easy. We trivially have $\|w\|_{L^2} \lesssim_A \|\W_\alpha\|_{L^2}$,
while for the converse we write
\[
\| \W_\alpha\|_{L^2}^2 \lesssim \int \bar \W_\alpha e^\phi \W_\alpha\  d\alpha
=  \int \bar \W_\alpha P w \  d\alpha \lesssim \| \W_\alpha\|_{L^2} \|P w\|_{L^2}.
\]
To obtain the  estimate for $r$ we write
\[
|D|^\frac12 P (e^\phi (1+\W) R_\alpha) = e^\phi(1+\W) |D|^\frac12 R_\alpha
+ [P|D|^\frac12,e^\phi(1+\W) ] R_\alpha.
\]
We bound all terms in $L^2$. The one on the left is $\|r\|_{\dot
  H^\frac12}$, while the first one on the right is $\approx \|
R_\alpha \|_{\dot H^\frac12}$. It remains to bound the commutator on the right,
for which we have,  with $\Phi =  e^\phi(1+\W)$,
\[
 \|[P|D|^\frac12,\Phi] R_\alpha\|_{L^2} \lesssim \| |D| \Phi\|_{L^2} \||D|^{\frac12} R\|_{BMO} \lesssim_A \|\W_\alpha\|_{L^2}.
\]
\end{proof}

Now, we write the system for $(Pw,Pr)$:
\[
\left\{
\begin{aligned}
& P w_{ t} + \M_b P w_{\alpha}  + P\left[\frac{P r_\alpha}{1+\bar \W}\right]  +
P\left[\frac{R_\alpha}{1+\bar \W} Pw \right] =   P[ R \bar \W_{\alpha \alpha}- \bar R_{\alpha \alpha} \W]+ G_2 + \errw,
\\
&  Pr_{t} + \M_bP r_{\alpha}  - i P \left[ \frac{(1+a) Pw}{1+\W}\right] =
 -  P[\bar R_{\alpha \alpha} R]+ K_2 + \errr,
\end{aligned}
\right.
\]
where
\[
\begin{split}
G_2 = - P [b \bar P w_\alpha] -   P\left[\frac{R_\alpha}{1+\bar \W} \bar Pw \right], \qquad K_2 =  \  P [b \bar P r_\alpha] + i P \left[ \frac{(1+a) \bar Pw}{1+\W}\right].
\end{split}
\]
We claim that $G_2 = \errw$ and $K_2 = \errr$. As in Lemma~\ref{en:n=2}
we have
\[
\|\bar P
w\|_{L^2} + \|\bar P r\|_{\dot H^\frac12} \lesssim_A A \norm_2.
\]
Then, using the commutator bounds in Lemma~\ref{l:com},
we estimate $G_2$ by
\[
\|G_2\|_{L^2} \lesssim_A (\|b_\alpha\|_{BMO}+ \|R_\alpha\|_{BMO}) \|\bar P
w\|_{L^2} \lesssim_A AB \norm_2.
\]
Similarly, we bound $K_2$ in $\dot H^\frac12$ by
\[
\|K_2\|_{\dot H^\frac12} \lesssim_A \|b_\alpha\|_{BMO} \|\bar P r\|_{\dot H^\frac12}
+ \left \| |D|^\frac12 \left( \frac{a-\W}{1+\W}\right) \right \|_{BMO}  \|\bar P w\|_{L^2} \lesssim_A AB \norm_2,
\]
and in $L^2$ by
\[
\|K_2\|_{L^2} \lesssim \||D|^\frac12 b\|_{BMO} \|\bar P r\|_{\dot H^\frac12}
+ \left \| \frac{a-\W}{1+\W}\right\|_{L^\infty}  \|\bar P w \|_{L^2} \lesssim_A A^2 \norm_2.
\]

Finally, in view of the bilinear estimates \eqref{com2w}, \eqref{com2r},
we can replace $ P[ R \bar \W_{\alpha \alpha}- \bar R_{\alpha \alpha} \W]$ and
$P[\bar R_{\alpha \alpha} R]$ by $ P[ R \bar P \bar w_{ \alpha}- \W \bar P \bar r_{ \alpha}]$, respectively
$P[R \bar P \bar r_{\alpha} ]$ modulo acceptable error terms.

Taking into account the discussion above, we obtain a system for $(Pw,Pr)$
which is very much like the linearized system in the previous section:
\[
\left\{
\begin{aligned}
& P w_{ t} + \M_b P w_{\alpha}  + P\left[\frac{P r_\alpha}{1+\bar \W}\right]  +
P\left[\frac{R_\alpha}{1+\bar \W} Pw \right] =   P[ R \bar P \bar w_{ \alpha}- \W \bar P \bar r_{ \alpha}]
  + \errw,
\\
&  Pr_{t} + \M_bP r_{\alpha}  - i P \left[ \frac{(1+a) Pw}{1+\W}\right] =
 -P[R \bar P \bar r_{\alpha} ] + \errr.
\end{aligned}
\right.
\]

The results of Proposition~\ref{t:en=large} and Proposition~\ref{t:en=small}
follow from the  energy estimates for the linearized equation,
namely part (a) of Propositions~\ref{plin-short}~\ref{plin-long}; further, if $n=2$ then  we can take
\[
\End(\W,R) = \Elind(Pw,Pr), \qquad \Ent(\W,R) = \Elint(Pw,Pr).
\]

\subsection{The case \texorpdfstring{$n \geq 3$}\ , large data}

We follow the same strategy as in the case $n=2$ and
 derive the equations for $(\W^{(n-1)},  R^{(n-1)})$.  We start again with the equations
\eqref{ww2d-diff} and differentiate $n-1$ times. Compared with the case $n=2$,
we obtain many more terms. To separate them into leading order and lower order,
we call lower order terms any terms which do not involve $\W^{(n-1)}$,
$ R^{(n-1)} $ or derivatives thereof.   In the computation
below we take care to separate all the leading order
 terms, as well as all the  quadratic terms which are lower order.
Toward that end we define again the notion of \emph{error term}. Unlike in the case
$n=2$, here we also include lower order quadratic terms into the error.
As before, we describe the error bounds in terms of the parameters $A$, $B$
and
\begin{equation}
\norm_n = \| (\W^{(n-1)},R^{(n-1)})\|_{L^2 \times \dot H^\frac12}.
\end{equation}

The acceptable errors in the $\W^{(n-1)}$ equation are denoted by
$\errw$ and are of two types, $\errw^{[2]}$ and
$\errw^{[3]}$. $\errw^{[2]}$ consists of holomorphic quadratic lower
order terms of the form
\[
P[\W^{(j)} \R^{(n-j)}], \quad P[\bar \W^{(j)} \R^{(n-j)}], \quad  P[\W^{(j)} \bar \R^{(n-j)}], \qquad
2 \leq j \leq n-2.
\]
By interpolation and H\"older's inequality,  terms $G$ in $\errw^{[2]}$ satisfy the bound
\begin{equation*}
\| G\|_{L^2} \lesssim B \norm_n.
\end{equation*}

By $\errw^{[3]}$ we denote terms $G$ which  satisfy the estimates
\begin{equation*}
\|P G\|_{L^2} \lesssim_A AB \norm_n,
\qquad
\end{equation*}
and
\begin{equation*}
\text{either} \quad \|\bar P G\|_{L^2} \lesssim_A B \norm_n \quad \text{or} \quad
 \|\bar P G\|_{\dot H^{-\frac12}} \lesssim_A A \norm_n.
\end{equation*}

The acceptable errors in the $R^{(n-1)}$ equation are denoted by
$\errr$ and are also of two types, $\errr^{[2]}$ and
$\errr^{[3]}$. $\errr^{[2]}$ consists of holomorphic quadratic lower
order terms of the form
\[
P[\R^{(j)} \R^{(n-j)}], \quad P[\bar \R^{(j)} \R^{(n-j)}],   \qquad
2 \leq j \leq n-2,
\]
and
\[
P[\W^{(j)} \W^{(n-j)}], \quad P[\bar \W^{(j)} \W^{(n-j-1)}], \qquad
1 \leq j \leq n-1.
\]
By interpolation and H\"older's inequality,  terms $K$ in $\errr^{[2]}$ satisfy the bound
\begin{equation*}
\| K\|_{\dot H^\frac12} \lesssim B \norm_n, \qquad \| K\|_{L^2} \lesssim A \norm_n.
\end{equation*}
By $\errr^{[3]}$ we denote terms $K$ which  satisfy the estimates
\begin{equation*}
\|P K\|_{\dot H^\frac12} \lesssim_A AB  \norm_n, \qquad \| P G\|_{L^2} \lesssim_A A^2  \norm_n,
\qquad \|\bar P G\|_{L^2}  \lesssim_A A  \norm_n.
\end{equation*}

We begin by differentiating the terms in the $\W$ equation, where we
expand using Leibnitz rule.
For the $b$ term we have
\[
\begin{split}
\partial^{n-1} (b\W_\alpha) = & \ b \W^{(n-1)}_\alpha  + (n-1) b_\alpha \W^{(n-1)}
+ b^{(n-1)} \W_\alpha + err_1,
\\
= & \ \ b \W^{(n-1)}_\alpha  + (n-1) \left(\frac{R_\alpha}{1+\bar \W}
+\frac{\bar R_\alpha}{1+\W} \right) \W^{(n-1)}
+ 2  \W_\alpha \Re R^{(n-1)}  + err_2.
\end{split}
\]
Here  $err_1$ only contains lower order  terms, so by interpolation and H\"older's
inequality we get\footnote{Here we remark that all terms in the $\W^{(n-1)}$ equation
have the same scaling; thus, whenever all the Sobolev exponents are within
the lower order range, we are guaranteed to get the correct $L^2$ estimate
after interpolation and H\"older's inequality. The same applies to all the terms in the $R^{(n-1)}$
equation. }
$err_1 = \errw$. The difference $err_2 - err_1$ is cubic,
\[
err_2 = err_1 + (n-1) M \W^{(n-1)} + \W_\alpha ( P[R^{(n-1)} \bar Y]+ \bar P[\bar R^{(n-1)}  Y]).
\]
Using the $L^\infty$ bound for $M$ in \eqref{M-bd}, Sobolev embeddings and
interpolation it is easily seen that $err_2 = \errw$.

A similar analysis leads to
\[
\begin{split}
\partial^{n-1} \frac{(1+\W) R_\alpha}{1+\bar \W} = & \  \frac{[(1+\W) R^{(n-1)}]_\alpha}{1+\bar \W}
 + \frac{R_\alpha}{1+\bar \W} \W^{(n-1)} - R_\alpha \bar \W^{(n-1)}
\\ & \ +
R^{(n-1)}((n-2)\W_\alpha - (n-1)\bar \W_\alpha) + \errw.
\end{split}
\]
In the $M$ term we also bound lower order terms by H\"older's inequality
 and interpolation to obtain
\[
\begin{split}
 \partial^{n-1} [(1+\W)M] = & \ \errw +R^{(n-1)} \bar \W_\alpha - \frac{\bar R_\alpha}{1+\W} \W^{(n-1)} \\ &+  P\left[
 R \bar \W^{(n-1)}_\alpha - \bar R^{(n-1)}_\alpha \W
+ (n-1)(R_\alpha \bar \W^{(n-1)} - \bar R^{(n-1)} \W_\alpha)\right] \\
& +
\bar P [- R^{(n-1)} \bar \W_\alpha + \frac{\bar R_\alpha}{1+\W} \W^{(n-1)} +\bar R^{(n-1)}  \W_\alpha - \frac{R_\alpha(1+\W)}{(1+\bar \W)^2} \bar \W^{(n-1)} 
\\ & \ \ \ \ \ \ \ \
+ \bar R  \W^{(n-1)}_\alpha -  R^{(n-1)}_\alpha \bar \W
+ (n-1)(\bar R_\alpha  \W^{(n-1)} -  R^{(n-1)} \bar \W_\alpha)].
\end{split}
\]
Estimating also the quadratic $\bar P$ terms,
the above relation takes the simpler form
\[
\begin{split}
\partial^{n-1} [(1+\W)M] = &\ R^{(n-1)} \bar \W_\alpha - \frac{\bar R_\alpha}{1+\W} \W^{(n-1)}
+  P\left[R \bar \W^{(n-1)}_\alpha - \bar R^{(n-1)}_\alpha \W \right] \\ &
+ (n-1)(R_\alpha \bar \W^{(n-1)} - \bar R^{(n-1)} \W_\alpha)+
 \errw .
\end{split}
\]

Now we turn our attention to the $R$ equation. We begin with
\[
\begin{split}
\partial^{n-1} (bR_\alpha) = & \ b R^{(n-1)}_\alpha  + (n-1) b_\alpha R^{(n-1)}
+ b^{(n-1)} R_\alpha + err_3
\\
= & \ \ b R^{(n-1)}_\alpha  +  (n-1)\left(\frac{R_\alpha}{1+\bar \W}
+\frac{\bar R_\alpha}{1+\W} \right) R^{(n-1)}
+ \frac{R_\alpha}{1+\bar \W}  R^{(n-1)} + \frac{R_\alpha}{1+\W} \bar R^{(n-1)}  \\ &  + err_4,
\end{split}
\]
where we trivially have $err_3 = \errr$ as it contains only
lower order terms, both quadratic and higher order. In addition, the difference
is cubic, and is given by
\[
err_4 - err_3 = (n-1) M R^{(n-1)}  + R_\alpha\left(b^{(n-1)} - \frac{R^{(n-1)}}{1+\bar \W}
- \frac{\bar R^{(n-1)}}{1+ \W} \right).
\]
We claim that this is also $\errr$. The $L^2$ bound follows trivially
by interpolation and H\"older's inequality. 

The $\dot H^\frac12$ bound
is also easy to obtain for the second term, where the unfavorable $R^{(n-1)}$
factors only appear with a convenient frequency balance as $R_\alpha(
\bar P[ R^{(n-1)}\bar Y] - P [R^{(n-1)}Y])$. Consider now the $\dot
H^\frac12$ bound for the first term.  Since $M= O_{L^\infty}(AB)$, the
nontrivial case is when $M$ has the high frequency, where we need to
estimate
\[
\| |D|^\frac12 \sum_k  M_k     R^{(n-1)}_{<k}\|_{L^2} \lesssim \||D|^{n-\frac32} M\|_{L^2}
\| DR\|_{BMO} \lesssim AB  \norm_n.
\]
Here, we have used the bound \eqref{M-L2}.

For the remaining term in the $R$ equation we write
\[
\begin{split}
\partial^{n-1} \frac{\W - a}{1+\W} = & \ \frac{(1+ a)\W^{(n-1)}}{(1+\W)^2} +
 \frac{a^{(n-1)}}{1+\W} + err_7 \\
=  & \ \frac{(1+ a)\W^{(n-1)}}{(1+\W)^2} + \frac{i}{1+ \W}\left( P[R\bar R_\alpha^{(n-1)}+(n-1) R_\alpha
\bar R^{(n-1)} + \bar R_\alpha R^{(n-1)} ] \right. \\ & \left.- \bar P[\bar R R_\alpha^{(n-1)}+(n-1) \bar R_\alpha
 R^{(n-1)} + R_\alpha \bar R^{(n-1)}] \right))
 + err_8.
\end{split}
\]
Here $err_7$ contains lower order quadratic terms in $\W$ (without $a$)
as well as cubic terms which can be easily estimated, so $err_7= \errr$.
The difference $err_8 - err_7$ only contains lower order terms so it also
can be placed in $\errr$. Just as in the case of the $\W^{(n-1)}$ equation,
quadratic $\bar P$ terms can also be placed in the error.
Then the above relation  becomes
\[
\begin{split}
\partial^{n-1} \frac{\W - a}{1+\W}
=  \frac{(1+ a)\W^{(n-1)}}{(1+\W)^2} + i \left( P[R\bar R_\alpha^{(n-1)}]+(n-1) R_\alpha
\bar R^{(n-1)} + \frac{\bar R_\alpha R^{(n-1)}}{1+\W}\right)
 + \errr.
\end{split}
\]

Combining the above computations we obtain the differentiated system
\[
\left\{
\begin{aligned}
 &  \W^{(n-1)}_{ t} + b \W^{(n-1)}_{ \alpha}  + \frac{((1+\W) R^{(n-1)})_\alpha}{1+\bar \W}
 + \frac{R_\alpha}{1+\W} \W^{(n-1)}  =  G,
\\ &
 R^{(n-1)}_t + bR^{(n-1)}_\alpha  -  i\left(\frac{(1+ a)\W^{(n-1)}}{(1+\W)^2}\right) =
K,
\end{aligned}
\right.
\]
where
\[
\begin{split}
G = & -  n \frac{\bar R_\alpha}{1+\W}  \W^{(n-1)}
-  (n-1) \frac{ R_\alpha}{1+\bar \W}  \W^{(n-1)}
 + P[ R \bar \W^{(n-1)}_\alpha - \W \bar R^{(n-1)}_\alpha]\\ &
+ R^{(n-1)}(n \bar \W_\alpha - (n-1)\W_\alpha)
+  n(R_\alpha \bar \W^{(n-1)} - \W_\alpha\bar R^{(n-1)})  + \errw,
\\
K = & - n\left( \frac{ R_\alpha}{1+\bar \W}+ \frac{\bar R_\alpha}{1+\W}\right)  R^{(n-1)}
 -  \left(P[R \bar R^{(n-1)}_\alpha] - nR_\alpha  \bar R^{(n-1)}\right)+\errr.
\end{split}
\]
After the usual substitution  $\R = (1+\W)R^{(n-1)}$, we get
\[
\left\{
\begin{aligned}
 & \W^{(n-1)}_{ t} + b \W^{(n-1)}_{ \alpha} + \frac{ \R_\alpha}{1+\bar \W}
 + \frac{R_\alpha}{1+\W} \W^{(n-1)}  =  G,
\\
& \R_t + b\R_\alpha -  i\left(\frac{(1+ a)\W^{(n-1)}}{1+\W}\right) = K_1,
\end{aligned}
\right.
\]
where
\[
K_1=
- (n+1)  \frac{ R_\alpha\R}{1+\bar \W} - n \frac{\bar R_\alpha\R }{1+\W}  - P[R \bar \R_\alpha] -
nR_\alpha  \bar \R + \errr.
\]
The more delicate terms here are the ones on the right where the
leading order terms appear unconjugated.  We would like to eliminate
those with an exponential factor as in the $n=2$ case, but their
coefficients on the right are not properly matched.  To remedy that
we take the additional step of the holomorphic substitution
\[
\tR = \R - R_\alpha \W^{(n-2)} +(2n-1) \W_\alpha R^{(n-2)}.
\]
With the exception of exactly three terms, the contribution
of the added quadratic correction is cubic and lower order, so  we
obtain
\[
\left\{
\begin{aligned}
 & \W^{(n-1)}_{ t} + b \W^{(n-1)}_{ \alpha} + \frac{ \tR_\alpha}{1+\bar \W}
 + \frac{R_\alpha}{1+\W} \W^{(n-1)}  =  -  n \left(\frac{\bar R_\alpha}{1+\W}
+   \frac{ R_\alpha}{1+\bar \W}  \right)\W^{(n-1)}
\\ & \ \ \ \  + P[ R \bar \W^{(n-1)}_\alpha - \W \bar R^{(n-1)}_\alpha] + n\tR(\bar \W_\alpha + \W_\alpha)
+  n (R_\alpha  \bar \W^{(n-1)} -  \W_\alpha\bar R^{(n-1)}) + \errw,
\\
& \tR_t + b\tR_\alpha -  i\frac{(1+ a)\W^{(n-1)}}{1+\W} =
- n  \left(\! \frac{ R_\alpha}{1+\bar \W} + \frac{\bar R_\alpha}{1+\W} \! \right) \tR-
P[R \bar \tR_\alpha] -  nR_\alpha  \bar \tR + \errr.
\end{aligned}
\right.
\]
Now we can multiply by $e^{n\phi}$ where, as before, $\phi = -2 \Re \log(1+\W)$, in order
to eliminate the unbounded terms on the right.  We get
an equation for $(w := e^{\phi} \W^{(n-1)}, r :=  e^{\phi} \tR)$:
\[
\left\{
\begin{aligned}
 & w_{ t} + b w_{ \alpha} + \frac{ r_\alpha}{1+\bar \W}
 + \frac{R_\alpha}{1+\W} w  =    P[ R  \bar \W^{(n-1)}_\alpha - \W  \bar R^{(n-1)}_\alpha]
+  n (R_\alpha  \bar w -  \W_\alpha\bar r)  + \errw,
\\
& r_t + br_\alpha -  i\left(\frac{(1+ a) w}{1+\W}\right) =
-  P[ R  \bar R^{(n-1)}_\alpha]  -  nR_\alpha  \bar r + \errr.
\end{aligned}
\right.
\]
As $(w,r)$ are no longer holomorphic, we project and work with the projected variables.
After some additional commutator estimates which are identical to those in the
$n=2$ case we obtain
\begin{equation}\label{energy(n,3)}
\left\{
\begin{aligned}
 & Pw_{ t} + \M_b P w_{ \alpha} + P \left[\frac{ P r_\alpha}{1+\bar \W}\right]
 + P\left[\frac{R_\alpha Pw}{1+\W} \right]  =   P[ R \bar P \bar w_\alpha - \W \bar P \bar r_\alpha]
\\ &\hspace{3in}\, \ \ \ \ +  n P [R_\alpha  \bar P \bar w -  \W_\alpha\bar P\bar r]
 + \err(L^2),
\\
& Pr_t + \M_bPr_\alpha -  iP\left[\frac{(1+ a) P w}{1+\W}\right] =
-  P[ R \bar P \bar r_\alpha]  -  nP [R_\alpha  \bar P\bar r] + \err(\dot H^\frac12).
\end{aligned}
\right.
\end{equation}

Compared to the linearized equation in the previous section, here we have two additional
terms that need to be estimated. We have

\begin{lemma}
\label{l:erori}
a) The energy of $(Pw,Pr)$ above is equivalent to that of $(\W^{(n-1)},R^{(n-1)})$,
\begin{equation}
\| (Pw,Pr)\|_{L^2 \times \dot H^\frac12} \approx_A\| (w,r)\|_{L^2 \times \dot H^\frac12} \approx_A
 \norm_n,
\end{equation}
b) The additional error terms above are bounded,
\begin{equation}\label{bierror-n}
\| ( P [R_\alpha  \bar P \bar w -  \W_\alpha\bar P\bar r]  , P [R_\alpha  \bar P\bar r])\|_{L^2 \times \dot H^\frac12} \lesssim_A B
\norm_n,
\end{equation}
\end{lemma}
\begin{proof}
a) For $w$ we  argue as in the proof of Lemma~\ref{en:n=2} to get
\[
\| P w\|_{L^2} \approx_A \|w\|_{L^2} \approx_A \| \W^{(n-1)}\|_{L^2}.
\]
For $r$ we need, again as in the proof of Lemma~\ref{en:n=2},
with $\Phi =  e^\phi(1+\W)$, to bound in $L^2$ the difference
\[
\begin{split}
K = & \ |D|^\frac12 P r - \Phi |D|^\frac12 R^{(n-1)},
\\ = & \
[ |D|^\frac12 P, \Phi]R^{(n-1)} + |D|^\frac12 P[ e^\phi( - R_\alpha
W^{(n-2)} +(2n-1) \W_\alpha R^{(n-2)})],
\end{split}
\]
as well as the similar difference  but with all $P$ omitted.
It suffices to prove the estimate
\[
\|K\|_{L^2} \lesssim_A \|\W^{(n-1)}\|_{L^2} +  \|\W^{(n-1)}\|_{L^2}^{\frac{1}{n-1}}
\|R^{(n-1)}\|_{\dot H^\frac12}^{\frac{n-2}{n-1}},
\]
which follows by standard multiplicative and commutator estimates.

b) For the first term $P[R_\alpha \bar P \bar w]$ we directly use the
Coifman-Meyer type estimates in Lemma~\ref{l:com}. For the second we
bound $\W_\alpha$ in $L^{4n-6}$ and $r$ in $L^{\frac{2n-3}{n-2}}$ by
H\"older's inequality and interpolation.  For the third we have to
bound $\| |D|^\frac12 P [R_\alpha \bar P\bar r]\|_{L^2}$.  For the
balanced frequency interactions, by Coifman-Meyer it suffices to bound
$R_\alpha$ in $BMO$ and $r$ in $\dot H^\frac12$. For the high-low
interactions, on the other hand, the half-derivative goes to
$R_\alpha$, and we need to bound $|D|^\frac12 R_\alpha$, and $r$ in
$L^{4n-6}$, respectively $L^{\frac{2n-3}{n-2}}$.

\end{proof}

Given the above Lemma~\ref{l:erori}, the $n \geq 3$ case of the result in Proposition~\ref{t:en=large} is a
direct consequence of our quadratic estimates for the linearized
equation in Proposition~\ref{plin-short}(a).

The small data cubic energy estimates in Proposition~\ref{t:en=small} are
proved in the next section.  The key is to produce a modified cubic
energy, whose leading part is given by
\[
\Enthigh  ( w,r) =
\int (1+a)  |w|^2 + \Im (\bar r r_\alpha)
 + 2n \Im (R_{\alpha} \bar w  \bar r) +
 2(  \Im[\bar{R} w r_\alpha] - \Re[\bar {W}_\alpha w^2])
\, d\alpha.
\]
We claim that the evolution of this energy is governed by the following

\begin{lemma}\label{l:hf-cubic}
Let $(w,r)$ be defined as above. Then

a) Assuming that $A \ll 1$,  we have
\begin{equation}\label{en3-equiv}
 \Enthigh ( Pw,P r) \approx \Ez  ( Pw,P r)
\approx  \norm_n,
\end{equation}

b) The solutions $(Pw,Pr)$ of \eqref{energy(n,3)} satisfy
\begin{equation}\label{en3-evolve}
\begin{split}
\frac{d}{dt} \Enthigh( Pw,P r) =  & \
2 \int  \Re (\bar w \cdot \err(L^2)^{\left[ 2\right] } ) - \Im(\bar r_\alpha \cdot \err(\dot H^\frac12)^{\left[ 2\right] })
\, d\alpha \\
& \ \ \  +  O_A(AB \norm_n).
\end{split}
\end{equation}
Further, the same relation holds if $(\bar w, \bar r)$ on the right are replaced
by $(\bar \W^{(n-1)},\bar R^{(n-1)})$.

\end{lemma}

\begin{proof}
  a) Given the bounds already proved in Proposition~\ref{plin-long} for
  the linearized equation, it suffices to estimate the additional term,
\[
\left| \int R_\alpha \bar w \bar r \ d\alpha \right| \lesssim A  \norm_n.
\]
For this we use interpolation to bound $R_\alpha$, $w$ and $r$ in $L^{4n-6}$, $L^2$,
respectively  $L^{\frac{2n-3}{n-2}}$ in terms of $A$ and $\norm_n$.

b) Here, we begin with the cubic linearized energy, $\Elint$. According to
the bound \eqref{elin3-dinhom} in Proposition~\ref{plin-long}, we have
\[
\begin{split}
\frac{d}{dt} \Elint( Pw,P r) = & \
 \int  2\Re \left( (n P [R_\alpha  \bar P \bar w -  \W_\alpha\bar P\bar r] + P\err(L^2))
\cdot ( \bar w - \bar P[\bar R r_\alpha] - \bar P[\bar W_\alpha w])  \right)
\\ & \ \  - 2 \Im\left( (-  nP [R_\alpha  \bar P\bar r] + P \err(\dot H^\frac12)) \cdot (\bar r_\alpha
+ \bar P[\bar R w]_\alpha)\right)
\, d\alpha \\
& \ \ +  O_A\left( AB \|(Pw,Pr)\|_{L^2 \times \dot H^\frac12}^2\right) .
\end{split}
\]
By the Coifman-Meyer type estimates in Lemma~\ref{l:com} the following bounds hold:
\begin{equation}
\|  \bar P[\bar R r_\alpha]\|_{L^2} + \|   \bar P[\bar W_\alpha w]\|_{L^2}
+\|  \bar P[\bar R w]\|_{\dot H^\frac12}
   \lesssim A \|(w,r)\|_{L^2 \times \dot H^\frac12}.
\end{equation}
Combining this with \eqref{bierror-n} and with the bounds for the error terms
we get
\[
\begin{split}
\frac{d}{dt} \Elint( Pw,P r) \leq & \
 \int  2\Re \left( (n P [R_\alpha  \bar P \bar w -  \W_\alpha\bar P\bar r] + P\err(L^2)^{\left[ 2\right] })
\cdot \bar w \right)
\\ & \ \ - 2\Im\left( (-  nP [R_\alpha  \bar P\bar r] + P \err(\dot H^\frac12)^{\left[ 2\right] })
\cdot  \bar r_\alpha \right) \, d\alpha \\
& \ \ +  O_A\left( AB \|(Pw,Pr)\|_{L^2 \times \dot H^\frac12}^2\right) ,
\end{split}
\]
where the output from all error terms which are cubic and higher error terms
is all included in the last RHS term.

It remains to consider the contribution of the extra term
in $\Enthigh$ and show that
\begin{equation}\label{extra-diff}
\begin{split}
\frac{d}{dt} \int  \Im (R_{\alpha} \bar P \bar w  \bar P \bar r) \, d\alpha = & \
 \int  \Re \left((R_\alpha  \bar P \bar w -  \W_\alpha\bar P\bar r) \bar P \bar w \right)
  + \Im\left(   R_\alpha  \bar P\bar r
  \bar P\bar r_\alpha \right) \, d\alpha
\\ & \ \ \,+  O_A\left( AB \|(Pw,Pr)\|_{L^2 \times \dot H^\frac12}^2\right) .
\end{split}
\end{equation}
Denote by $G_n$, respectively $K_n$ the two right hand sides in
\eqref{energy(n,3)}. By the definition of error terms and by
\eqref{bierror-n} they satisfy the bounds
\[
\|(G_n,K_n)\|_{L^2 \times \dot H^{\frac12}} \lesssim_A B  \norm_n, \qquad \|K_n\|_{L^2} \lesssim_A A  \norm_n.
\]
Then their contribution in the above time derivative is estimated
\[
\left| \int  \Im (R_{\alpha} \bar P \bar G_n  \bar P \bar r + R_{\alpha} \bar P \bar w  \bar P \bar K_n)\ d\alpha\right| = \left| \int  \Im (R_{\alpha} \bar P \bar F_n  \bar P \bar r + P[R_{\alpha} \bar P \bar w]  \bar P \bar K_n)\ d\alpha\right| \lesssim_A AB \norm_2,
 \]
by using H\"older's inequality for the first term and the Coifman-Meyer commutator
estimate in Lemma~\ref{l:com} for the second.

The contributions of the $b$ terms  are
collected together in the imaginary part of the expression
\[
\begin{split}
I = & \  \int  \partial_\alpha (b R_{\alpha}) \bar P \bar w  \bar P \bar r
+ R_\alpha  \bar P(b\bar P \bar w_\alpha)  \bar P \bar r  +  R_\alpha
 \bar P \bar w \bar P(b\bar P \bar r_\alpha )\, d\alpha \\  = & \
 \int
 R_\alpha   \left([b,P](\bar P \bar w_\alpha)  \bar P \bar r  +
 \bar P \bar w [b,P] (\bar P \bar r_\alpha )\right)\, d\alpha.
\end{split}
\]
Since $\|b_\alpha\|_{BMO} \lesssim B$, we can bound using Lemma~\ref{l:com}, and then use H\"older's inequality for all terms.

Next, we consider the remaining contribution of the time
derivative of $R_\alpha$, for which we use the equation \eqref{ww2d-diff}.
This is
\[
\Im \int  \bar P \bar w  \bar P \bar r \partial_\alpha \left( \frac{\W-a}{1+\W}\right) \, d\alpha
=  \Re \int \bar P \bar w  \bar P \bar r \W_\alpha \, d\alpha -
\Re \int \bar P \bar w  \bar P \bar r  \partial_\alpha \left( \frac{\W^2+ a}{1+\W}\right)  \, d\alpha.
\]
The first term  on the right yields the second term on the
right of \eqref{extra-diff}, while the rest of the terms are directly
bounded using H\"older's inequality.

It remains to consider the contribution of the remaining left hand side terms in
\eqref{energy(n,3)}. The expression   $\dfrac{ P r_\alpha}{1+\bar \W}$ in the $r$ equation
yields the third term  on the right of \eqref{extra-diff},
plus the quartic term
\[
\int  \Im R_{\alpha} \bar P (\bar P \bar r_\alpha Y)   \bar P \bar r \, d\alpha =
\int  \Im( R_{\alpha} [\bar P,Y] (\bar P \bar r_\alpha)  \bar P \bar r
+   R_{\alpha} Y  \bar P \bar r_\alpha   \bar P \bar r)
\, d\alpha .
\]
In the first term we apply a commutator estimate and then H\"older's inequality, and
in the second we use H\"older inequality directly.

The contribution of $ P\left[\dfrac{R_\alpha}{1+\W} Pw\right]$ is
purely a H\"older term. Finally, the contribution of
$P\left[\dfrac{1+ a }{1+\W}P w\right]$ yields the first
term on the right of \eqref{extra-diff}, plus a H\"older quartic term.

\end{proof}

\subsection{Normal form  energy estimates: 
\texorpdfstring{$n \geq 3$}\ \ , small data}
\label{s:ee3+}

In this section, we construct an $n$-th order energy with cubic
estimates. One ingredient for this is the high frequency cubic energy
$\Enthigh$ in Lemma~\ref{l:hf-cubic}.  However, this does not
suffice, as the right hand side of the energy relation
\eqref{en3-evolve} still contains lower order cubic terms.
Here we use  normal forms in order to add a lower order correction
to $\Enthigh$, which removes the above mentioned cubic terms.
We recall that the normal form variables $(\tilde{W}, \tilde{Q})$ are given by
\begin{equation}
\left\{
\begin{aligned}
\tilde W &= W - 2 \M_{\Re W} W_\alpha,
\\
\tilde Q &= Q - 2 \M_{\Re W} R,
\end{aligned}
\right.
\label{nft}
\end{equation}
where $\M_u F = P [uF]$. 
They solve an equation where all nonlinearities are cubic and higher,
\begin{equation}
\left\{
\begin{aligned}
&\tW_t + \tQ_\alpha = \tG,
\\
&\tQ_t - i \tW =  \tK,
\end{aligned}
\right.
\label{nft2eq}
\end{equation}
see Proposition~\ref{p:normal}.

The obvious energy functional associated to the normal form equations (\ref{nft1eq}) is
\[
\Ennfz = \int \left(|\tW^{(n)}|^2
+ \Im [ \tQ^{(n)} \bar\tQ^{(n)}_{\alpha} ]\right)\, d\alpha.
\]
In view of Proposition~\ref{p:normal}, this functional satisfies an energy equation of the form
\begin{equation}
\frac{d}{dt} \Ennfz = quartic + higher,
\label{cubic_energy_estimate}
\end{equation}
but it has several defects:
\begin{enumerate}
\item It is expressed in terms
of $Q^{(n)}$ rather than the natural variable $R^{(n-1)}$,
\item It is not equivalent to the linear energy $\Elind(\W^{(n-1)},R^{(n-1)})$,
\item Its energy estimate has a loss of derivatives.
\end{enumerate}
However, the last two issues concerning $\Ennfz$ arise at the
level of quartic and higher order terms, and they are specific to the water wave problem.  This motivates our strategy,
which is modify $\Ennfz$ by quartic and higher terms to obtain a
``good'' energy $\Ent$ without spoiling the cubic energy estimate
(\ref{cubic_energy_estimate}).

We carry out this procedure in two steps: \textbf{(i)} we construct a modified normal form energy $\Ennf$ that depends
on $(\W^{(n-1)}, R^{(n-1)})$ and is equivalent to the linearized energy
$\Elind(\W^{(n-1)},R^{(n-1)})$;  this addresses the issues (1) and (2) above, but not (3);
\textbf{(ii)} we separate the leading order part $\Ennfhigh$ and modify
that to the correct high frequency expression $\Enthigh$ defined in
the previous section, which was inspired from the analysis of the linearized
equation. This modification is needed due to the quasilinear nature of
our problem.  Thus, we obtain an energy $\Ent$ with good, cubic
estimates.

The first step described above is implemented in the following proposition:

\begin{proposition}
There exists a modified normal form energy $\Ennf$ of  the form
\begin{align}
\begin{split}
\Ennf = & \ \Ennfhigh + \Ennflow,
\\
\Ennfhigh = & \ \int  (1- 4n \Re \W)   \left(|\W^{(n)}|^2
+  \Im[\tR \bar \tR_\alpha]\right) + 2n \Im[R_{\alpha} \bar \W^{(n-1)} \bar \tR]\, d\alpha,
\\ &   \quad + 2 \int \Im[\bar{R}{\W}^{(n-1)} \tR_\alpha] - \Re[\bar {\W} (\W^{(n-1)})^2] \, d\alpha,
\\
\Ennflow = & \ \Re \int \left(\sum_{j+k+l=2n-2} c_{jkl}\W^{(j)}  \W^{(k)} \bar \W^{(l)}
 + \sum_{j+k+l=2n-1}  d_{jk_1l_1} \W^{(j)}  R^{(k)} \bar R^{(l)}\right) \,  d\alpha,
\\
&\
\end{split}
\label{def_En_hl}
\end{align}
such that
\begin{equation}\label{enlow-diff}
\Ennf = \Ennfz + \text{(quartic and higher terms)},
\end{equation}
and
\begin{equation}\label{enlow-equiv}
\Ennfhigh = [1+ O(A)] \Ez (\W^{(n-1)}, R^{(n-1)}), \qquad \Ennflow =
O(A) \Ez (\W^{(n-1)}, R^{(n-1)}).
\end{equation}
Moreover, the sums in (\ref{def_En_hl}) for $\Ennflow$ contain
only indices $(j,k, l)$ with  $1\le j,k,l \le n-1$.
\end{proposition}

\begin{remark}
  The normal form transformation is expressed at the level of $(W,Q)$
  variables, and cannot be easily switched to the level of
  $(\W,R)$. For this reason, initially the computation of the normal
  form energy is done in terms of the original variables $(W,Q)$.  The
  interesting fact in the above proposition is that in the end we are able 
  express the energies in the convenient variables $(\W, R)$. 
\end{remark}

\begin{proof}
  We start from the normal form energy $\Ennfz$ and express it in
  terms of $(\W,R)$ and their derivatives.  First, consider the term
  involving $\tilde{W}^{(n)}$. Using (\ref{nft}), we get that
\[
\int |\tW^{(n)}|^2 \, d\alpha
= \int |W^{(n)}|^2 - 4 \Re \left[ \bar W^{(n)} \partial_\alpha^n(\M_{\Re W} W_\alpha)\right] + 4\left|\partial_\alpha^n(\M_{\Re W} W_\alpha)\right|^2 \, d\alpha.
\]
The higher-order derivatives $W^{(n+1)}$ cannot be removed from the last term, but it is quartic and therefore harmless.
The cubic term also contain derivatives of order $n+1$, but as we show next they integrate out; as a result, the cubic energy is equivalent
to the linear energy.

Moving the projection $P$ across the inner product, we have
\begin{equation*}
\int \bar W^{(n)} \partial_\alpha^n(\M_{\Re W} W_\alpha) \, d\alpha  =\int \bar W^{(n)} \partial_\alpha^n( P[ W_\alpha \Re W])\, d\alpha
= \int \bar W^{(n)} \partial_\alpha^n( W_\alpha \Re W)\, d\alpha,
\end{equation*}
which shows that
\[
\int |\tW^{(n)}|^2 \, d\alpha
=
\int |W^{(n)}|^2 - 4 \Re \left[ \bar W^{(n)} \partial_\alpha^n(  W_\alpha \Re W)\right] \, d\alpha + \text{quartic}.
\]
Thus, expanding derivatives, we get
\begin{align}
\begin{split}
\int |\tW^{(n)}|^2 \, d\alpha
&=\int |W^{(n)}|^2 - 4\Re [\bar W^{(n)} W^{(n)}_\alpha] {\Re W}
-4 n  |W^{(n)}|^2\Re W_\alpha
\\
&\quad
- 4\sum_{j = 2}^{n-1}\binom{n}{j}\Re \left[\bar W^{(n)} W^{(n-j+1)}\right] \Re W^{(j)}
 - 4 \Re [W_\alpha \bar W^{(n)}]\Re W^{(n)}
 \, d\alpha
\\
&\quad
+ \text{quartic}.
\end{split}
\label{twn_eq}
\end{align}
Integrating by parts in the cubic term that contain derivatives of $W$ of the order $n+1$, we get
\begin{equation*}
\int \Re [\bar W^{(n)} W^{(n)}_\alpha]{\Re W}\, d\alpha
= - \frac{1}{2} \int (\Re W_\alpha) |W^{(n)}|^2\, d\alpha.
\end{equation*}
In addition,
\[
\int \Re [W_\alpha \bar W^{(n)}] \Re W^{(n)}\, d\alpha
= \frac{1}{2}\int (\Re W_\alpha) |W^{(n)}|^2 + \Re [W_\alpha (\bar W^{(n)})^2] \, d\alpha.
\]
It follows that
\begin{align*}
\int |\tW^{(n)}|^2 \, d\alpha
&=\int \left(1 - 4n\Re W_\alpha\right) |W^{(n)}|^2
 - 4 \sum_{j = 2}^{n-1}\binom{n}{j}  \Re \left[\bar W^{(n)}W^{(n-j+1)}\right] \Re W^{(j)}
\\
&\qquad- 2\Re[{W}_\alpha (\bar W^{(n)})^2]\, d\alpha + \text{quartic}.
\end{align*}

A similar, but longer, computation for the terms involving $\tQ$ yields
\begin{align*}
\int\Im [\tQ^{(n)}\bar \tQ^{(n)}_{\alpha}] \, d\alpha &=
\Im \int \left(1 - 4n\Re W_\alpha\right) \left(Q^{(n)} - Q_\alpha W^{(n)}\right)
\left(\bar Q^{(n)} - \bar Q_\alpha \bar W^{(n)}\right)_\alpha
\\
&\quad + 2n \left({Q}_{\alpha\alpha} W^{(n)} \bar Q^{(n)}
+ {Q}_{\alpha\alpha} \bar W^{(n)} \bar Q^{(n)}\right) + 2\bar{Q}_\alpha {W}^{(n)} Q^{(n)}_\alpha \, d\alpha
\\
&\quad + 4 \int \sum_{j=3}^{n-1} \binom{n+1}{j}  \Im[\bar Q^{(n)} {Q}^{(n-j+2)} ] \Re W^{(j)}\, d\alpha
+ \text{quartic}.
\end{align*}
Up to quartic corrections, we may replace $Q^{(n)}$ by $R^{(n-1)}$ and
$Q^{(j)}$ by $R^{(j-1)}$ for $j \leq n$ in the cubic terms on the
right-hand side of this equation. Further, we have
\[
\begin{split}
 \partial_\alpha^n Q - R \partial_\alpha^n W = & \  (1+W_\alpha)R^{(n-1)} + \sum_{j = 1}^{n-2}  \binom{n-1}{j} R^{(j)} W^{(n-j)}
\\ = &\  \tR   + n(R_\alpha W^{(n-1)} -  W_{\alpha\alpha} R^{(n-2)}) +
\sum_{j = 2}^{n-3}  \binom{n-1}{j} R^{(j)} W^{(n-j)}.
\end{split}
\]
Thus, we obtain
\begin{align}
\begin{split}
\int\Im [\tQ^{(n)}\bar \tQ^{(n)}_{\alpha}] \, d\alpha = & \
\Im \int \left(1 - 4n\Re W_\alpha\right) \tR\bar \tR_\alpha
+ 2n   {R}_\alpha \bar W^{(n)} \bar \tR  +
2 \bar{R}_\alpha{W}^{(n)} R^{[n]}_\alpha
d\alpha
\\
&\ + 2 \Im \!\! \int n \bar \tR(  W^{(3)} R^{(n-2)} - R^{(2)} W^{(n-1)})
 + \bar \tR_\alpha
\sum_{j = 2}^{n-3}  \binom{n-1}{j} R^{(j)} W^{(n-j)})
d\alpha
\\
&\ + 4 \int \sum_{j=3}^{n-1} \binom{n+1}{j}  \Im[\bar R^{(n-1)} {R}^{(n-j+1)} ] \Re W^{(j)}\,
 + \text{quartic}.
\end{split}
\label{tqn_eq}
\end{align}

Adding (\ref{twn_eq}) and (\ref{tqn_eq}), we find that
\[
\Ennfz = \Ennf + \text{quartic terms},
\]
where $\Ennf$ is given by (\ref{def_En_hl}) with
\begin{align*}
\Ennflow &= - 4  \int \sum_{j = 2}^{n-1}\binom{n}{j} \Re \left[\bar W^{(n)}W^{(n-j+1)}\right] \Re W^{(j)}\, d\alpha
\\
&\quad + 4 \int \sum_{j=3}^{n-1} \binom{n+1}{j}  \Im[\bar R^{(n-1)} {R}^{n-j+1} ] \Re W^{(j)}\, d\alpha
\\
&\quad + 2 \Im \int n \bar \tR(  W^{(3)} R^{(n-2)} - R^{(2)} W^{(n-1)})+  \bar \tR_\alpha
\sum_{j = 2}^{n-3}  \binom{n-1}{j} R^{(j)} W^{(n-j)}\, d\alpha,
\end{align*}
which, after we substitute $W_{\alpha}$ by $\W$, gives us an energy of
the form stated in the proposition.

It remains to establish \eqref{enlow-diff}.  The second estimate
follows immediately from  H\"older's inequality and
interpolation. So does most of the first, except for two terms.
By the Coifman-Meyer estimate in Lemma~\ref{l:com} we have
\[
\int \bar R \W^{(n-1)} R^{(n-1)}_\alpha d\alpha  = \int \W^{(n-1)} \bar P[\bar R R^{(n-1)}_\alpha ] d\alpha
= O(A) \| \W^{(n-1)}\|_{L^2} \|R^{(n-1)}\|_{\dot H^\frac12}.
\]
On the other hand, for the integral
\[
\int \Re \W \Im [R^{(n-1)} \bar R^{(n-1)}_\alpha] d \alpha,
\]
we do a Littlewood-Paley decomposition, using the $\dot H^\frac12$ norm
of $R^{(n-1)} $ if the two $R$ frequencies are high, and interpolation and H\"older's inequality
otherwise.

\end{proof}

To get our final energy functionals $\Ent$, we replace $\Ennfhigh$ in
$\Ennf$ by its nonlinear version, $\Enthigh := \Enthigh(Pw,Pr)$.
That is, we define
\begin{equation}
\Ent =   \Ennf - \Ennfhigh + \Enthigh=  \Ennflow + \Enthigh.
\end{equation}
Note that $\Ent$ differs from $\Ennf$ only by a quartic term.

Now we proceed to prove Proposition~\ref{t:en=small}. The norm equivalence
is already known from \eqref{en3-equiv} and \eqref{enlow-equiv}, so we
still need the energy estimate. First, we write
\[
\frac{d}{dt} \Ent  = \frac{d}{dt}  \Ennf + \frac{d}{dt}  \left(\Enthigh - \Ennfhigh \right).
\]
This equation shows that there are no cubic terms on the right-hand side, since
the derivatives of  $\Ennf$ and
$\Enthigh - \Ennfhigh$ contain only terms that are quartic or higher order.

Next, we write
\[
\frac{d}{dt} \Ent = \frac{d}{dt}  \Ennflow + \frac{d}{dt}  \Enthigh.
\]
Both expressions have cubic terms, but these cancel due to the prior computation.
To make this cancellation precise, at this point we make the convention that all
multilinear expansions are in terms of $\W$ and $R$.
To make this cancellation explicit, we introduce a truncation operator $\Trunc^4$
that removes the cubic terms and retains everything
which is quartic and higher.

Hence, we obtain
\begin{equation}
\frac{d}{dt} \Ent = \Trunc^4\left(\frac{d}{dt} \Ennflow\right) +  \Trunc^4\left( \frac{d}{dt}  \Enthigh\right).
\end{equation}
It remains to prove the following estimates:
\begin{align}\label{quad-low}
&\left|\Trunc^4\left(\frac{d}{dt}  \Ennflow \right)\right| \lesssim_A AB \norm_n^2,
\\
\label{quad-hi}
&\left|\Trunc^4\left(\frac{d}{dt} \Enthigh \right)\right| \lesssim_A AB \norm_n^2.
\end{align}
The second bound follows directly from \eqref{en3-evolve}, so it
remains to prove \eqref{quad-low}.

\subsection{ Estimates for lower order terms: proof of
 (\protect\ref{quad-low})}

We have two main types of energy terms (or their complex conjugates) to consider,
\begin{align*}
I_1 &= \int  \W^{(j)}  \W^{(k)}
\bar \W^{(l)} d\alpha, \qquad j+k+l = 2n-2,\quad  1 \leq j,k,l \leq n,
\\
I_2 &= \int  \W^{(j)}  R^{(k)}
\bar R^{(l)} d\alpha, \qquad \ \  \ j+k+l = 2n-1, \quad  1 \leq j, k, l \leq n.
\end{align*}

To estimate their time derivatives it is easiest to use the
unprojected form \eqref{ww2d-diff} of the equations for $\W$ and $R$, which for
our purposes here we write in the form
\begin{equation}
\left\{
\begin{aligned}
&(\partial_{t}+ b\partial_{\alpha}) \W =  \ - b_\alpha(1+\W) - \bar R_\alpha := G,
\\
&(\partial_{t}+ b\partial_{\alpha}) R= \ i \frac{\W - a}{1+\W}:= K.
\end{aligned}
\right.
\end{equation}
Of $G$ and $K$ we will only need their quadratic parts and higher,
\[
G^{2+} = - b_\alpha \W + P(R\bar Y) + \bar P(\bar R Y), \qquad K^{2+} =
 - i \frac{\W^2 + a}{1+\W}.
\]
Then, we have
\[
\begin{split}
\Trunc^4\left(\frac{d}{dt} I_1\right) = & \ \int  \partial^{j-1} (- b \W_{\alpha}+G^{2+} )  \W^{(k)}
\bar \W^{(l)} +  \W^{(j)}   \partial^{k-1} (- b \W_{ \alpha}+G^{2+} )
\bar \W^{(l)}  \\ & +
\W^{(j)}  \W^{(k)} \partial^{l-1} (- b \bar \W_{ \alpha}+\bar G^{2+} ) \,
d\alpha
\end{split}
\]
Distributing derivatives, we separate the terms with undifferentiated $b$ as
\[
\int -b \partial_\alpha( \W^{(j)}  \W^{(k)}
\bar \W^{(l)}) \, d\alpha = \int b_\alpha \W^{(j)}  \W^{(k)}
\bar \W^{(l)}\, d\alpha,
\]
therefore all terms involving $b$ have the form
\[
 \int b^{(m)} \W^{(j)}  \W^{(k)}
\bar \W^{(l)}\, d\alpha, \qquad m+j+k+l = 2n-1, \quad 1 \leq m \leq n-1,\quad  1 \leq j,k,l \leq n-1,
\]
which we can estimate by H\"older's inequality and interpolation, using Lemma~\ref{l:b},
to get the $b$ bounds
\[
\| |D|^\frac12 b\|_{BMO} \lesssim_A A, \qquad \||D|^{n-\frac12} b\|_{L^2} \lesssim_A \norm_n.
\]

The remaining terms have the form
\[
\int \partial^{j-1} P(R \bar Y) W^{(k)} \bar W^{(l)} d\alpha.
\]
These are again estimated by H\"older's inequality and interpolation, using the bounds proved
in Lemma~\ref{l:b}, which show that
\[
\| |D|^\frac12 P(R \bar Y)\|_{BMO} \lesssim_A A^2, \qquad \||D|^{n-\frac12}P(R \bar Y) \|_{L^2}
\lesssim_A A \norm_2.
\]
.

The argument for $I_2$ is similar, using the algebra property
of $L^\infty \cap \dot H^s$, together with the $L^\infty$ and $\dot{H}^{n-1}$
bound  for $a$ in Proposition~\ref{regularity for a} in order to show that
\[
\| K^{2+}\|_{BMO} \lesssim_A A^2, \qquad \||D|^{n-1} K^{2+}\|_{L^2}
\lesssim_A A \norm_2.
\].

\section{ Local well-posedness}
\label{s:lwp}

As the water wave equations \eqref{ww2d1} are fully nonlinear, the
standard strategy to prove well-posedness would be to differentiate
the equations to turn them into a system of  quasilinear equations for
$(w,q):=(W_\alpha,Q_\alpha)$, and then apply an iteration scheme. The problem
with a direct implementation of this idea is that the quasilinear
problem is degenerate, and diagonalizing it requires using the exact
equations; thus the diagonalization would fail in an approximation
scheme.

To remedy this, we use  the form \eqref{ww2d-diff} of the equations in terms of the diagonal
variables  $(\W,R)$ directly. Projecting those on the holomorphic space we obtain
\begin{equation} \label{ww(WaR)}
\left\{ \begin{aligned}
& (\partial_t + \M_b \partial_\alpha) \W + P \left[ \frac{1+\W}
{1+\bar \W} R_\alpha\right]  =K(\W, R) ,
\\
 &(\partial_t + \M_b \partial_\alpha)  R  - i P\left[ \frac{(1+a)\W }{1+\W} \right]
= K(\W, R),
\end{aligned}\right.
\end{equation}
where $K(\W, R):=P\left[ (1+\W)M \right]$ and $K(\W, R):= P\left[ a\right]$.

We now turn to the business of solving the system \eqref{ww(WaR)}.
The state space for this will be the space $\dH_n$ endowed with the  norm
\[
\| (\W,R) \|_{\dH_n} := \sum_{k=0}^n \| \partial^k_\alpha (\W,R)\|_{ L^2 \times \dot H^\frac12},
\]
where $n \geq 1$. As a preliminary step, we will also consider
solutions in the smaller space
\[
\H_n = H^n \times H^{n+\frac12}, \quad \mbox{with }n \geq 2.
\]
We remark that, given a solution in $\dH_n$ for the above
equation, we already know how to obtain uniform energy estimates for it for $n \geq 1$.
The issue at hand is to convert those estimates into a well-posedness statement.
We also remark that our energy estimates are expressed in terms of the control norms
$A$ and $B$. These are in turn mostly controlled using the $\dH_1$ norm of 
$  (\W,R)$. The  exception is  the $L^\infty$ bound for $Y$, which, as it turns out, can be bounded
in terms of its initial data and the $\dH_1$ norm of $  (\W,R)$.

To better understand the evolution of the $\dH_1$ norm of the solution
it is convenient to use the language of frequency envelopes. We say
that a sequence $c_k \in \ell^2$ is a $\dH_1$ frequency envelope for
$(\W,R) \in \dH^1$ if (i) it is slowly varying, $c_j/c_k \leq
2^{-\delta|j-k|}$ with a small universal constant $\delta$, and (ii)
it bounds the dyadic norms of $ (\W,R) $, namely $\|P_k (\W,R)
\|_{\dH_1} \leq c_k$.

Our main result here is:
\begin{proposition}
  a) Let $n \geq 1$. Then the problem \eqref{ww2d1} is locally
  well-posed in for initial data $(\W,R)$ in $\dH_n$.

b) (lifespan) There exists $T= T(\|(\W,R)\|_{\dH_1}, \|Y\|_{L^\infty})$
so that the above solutions are well defined in $[0,T]$, with uniform bounds.

c) (frequency envelopes) Given a frequency envelope $c_k$ for the
initial data in $\dH_1$, a similar frequency envelope $C(\|(\W,R)\|_{\dH_1}, \|Y\|_{L^\infty}) c_k$   
applies for the solutions in $[0,T]$.
\end{proposition}
Thereom ~\ref{baiatul} is a consequence of the above proposition. The statement about the persistence of solutions for as long as $A, B$ remain bounded is a consequence of the energy estimates in Proposition~\ref{plin-short} and Proposition~ \ref{t:en=large}, where the constants depend only on $A$ and $B$.

We remark that the well-posedness result in part (a) carries different meanings
depending on $n$. If $n \geq 2$, then we obtain existence and uniqueness
in $C(\dH_n)$ together with continuous dependence on the data with respect to the
stronger $\H_n$ topology. On the other hand if $n =1$ then we produce rough
solutions $C(\dH_1)$  as the unique strong limit of smooth solutions, with continuous
dependence  on the data with respect to the
stronger $\H_1$ topology. The $\H_1$ continuous
dependence is a standard consequence of the strong $\H_n$ continuous dependence
on data together with the frequency envelope bounds.
However, for $n=1$, we do not establish a direct uniqueness
result.

The proof  proceeds in several steps:

\subsection{Existence of regular solutions}
Here we consider data $(\W,R)(0) \in \H_n$ with $n \geq 2$,
and prove the existence of solutions in the same space.
Our strategy is to obtain approximate solutions by solving the mollified system
\begin{equation} \label{ww(WaR)-N}
\left\{\begin{aligned}
& (\partial_t + P_{<N} \M_{b_{N}} \partial_\alpha P_{<N} ) \W + P_{<N} P \left[ \frac{1+P_{<N} \W}
{1+P_{<N} \bar \W} P_{<N} R_\alpha\right]  = P_{<N} G(P_{<N} \W,P_{<N}R),
\\
 &(\partial_t + P_{<N}\M_{b_N}  \partial_\alpha P_{<N})  R  - i P_{<N}
P\left[ \frac{(1+a_N)P_{<N}\W }{1+P_{<N}\W} \right]
=  P_{<N} K(P_{<N} \W,P_{<N} R),
\end{aligned}\right.
\end{equation}
where $P_{<N}$ is a multiplier which selects frequencies less than $N$, and
\[
b_N = b(P_{<N} \W,P_{<N}R), \qquad a_N= a(P_{<N}R).
\]

For fixed $N$ these equations form a system of ordinary differential equations in
$\H_n$, which admits a local solution.  We can consider it with a single
data, or with a one parameter family of data.  The latter will help
with the dependence of data for our original equation.

We will prove uniform estimates for this evolution in $\H_n$, $n \geq 1$,
and then obtain our solution (or one parameter family of solutions)
as a weak limit on a subsequence as $N \to \infty$.

The $(G,K)$  terms are Lipschitz, indeed $C^1$ from $\H_n$ to $\H_n$,
therefore harmless. The $\H_{n-1}$ norm of $(\W,R)$ is estimated directly by
time integration,
\begin{equation}
\label{bubu}
\frac{d}{dt}\Vert (\W, R)\Vert^2 _{\H_{n-1}}\lesssim c(\|(\W,R)\|_{\H_n}^2) \|(\W,R)\|_{\H_n}^2.
\end{equation} 
It remains to estimate the $\H_0$ norm of $\partial^n_\alpha(\W,R)$.
We differentiate the equations \eqref{ww(WaR)-N} $n$ times. This yields
\begin{equation}\label{ww(WaR)Nd}
\left\{
\begin{aligned}
& (\partial_t + P_{<N} \M_{b_{N}} \partial_\alpha  P_{<N} ) \W^{(n)}   +  P_{<N} P \left[ \frac{(1+ P_{<N} \W) }
{1+ P_{<N} \bar \W}\partial_\alpha  P_{<N} R^{(n)} \right]
= G_n,
\\
 &(\partial_t +  P_{<N} \M_{b_N}  \partial_\alpha  P_{<N})  R^{(n)}
- i  P_{<N}
P\left[ \frac{(1+a_N) P_{<N}\W^{(n)} }{(1+ P_{<N}\W)^2} \right]
=  K_n,
\end{aligned}\right.
\end{equation}
where all other terms, included in $G_n$ and $K_n$, are estimated
directly in $\H_0$ in terms of the $\H^n$ norm of $(\W,R)$.  We
observe that the fact that we work in $\H_n$ with $n \geq 2$ allows us
to use pointwise bounds for $R$, $R_\alpha$, $b$, $b_\alpha$, and
thus deal with a larger number of terms in this fashion.

To bring this to the standard form, where we can apply energy estimates previously obtained  in Section~\ref{s:linearized}, we make the substitution
\[
\R^{(n)} := R^{(n)}(1+ P_{<N} \W).
\]
Multiplying in the second equation by $(1+ P_{<N} \W)$, all of the commutator terms
are also perturbative, and we obtain the system
\[
\left\{
\begin{aligned}
& (\partial_t + P_{<N} \M_{b_{N}} \partial_\alpha  P_{<N} ) \W^{(n)} +  P_{<N} P \left[ \frac{1}
{1+ P_{<N} \bar \W}  P_{<N} \R_{\alpha}^{(n)}\right]  = \mathbf{G}_n,
\\
 &(\partial_t +  P_{<N} \M_{b_N}   \partial_\alpha  P_{<N})  \R^{(n)}  - i  P_{<N}
P\left[ \frac{(1+a_N) P_{<N}\W^{(n)} }{(1+ P_{<N}\W)} \right]
=  \mathbf{K}_n,
\end{aligned}\right.
\]
where $\mathbf{G}_n$ and $\mathbf{K}_n$ are appropriate replacements of the (perturbative) terms in \eqref{ww(WaR)Nd},  $G_n$ and $K_n$ respectively.

For this system we do energy estimates as before, with the energy functional
\[
E^{n} = \int (1+a_N) |\W^{(n)}|^2 + Im (\R^{(n)} \partial_\alpha \R^{(n)}) + |\R^{(n)}|^2\, d\alpha.
\]
We obtain
\[
\frac{dE^n}{dt} \lesssim c(\|(\W,R)\|_{\H_n}^2) \|(\W,R)\|_{\H_n}^2.
\]
We combine this with \eqref{bubu}. Since
\[ 
\Vert R\Vert_{\dot{H}^{\frac{1}{2}}}\lesssim \Vert \R\Vert_{\dot{H}^{\frac{1}{2}}}\Vert Y\Vert_{L^{\infty}}+\Vert \R \Vert_{L^2}\Vert |D|^{\frac{1}{2}}Y\Vert_{BMO},
\]
we have that
\[
\Vert(\W, R)\Vert^2_{\H_n} \lesssim c(\|(\W,R)\|_{\H_{n-1}}^2)\left( E^n +\|(\W,R)\|_{\H_{n-1}}^2\right),
\]
which  leads to a bound for our approximate system which is uniform in $N$,
\begin{equation}
\| (\W,R)(t)\|_{\H_n} \lesssim \| (\W,R)(0)\|_{\H_n}, \quad 0 \leq t \leq
T( \| (\W,R)(0)\|_{\H_n},\|Y(0)\|_{L^\infty}).
\end{equation}

Similarly, one can consider a smooth family of data $(\W_h,R_h)$ in $\H_n$
for $h \in [0,1]$. Then the solutions depend smoothly on $h$, with a  lifespan
uniformly bounded from below. We consider the $h$ derivatives $(\tw,\tr)
= \partial_h (\W_h,R_h)$. These solve the linearized equation, which
when considered in $\H_{n-1}$, can be written in the same form as
\eqref{ww(WaR)Nd}, with perturbative terms on the right. Thus, we obtain
\begin{equation}
\| (\tw,\tr)(t)\|_{\H_{n-1}} \lesssim \| (\tw,\tr)(0)\|_{\H_{n-1}}, \quad 0 \leq t \leq
T( \| (\W,R)(0)\|_{\H_n},\|Y(0)\|_{L^\infty}).
\end{equation}
In the same manner one can  obtain estimates for the second order derivatives with respect to $h$
in $\H_{n-2}$, \emph{etc}. Passing to a weak limit on a subsequence as $N \to \infty$ we obtain a family of solutions
$(\W_h,R_h)$  which is uniformly bounded in $\H_n$, with $h$ derivatives
uniformly bounded in $\H_{n-1}$, \emph{etc}.

\subsection{ Uniqueness of regular solutions}

In the previous subsection we have constructed $\H_{n}$ solutions for
$n \geq 2$. Here we prove that these solutions are unique.
For later use, we show that uniqueness holds in the larger class
of $\dH_n$ solutions for $n \geq 2$.

Suppose we have two $\dH_2$ solutions $(\W_1,R_1)$ and
$(\W_2,R_2)$ to \eqref{ww(WaR)}. Subtracting the two sets of
equations we obtain a system for the difference $(\tw,\tr)$, namely
\begin{equation} \label{ww(WaR)diff}
\left\{\begin{aligned}
& (\partial_t + \M_{b_1} \partial_\alpha) \tw + P \left[ \frac{1+\W_1}
{1+\bar \W_1} \tr_\alpha\right]  =\tilde G,
\\
 &(\partial_t + \M_{b_1} \partial_\alpha)  \tr  - i P\left[ \frac{(1+a_1)\tw }{(1+\W_1)^2} \right]
= \tilde K,
\end{aligned}\right.
\end{equation}
where
\[
\left\{\begin{aligned}
\tilde G = & G(\W_1,R_1) - G(\W_2,R_2) +\M_{b_1-b_2}   \partial_\alpha \W_2 + P \left[ \left(\frac{1+\W_1}{1+\bar \W_1} - \frac{1+\W_2}{1+\bar \W_2} \right) \partial_{\alpha}R_{2}\right],
\\
\tilde K = &  K(\W_1,R_1)-   K(\W_2,R_2)+
\M_{b_1-b_2}   \partial_\alpha R_2 +  i P\left[ \frac{(1+a_1)\tw^2 }{(1+\W_1)^2(1+\W_2)} +
 \frac{(a_1-a_2)\W_2 }{1+\W_2} \right]
 \end{aligned}\right.
\]
With implicit constants depending on the $\dH_2$ solutions $(\W_1,R_1)$ and $(\W_2,R_2)$, we have
\[
\| (\tilde G,\tilde K) \|_{\H_0} \lesssim \|(\tw,\tr)\|_{\H_0}.
\]
Then we simultaneously do energy estimates for $(\tw,\tr(1+\W_1))$ in
$\H^{\frac{1}{2}}=L^2 \times \dot H^\frac12$ and for $R$ in $L^2$, and then apply Gronwall's inequality to get $(\tw,\tr)=(0,0)$.

\subsection{\texorpdfstring{$\mathcal \dH_1 $}\ \  bounds}

The solutions produced above have a lifespan which depends on the $\H_n$
size of data. Here we prove that in effect the lifespan depends only on the
$\dH_1$ size of data, and that we have uniform bounds for as long as
the $\dH_1$ size of the solutions is controlled.

Precisely, suppose we have an   $\H_n$ solution $(\W,R)$ which satisfies the bounds
\begin{equation*}
\| (\W,R)(0)\|_{\dH_1} < \cM_0, \qquad \|Y(0)\|_{L^\infty} < \mathcal \cK_0.
\end{equation*}
Then we claim that there exists $T = T(\cM_0,\cK_0)$ so that the solution
exists in $[0,T]$ and satisfies the bounds
\begin{equation}\label{h1bd}
\| (\W,R)\|_{L^\infty(0,T; \dH_1)} < \cM(\cM_0,\cK_0),
\qquad
\|Y\|_{L^\infty([0,T]\times \R)} < \cK(\cM_0,\cK_0),
\end{equation}
as well as the $\H_n$ and $\dH_n $ bounds
\begin{equation*}
\| (\W,R)\|_{L^\infty(0,T; \H_n)} \leq  C (\cM_0,\cK_0)\| (\W,R)(0)\|_{ \H_n},
\end{equation*}
\begin{equation*}
\| (\W,R)\|_{L^\infty(0,T; \dH_n)} \leq  C (\cM_0,\cK_0)\| (\W,R)(0)\|_{ \dH_n}.
\end{equation*}

To prove this, we begin by making the bootstrap assumption
\begin{equation*}
\| (\W,R)\|_{L^\infty(0,T;\dH_1)} < 2\cM,
\qquad
\|Y\|_{L^\infty([0,T]\times \R)} < 2\cK.
\end{equation*}
We will show that for a suitable choice  $\cM(\cM_0,\cK_0)$ and $\cK(\cM_0,\cK_0)$, depending only on $\cM_0$ and $\cK_0$, we can improve 
this to \eqref{h1bd}, provided that $T < T(\cM_0,\cK_0)$.

We begin by applying the linearized energy estimates obtained in Proposition~\ref{plin-short}
to  $(\W,R)$
\begin{equation}\label{h1-dh0}
\| (\W,R)(t)\|_{\dH_0} \lesssim  e^{C  t } \| (\W,R)(0)\|_{\dH_0}
, \qquad C = C(\cM,\cK).
\end{equation}

 Applying the energy estimates proven in Proposition~ \ref{t:en=large} $(ii)$ for the  pair
 $(\W_\alpha,(1+\W)R_\alpha)$ we get
\begin{equation}\label{h1-dh1}
\| (\W_\alpha,(1+\W)R_\alpha)(t)\|_{\H_0} \lesssim  e^{C  t },
\| (\W_\alpha,(1+\W)R_\alpha)(0)\|_{\dH_0}.
\end{equation}
 To combine \eqref{h1-dh0} and  \eqref{h1-dh1} we need to invert $1+\W$.
However a brute force argument introduces a constant which depends
on both $\cK$ and $\cM$, which wreaks havoc with our bootstrap.
Instead we do a more careful argument, using
the pair of bounds
\begin{equation}
\begin{split}
&\|(1+\W)R_\alpha\|_{\dot H^\frac12} \lesssim_{\cK}  \|R_\alpha\|_{\dot H^\frac12} +
 \|  \W_\alpha\|_{L^2} \||D|^\frac12 R\|_{L^\infty},
\\
& \|R_\alpha\|_{\dot H^\frac12}   \lesssim_{\cK}  \|(1+\W)R_\alpha\|_{\dot H^\frac12}+
\|  \W_\alpha\|_{L^2} \||D|^\frac12 R\|_{L^\infty}.
\end{split}
\end{equation}
Since
\[
\| |D|^\frac12 R\|_{L^\infty}^2 \lesssim_{\cK}   \|R\|_{\dot H^\frac12}
 \|R_\alpha\|_{\dot H^\frac12},
\]
we obtain
\[
 \||D|^\frac12 R\|_{L^\infty}^2  \lesssim_{\cK}  \cM_0^2 e^{2 C t } (1+ \||D|^\frac12 R\|_{L^\infty}),
\]
so
\begin{equation*}
 \||D|^\frac12 R\|_{L^\infty}  \leq C_0  \cM_0^2 e^{2C t }, \qquad C_0 = C_0(\cK).
\end{equation*}
Then  it follows that
\begin{equation}\label{H1cont}
\| (\W,R)(t)\|_{\dH_1} \leq C_0 \cM_0^3 e^{3 C t }.
\end{equation}

Since $\cM$ appears only in the exponent where it is controlled by choosing
$t$ small, the bound \eqref{H1cont}  suffices in order to bootstrap $\cM$.
It remains to recover the bootstrap assumption on $\|Y\|_{L^\infty}$.
For this we use an estimate of the form
\[
\|Y\|_{L^\infty}^2 \lesssim \|\W_\alpha\|_{L^2} \|\W(1+ \W)^{-3}\|_{L^2}.
\]
The bound for the first factor is independent of $\cK$. For the
second we write the transport equation
\[
(\partial_t + b \partial_\alpha) \frac{\W}{(1+ \W)^{3}} =
\frac{3-2\W}{(1+\W)^3} \left ([P,W_\alpha] \frac{\bar R}{(1+\W)^2}
- P \left[ \frac{R}{1+ \bar \W}\right] _\alpha \right).
\]
We can estimate the right hand side in $L^2$ with constants depending on $\cK$.
To bound $\dfrac{\W}{(1+ \W)^{3}}$ in $L^2$ we use an estimate of the form
\begin{equation*}
\begin{aligned}
\frac{d}{dt}\Vert u\Vert_{L^2}^2=\int_{\mathbb{R}}b_{\alpha}\vert u\vert^2 + 2\Re (\partial_t+b\partial_{\alpha})u \bar{u}\, d\alpha.
\end{aligned}
\end{equation*}
For the second term on the right we use the Cauchy-Schwarz inequality
and for the first term we use a Littlewood-Paley trichotomy.  When the
frequency of $b_{\alpha}$ is strictly less than the frequencies of $u$
and $\bar{u}$ then we can move half of derivative on either of $u$ or
$\bar{u}$; otherwise Coiman-Meyer type estimates apply, and we obtain
\[
| \int _{R}b_{\alpha}\vert u\vert^2\, d\alpha |\lesssim \Vert
b_{\alpha}\Vert_{BMO}\Vert u\Vert^2_{L^2}+\Vert
|D|^{\frac{1}{2}}b\Vert_{BMO}\Vert u\Vert_{\dot{H}^{\frac{1}{2}}}\Vert
u\Vert_{L^2}.
\] 
We conclude that
\begin{equation}
\label{dany}
\begin{aligned}
\frac{d}{dt}\Vert u\Vert_{L^2}\lesssim \Vert
b_{\alpha}\Vert_{BMO}\Vert u\Vert^2_{L^2}+\Vert
|D|^{\frac{1}{2}}b\Vert_{BMO}\Vert u\Vert_{\dot{H}^{\frac{1}{2}}}\Vert
u\Vert_{L^2}+ \Vert u\Vert_{L^2}\Vert (\partial_t+b\partial_{\alpha})u\Vert_{L^2}.
\end{aligned}
\end{equation}

We apply this estimate to $\dfrac{\W}{(1+ \W)^{3}}(t)$ to obtain
\[
\|\frac{\W}{(1+ \W)^{3}}(t) \|_{L^2} \leq \|\frac{\W}{(1+ \W)^{3}}(0) \|_{L^2}
+ t C(\cK,\cM).
\]
This leads to
\[
\|Y\|_{L^\infty}^2 \lesssim \cM_0 \cK_0^3  + t C(\cK,\cM).
\]
Hence in order for our bootstrap argument to succeed we need to find $\cK,\cM$ and $T$
so that
\[
 \cM > 2 C_0(\cK) \cM_0^3 e^{C(\cK,\cM) T}, \qquad \cK^2 > 2(  \cM_0 \cK_0^3  + t C(\cK,\cM)).
\]
This is easily achieved by succesively choosing
\[
\cK^2 = 10  \cM_0 \cK_0^3, \qquad    \cM =  10 C_0(\cK) \cM_0^3 , \qquad T < C(\cK,\cM)^{-1}.
\]
Thus, the bootstrap  is  complete.

The next step is to show that we can propagate the full $\H_{n} $ norm
given control of $\dH_1$ norm of the solution $(\W, R)$. For higher
derivatives we can use Proposition~\ref{t:en=large} to obtain
\begin{equation}\label{Hncont}
\| (\W,R)(t)\|_{\dH_n} \leq C e^{C t } \| (\W,R)(t)\|_{\dH_n}, \qquad C=C(\cK,\cM).
\end{equation}

We also need to control the growth of the $L^2$ norm of $R$; for this we use equation \ref{ww2d-diff} for which we
  can easily obtain  $L^2$ bounds of the RHS. Applying \eqref{dany} we obtain 
\[
\|R(t) \|_{L^2} \leq \|R(0) \|_{L^2}
+ t C(\cK,\cM).
\]
The $\H_n$ bound shows that the solution can be continued for as
long as it stays bounded in $\dH_1$, i.e., at least until time $T(\cK_0,
\cM_0)$.

\subsection{ \texorpdfstring{$\dH_n$}\ \  solutions for 
 \texorpdfstring{$n \geq 2$}{}\ \ }

Our goal here is to obtain solutions for $\dH_n$ data. We already know that
such solutions, if they exist, are unique. The idea is to approximate
a $\dH_n$ data set $(\W,R)(0)$ with $\H_n$ data in the $\dH_n$ topology.
As the uniform $\dH_n$ bounds hold uniformly for the approximating sequence,
we would like to conclude that on a subsequence these approximate solutions
converge  weakly to the desired solution. The only difficulty with this plan
is that the $\dH_n$ convergence does not guarantee uniform pointwise
convergence for $R$. This is because the lowest Sobolev norm we control for $R$
is the $\dot H^\frac12$ norm, and that does not see constants.

To address the above difficulty, we take an approximating sequence of data $(\W_k,R_k)(0)$ which has the following two properties:

(i) $(\W_k,R_k)(0) \to (\W,R)(0)$ in $\dH_n$,

(ii) $R_k(0) \to R(0)$ uniformly on compact sets.

The second requirement effectively removes the Galilean invariance. It suffices
to ask for pointwise convergence at a single point; in view of the
known average growth rates for BMO functions, this implies the weighted
uniform convergence
\[
\| \log(2+|\alpha|)^{-1}  R_k(0) - R(0)\|_{L^\infty} \to 0.
\]

We will use the second requirement (ii) to produce weighted uniform
bounds for the $R_k$ part of the solution. Starting from the uniform bound
\[
\| (\W_k,R_k)\|_{\dH_n} \lesssim 1,
\]
we estimate uniformly most of the terms in the $R_k$ equation to obtain
\[
\| (\partial_t + 2 \Re R_k \partial_\alpha) R_k\|_{ L^\infty} \lesssim 1.
\]
This yields a uniform bound for $R_k$ along the corresponding characteristic
\[
\dot \alpha(t) = 2 \Re R_k (\alpha), \qquad \alpha(0) = 0,
\]
namely
\[
|R_k(t,\alpha(t))| \lesssim 1.
\]
This in turn shows that locally in time we have
\[
|\alpha(t)| \lesssim 1,
\]
which leads to the uniform bound
\[
|R_k(t,0)| \lesssim 1,
\]
and further to the global bound
\[
|R_k| \lesssim \log(2+|\alpha|).
\]
This in turn yields a similar bound for $\partial_t \W_k$ and
$\partial_t R_k$, and suffices in order to insure local uniform
convergence of $(\W_k,R_k)$ on a subsequence. Thus, the desired
solution $(\W,R)$ is obtained in the limit.

\subsection{Rough solutions}

Here we construct solutions for data in $\dH_1$ as unique limits of
smooth solutions. Given a $\dH_1$ initial data $(\W_0,R_0)$ as above
we regularize it to produce smooth approximate data $(\W_0^k,R_0^k) =
P_{< k} (\W_0,R_0) $.  We denote the corresponding solutions by
$(\W^k,R^k)$. By the previous analysis, these solutions  exist on a
$k$-independent time interval $[0,T]$ and  satisfy uniform $ \dH_1$
bounds. Further, they are smooth and have a smooth dependence on $k$.

Consider a frequency envelope $c_k$ for the initial data $(\W_0,R_0)$  in $\dH_1$.
Then for the regularized data we have
\[
\| (\W^k_0,R^k_0)\|_{\mathcal H_n} \lesssim c_k 2^{(n-1)k}, \qquad n \geq 2.
\]
Hence, in the time interval $[0,T]$  we also have the estimates
\begin{equation}\label{high(W,K)}
\| (\W^k,R^k)\|_{\mathcal H_n} \lesssim c_k 2^{(n-1)k}, \qquad n \geq 2.
\end{equation}
We will use these for the high frequency part of the regularized solutions.

For the low frequency part, on the other hand, we 
view $k$ as a continuous rather than a discrete parameter,
differentiate $(\W^k,R^k)$ with respect to $k$ and use the estimates
for the linearized equation. One minor difficulty is that the linearized 
equation \eqref{lin(wr)0} arises from the linearization of 
the $(W,Q)$ system in \eqref{ww2d1} rather than the differentiated
$(\W,R)$ system in \eqref{ww2d-diff}. 
Assuming that $(W^k,Q^k)$ were also defined, we formally  
denote  
\[
(w^k,r^k) = (\partial_k
W^k, \partial_k Q^k - R \partial_k W^k).
\]
These would solve the linearized equation around the $(\W^k,R^k)$ solution. 
For our analysis we want to refer only to the differentiated variables, so we 
we compute
\[
\begin{split}
\partial_\alpha w^k = & \  \partial_k W^k, \\
\partial_\alpha r^k = & \ (1+\W^k) \partial_k R^k - R_\alpha^k w^k.
\end{split}
\]
We take these formulas as the definition of $(w^k,r^k)$, and observe 
that inverting the $\partial_\alpha$ operator is straightforward 
since the above multiplications involve only holomorphic factors 
therefore the products are at frequency $2^k$ and higher.
To take advantage of the bounds in Proposition~\ref{plin-short} for
the linearized equation, we need a $\dH_0$ bound for
$(w^k(0),r^k(0))$, namely
\begin{equation}
\| (w^k(0),r^k(0)) \|_{\dH_0} \lesssim c_k 2^{-2k}.
\end{equation}
The bound for $w^k(0)$ is straightforward, but some work is required
for $r^k(0)$.  This follows via the usual Littlewood-Paley trichotomy
and Bernstein's inequality for the low frequency factor, with the
twist that, since both factors are holomorphic, no high-high to low
interactions occur.

In view of the uniform $\dH_1$ bound for $(W^k,Q^k)$,
Proposition~\ref{plin-short} shows that in $[0,T]$ we have the uniform
estimate
\begin{equation}
\| (w^k,r^k) \|_{\dH_0} \lesssim  c_k 2^{-2k}.
\end{equation}
Now, we return to $(\W^k, R^k)$ and claim the bound
\begin{equation}\label{diff(W,K)}
\| P_{\leq k} (\partial_k \W^k, \partial_k R^k) \|_{\dH_0} \lesssim  c_k 2^{-k}.
\end{equation}
Again the $\W^k$ bound is straightforward. For $\partial_k R^k$ we write
\[
\partial_k R^k = (1- Y^k)(\partial_\alpha r^k + R_\alpha^k \partial_k W^k),
\]
where again all factors are holomorphic. Then applying $P_{\leq k}$ restricts all frequencies
to $\lesssim 2^k$, and the  Littlewood-Paley trichotomy and Bernstein's
inequality again apply.

Now we integrate \eqref{diff(W,K)} over unit $k$ intervals and use it
to estimate the differences. Combining the result  with \eqref{high(W,K)} we obtain
\begin{equation}
\begin{split}
\| (\W_{k+1} - \W_k, R_{k+1}-R_k)\|_{\dH_0} \lesssim  & \ c_k 2^{-k},
\\
\| \partial_\alpha^2(\W_{k+1} - \W_k, R_{k+1}-R_k)\|_{\dH_0} \lesssim  & \ c_k 2^{k}.
\end{split}
\end{equation}

Summing up with respect to $k$ it follows that the sequence
 $(\W^k,R^k)$ converges uniformly in $\dH_1$ to a solution
$(\W,R)$, which also inherits the frequency envelope bounds from the data.

The frequency envelope bounds allow us to prove continuous dependence
on the initial data in the $\H^1$ topology. This is standard, but we
briefly outline the argument. Suppose that $(\W_j,R_j)(0) \in \dH_1$ and
$(\W_j,R_j)(0)-(\W,R)(0) \to 0 $ in $\H_1$. We consider the approximate
solutions $((\W_j^k,R_j^k)$, respectively $(\W^k,R^k)$. According to
our result for more regular solutions, we have 
\begin{equation}\label{reg-lim}
((\W_j^k,R_j^k) - (\W^k,R^k) \to 0 \qquad \text{in} \ \ \H_n.
\end{equation}
On the other hand, from the $\H_1$ data convergence we get
\[
(\W_j^k,R_j^k)(0) - (\W_j,R_j)(0) \to 0  \qquad \text{in} \ \ \H_1 \ \ \text{uniformly in} \ \ j.
\]
Then the above frequency envelope analysis,  shows that
\[
(\W_j^k,R_j^k) - (\W_j,R_j) \to 0 \qquad \text{in} \ \ \H_1 \ \ \ \text{uniformly in} \ \ j.
\]
Hence we can let $k$ go to infinity in \eqref{reg-lim}
and conclude that
\[
((\W_j,R_j) - (\W,R) \to 0 \qquad \text{in} \ \ \H_1.
\]

\section{ Enhanced cubic lifespan bounds}
\label{s:cubic}
In this section we prove Theorem~\ref{t:cubic}. Given initial data $(\W,R)$ for \eqref{ww2d-diff} 
satisfying 
\[
\| (\W,R)(0)\|_{\dH_1}  \leq \epsilon,
\]
we consider the solutions on a time interval $\left[0, T \right] $ and seek to prove the estimate
\begin{equation}
\label{miki}
\| (\W,R)(t)\|_{\dH_1}  \leq C\epsilon,\quad t\in \left[0, T \right],
\end{equation}
provided that  $T\ll e^{-2}$.  In view of our local well-posedness result this shows that the solutions can be extended up to time $T_{\epsilon}=ce^{-2}$ concluding the proof of the theorem.

In order to prove \eqref{miki} we can harmlessly make the bootstrap assumption
\begin{equation}
\label{miki2}
\| (\W,R)(t)\|_{\dH_1}  \leq 2C\epsilon,\quad t\in \left[0, T \right].
\end{equation}

From \eqref{miki2} we obtain the bounds
\begin{equation*}
A,B\lesssim C\epsilon.
\end{equation*}
Hence, by the energy estimates in Proposition~\ref{plin-long} applied to $(\W, R)$, and those in Proposition~\ref{t:en=small}, with $n=2$, applied to $(\W_{\alpha}, R_{\alpha})$ we obtain
\[
\Vert (\W, R)\Vert_{L^{\infty}(0,T;\dH_1)}\lesssim \Vert (\W, R)(0)\Vert_{\dH_1}+TAB\Vert (\W, R)\Vert_{L^{\infty}(0,T;\dH_1)}\lesssim \epsilon +TC^3\epsilon^3.
\]
Hence, the desired estimate \eqref{miki} follows provided that $T\ll (C\epsilon)^{-2}$.

\section{ Pointwise decay and long time solutions}
\label{s:decay}

In this section we prove the almost global existence result in
Theorem~\ref{t:almost}.  This is achieved via a bootstrap argument for
the energy norm $\| (W,Q)(t)\|_{\WH}$ defined in \eqref{WH} as well as
the control norms $A(t)$ and $B(t)$ in \eqref{A-def},\eqref{B-def}. 
In order to have a more robust argument  we 
will work with   a stronger norm $\| (W,R)\|_{\XX} \gtrsim A(t)+B(t)$, namely
\[
\| (W,R)\|_{\XX} =  \| W\|_{L^\infty} + \| R \|_{L^\infty}
+  \| |D|^\frac12 W_{\alpha}\|_{L^\infty} + \|R_\alpha \|_{L^\infty}
\]

 Then we will establish the
energy estimates
\begin{equation}
 \sup_{|t|  \leq T_\epsilon}\| (W,Q)(t)\|_{\WH}  \lesssim \epsilon,
\label{energy}\end{equation}
as well as the pointwise bounds
\begin{equation} \label{point}
 \| (W,R)\|_{\XX}   \lesssim \epsilon \langle t  \rangle^{-\frac12}, \qquad |t| \leq T_\epsilon,
\end{equation}
for times $T_\epsilon$ satisfying
\begin{equation}\label{which-T}
 T_\epsilon \leq e^{c \epsilon^{-2}}, \qquad c \ll 1.
\end{equation}

A continuity argument based on our local well-posedness results shows that 
 it suffices to prove that \eqref{point} and \eqref{energy}  hold for all $T_\epsilon$ 
as in \eqref{which-T}, given the bootstrap assumptions
\begin{equation}\label{energyboot}
 \sup_{|t|  \leq T_\epsilon}\| (W,Q)(t)\|_{\WH}  \leq C \epsilon,
\end{equation}
\begin{equation}\label{pointboot}
 \| (W,R)\|_{\XX}  \leq C \epsilon \langle t  \rangle^{-\frac12}, \qquad 0 \leq t \leq T_\epsilon,
\end{equation}
with a large constant $C$ (independent of $\epsilon$).

\subsection{The energy estimates in \protect(\ref{energy}).}
Here we use the bootstrap assumption \eqref{pointboot} in order to
establish \eqref{energyboot}. The only role of \eqref{energyboot} is to insure
that a solution with appropriate regularity exists up to time $T_\epsilon$.
We summarize the result in the following
\begin{proposition}
  Assume that in a time interval $[-T,T]$ we have a solution $(W,Q)$ to
  \eqref{ww2d1} which satisfies \eqref{data} and \eqref{pointboot}.
  Then we also have the energy estimate
\begin{equation}\label{e}
\|(W,Q)(t)\|_{\WH}^2  \lesssim \epsilon \langle t \rangle^{C_1 \epsilon^2} , \qquad t \in [-T,T] 
\end{equation}
for some $C_1 \gg C$.
\end{proposition}
Then the bound \eqref{energy} holds with a constant independent of $C$ for times 
as in \eqref{which-T} if we choose $c = C_1^{-1}$.
 
\begin{proof}
 The energy bound for $(W,Q)$ is a
consequence of the conserved energy \eqref{ww-energy}. The energy
bounds for $(\W,R)$ and $(w,r):= \AA\S(W,Q) $ follow by Gronwall's inequality from
the cubic energy estimates for the linearized equation in
Proposition~\ref{plin-long}; indeed, by our bootstrap assumption
\eqref{pointboot} we have $A(t),B(t) \leq C \epsilon \langle t
\rangle^{-\frac12} $, therefore,
\[
 \| (w,r)(t)\|_{\dH_0} \lesssim  e^{\int_{0}^{t} 
C^2 \epsilon^2  \langle s \rangle^{-1} ds}  \| (w,r)(0)\|_{\dH_0} \lesssim 
\epsilon e^{C^2 \epsilon^2 \log t},
\]
which suffices for $T_\epsilon$ as in \eqref{which-T}.  Finally, the
bound for $\partial^k(\W,R)$ with $1 \leq k \leq 5$ follows also by
Gr\"onwall's inequality from the cubic energy estimates in
Proposition~\ref{t:en=small}.
\end{proof}

\subsection{The pointwise estimates} 
Here we use the bootstrap assumption \eqref{energyboot} in order to establish 
\eqref{point}.  To state the main result here we introduce the notation
\begin{equation}\label{omega}
\omega(t,\alpha) = \frac{1}{\langle t \rangle^{\frac{1}{22}}} + \frac{1}{(\langle \alpha\rangle/\langle t \rangle  + \langle t\rangle /\langle \alpha\rangle)^\frac12} \lesssim 1. 
\end{equation}
Then we have:
\begin{proposition} \label{p:point}
Assume that \eqref{e} and \eqref{pointboot} hold in some interval $[-T,T]$.
Then we also have
\begin{equation}\label{strong-point}
  |W| + |R| + ||D|^\frac12 W_{\alpha}| + | R_\alpha|  \lesssim  
\epsilon \langle t\rangle^{-\frac12} \langle t \rangle^{C_1 \epsilon^2}\omega(t,\alpha)
\end{equation}
\end{proposition}
Then our pointwise bound \eqref{point} follows for times as in
\eqref{which-T}, and the proof of Theorem~\ref{t:almost} is concluded.

We remark that the result we prove here is somewhat stronger than what we
need. However, on one hand this is what follows from our analysis, and  on the other hand
this stronger result will come in handy when we prove the global result
in a follow-up paper.

The rest of this section is devoted to the proof of the above proposition. 
We note that  \eqref{which-T} plays no role in this argument.

In order to obtain pointwise bounds it is convenient to work with the
normal form variables $(\tW,\tQ)$, given by \eqref{nft1}.  Then we
prove several very simple Lemmas. The first one shows that we can
harmlessly replace $(W,R)$ by $(\tilde W, \tilde Q_\alpha)$ in the
pointwise estimates.

\begin{lemma}
Assume that \eqref{e} and \eqref{pointboot} hold in some interval $[-T,T]$.
Then
\begin{equation}\label{point-eq}
\|(W-\tW,R-\tQ_\alpha)\|_{\XX} \lesssim \langle t \rangle^{-\frac18} \|(\tW,\tQ_\alpha)\|_{\XX}
\end{equation}
\end{lemma}
\begin{proof}
It  suffices to show that
\[
\|(\tilde W-W, \tilde Q_\alpha-R)\|_{\XX} \lesssim \langle t \rangle^{-\frac18} \|(W, R)\|_{\XX} 
\]
For the $\tilde W$ bound we have $W-\tW = \M_{\Re W} W_\alpha$ so we 
use Sobolev embeddings, product Sobolev bounds and interpolation to estimate
\[
\begin{split}
\| \M_{\Re W} W_\alpha\|_{L^\infty}  + \| |D|^\frac32 ( \M_{\Re W} W_\alpha)\|_{L^\infty} 
\lesssim & \ \| \Re W W_\alpha\|_{L^4} + \| D^2  (\Re W W_\alpha)\|_{L^4} 
\\
 \lesssim & \  
\|  D^{2}   W\|_{L^3}  \| W_\alpha\|_{L^\infty} +   \|W\|_{L^\infty} (\|W_\alpha\|_{L^4} + \|D^{3}   W\|_{L^4})
\\
\lesssim & \
 (\|W\|_{L^\infty}  + \|W_\alpha\|_{L^\infty}) A(t)^\frac12 \|W\|^{\frac12}_{H^5}
\\ 
\lesssim & \  C^{\frac12}  \epsilon^2 
\langle t \rangle^{-\frac14+C_1 \epsilon^2}\|(W, R)\|_{\XX}  .
\end{split}
\]
which suffices since $\epsilon$ is small. 

For the $R$ bound we write
\[
\tQ_\alpha-R  = W_\alpha R - 2 \partial_\alpha(\M_{\Re W} R)
\]
and  a similar argument as above applies.

\end{proof}

Our second lemma translates the energy bounds to $(\tW,\tQ)$:
\begin{lemma}
Assume that \eqref{e} and \eqref{pointboot} hold in some time interval $[-T,T]$.
Then
\begin{equation}
\|( \tW,\tQ)\|_{\dH_5} + \|\S( \tW,\tQ)\|_{\dH_0+\dH_{-1}} \lesssim \epsilon \langle t \rangle^{C_1 \epsilon^2}.
\end{equation}
\end{lemma}
\begin{proof}
For $\tilde W$ we estimate the quadratic terms
\[
\begin{split}
\| \Re W W_\alpha\|_{L^2} + \|\partial^5 (\Re W W_\alpha)\|_{L^2} 
\lesssim & \ \|W\|_{L^\infty} (\| W_\alpha\|_{L^2} +  \| \partial^5 W_\alpha\|_{L^2})
+ \| W_\alpha\|_{L^\infty} \|  \partial^5 W\|_{L^2}
\\ \lesssim & \ \epsilon (\| W_\alpha\|_{L^2}+ \|\partial^6 W\|_{L^2}),
\end{split}
\]

By interpolation and Sobolev embeddings we can combine \eqref{e} and 
\eqref{pointboot} to obtain the rough bound
\[
\|W\|_{L^\infty} + \|R\|_{L^\infty} \lesssim \epsilon,
\]
which we will use to supplement \eqref{pointboot}.
For $\tilde Q$ we first bound the quadratic term in $\dot H^\frac12$,
\[
\| \Re W R\|_{\dot H^\frac12} \lesssim \|W\|_{L^\infty} \|R\|_{\dot H^\frac12}
+ \| W\|_{\dot H^\frac12} \| R\|_{L^\infty} \lesssim \epsilon (\|R\|_{\dot H^\frac12}
+ \| W\|_{\dot H^\frac12}).
\]
For higher derivatives we write
\[
\tilde Q_\alpha = R(1+W_\alpha) - 2 \partial_\alpha (\M_{\Re W} R)
\]
and apply the same method.

The goal of the reminder of the proof is to prove that
\begin{equation} \label{stwq}
\|S(\tW,\tQ)\|_{\dH_0 + \dH_{-1}} \lesssim \epsilon \|S(W,R)\|_{\dH_0}
\end{equation}
Recalling the notation $(w,r) = (SW,SQ-RSW)$,
we first write  $S\tW$  as
\[
S \tW = w + 2P[ \Re w W_\alpha  - \Re W_\alpha  w
- 2\Re W W_\alpha + 2 P \partial_\alpha [ \M_{2\Re W} w],
\]
and use the $L^2$ bound on $SW$ to estimate all but the last term in
$L^2$, and the last term in $\dot H^{-1}$.
Finally, for $S \tQ$ we have
\[
S\tilde Q = S Q - \M_{ 2\Re SW} R - \M_{ 2\Re W} S R = r + R w 
- 2P (\Re w R)  - 2 P \left[ \frac{\Re W (r_\alpha + R_\alpha w)}{1+W_\alpha}\right] 
\]
Here it suffices to estimate the contribution of $w$ in $L^2$, while 
the contribution of $r_\alpha$ is estimated by 
\[
\| r_\alpha H\|_{H^{-\frac12}} \lesssim \|r\|_{\dot H^\frac12} (\|H\|_{L^\infty} +
 \||D|^\frac12 H\|_{BMO}), \qquad H := \frac{\Re W}{1+W_\alpha},
\]
where $\||D|^\frac12 H\|_{BMO}$ is estimated using \eqref{bmo-alg}, \eqref{bmo-moser}.
The proof of \eqref{stwq} is concluded.
\end{proof}

The advantage of working with $(\tW,\tQ)$ is that they solve an
equation with a cubic nonlinear term, namely \eqref{nft1eq}, where the nonlinearities
$\tilde G$ and $\tilde K$ are given by \eqref{gk-tilde}.
They involve second order derivatives of $W$ and $Q$, which is why one cannot
simply use the above equations as the main evolution.

\begin{lemma}\label{cubic-RHS}
Assume that \eqref{energy} and \eqref{pointboot} hold in some time interval $[-T,T]$. Then
\begin{equation}
\| ( \tilde G,\tilde K)\|_{\dH_0}  \lesssim \frac{\epsilon^3}{\langle t \rangle} \langle t \rangle^{C_1 \epsilon^2}.
\label{energy-rhs}\end{equation}
\end{lemma}
\begin{proof} Given the expression above for $\tilde G$, it suffices
  to bound each factor in each term in suitable $L^p$ norms,
  interpolating between the $L^2$ norms in \eqref{energy} and the
  $L^\infty$ norms in \eqref{pointboot}. For $\tilde K$ the argument
  is similar, but we also need to use Lemma~\ref{l:multi} in the
  Appendix~\ref{s:multilinear} in order to distribute the half
  derivative.
\end{proof}

Taking into account the correspondence, established in the last three
lemmas, between the original variables $(W,Q)$ and the normal form
variables $(\tilde W,\tilde Q)$, it follows that we can restate
Proposition~\ref{p:point} in the following linear form:

\begin{proposition} \label{p:point1} Suppose $(\tW,\tQ)$ solve
  \eqref{nft1eq} and that the following bounds hold at some time
  $t$:
\[
 \| \S (\tW,\tQ)\|_{\dH_0 + \dH_{-1}} +   \| (\tW, \tQ) \|_{\dH_{5}} \lesssim 1,
\]
 \[
 \| ( \tilde G,\tilde K)\|_{\dH_0} \lesssim \langle t \rangle^{-1}.
\]
Then
\begin{equation}
|\tW|+    | |D|^\frac12 \tQ| +
| D^2 \tW| +     ||D|^\frac52 \tQ|
 \lesssim \langle t \rangle^{-\frac12} \omega(t,\alpha).
\end{equation}
\end{proposition}

Combining the scaling bound with the equation \eqref{nft1eq} we are led to a system 
of the form
\begin{equation} \label{fixed-t}
\left\{ 
\begin{aligned}
& 2 \alpha \partial_\alpha \tilde W +  t \partial_\alpha \tilde Q =  \tilde G_1:= S\tilde W - \tilde G ,
\\
& 2\alpha \partial_\alpha \tilde Q - it \tilde W   = \tilde K_1 := S \tilde Q - \tilde K.
\end{aligned}
\right.
\end{equation}
where 
\begin{equation}
\| (\tilde G_1,\tilde K_1)\|_{\dH_0 + \dH_{-1}} \lesssim 1.
\end{equation}
From here on, all our analysis is at fixed $t$.

After the substitution
\[
(w,r) = (\tilde W, |D|^\frac12 \tilde Q), \qquad (g,k) =(\tilde G_1, |D|^\frac12 \tilde K_1 + |D|^\frac12 \tilde Q),
\]
the above system is written in a more symmetric form as 
\begin{equation} \label{fixed-t-i}
\left\{ 
\begin{aligned}
& 2 \alpha \partial_\alpha w - i t |D|^\frac12  r = g,
\\
& 2\alpha \partial_\alpha r  - it |D|^\frac12 w   = k.
\end{aligned}
\right.
\end{equation}
For this it suffices to establish the following result:

\begin{lemma}
The following pointwise bounds hold for solutions to \eqref{fixed-t-i}:
\begin{equation}\label{point0}
|w| + |r| \lesssim |\alpha|^{-\frac12} ( \| (w,r)\|_{L^2} +
\| (g,k)\|_{L^2}),
\end{equation}
\begin{equation}\label{point1}
|w| + |r| \lesssim \langle t \rangle^{-\frac12}  \left(\frac{1}{\langle t \rangle^\frac14} + \frac{1}{(\langle \alpha\rangle /\langle t \rangle  + \langle t\rangle/ \langle \alpha \rangle)^\frac12}\right)  ( \| (w,r)\|_{H^2} +
\| (g,k)\|_{H^{-1}}),
\end{equation}
\begin{equation}\label{pointk}
|\partial^k w| + |\partial^k r| \lesssim \langle t \rangle^{-\frac12}  \left( \! \frac{1}{\langle t \rangle^{\frac{1}{2(4k+5)}}} + \frac{1}{(\langle \alpha\rangle /\langle t \rangle  + \langle t\rangle/ \langle \alpha \rangle)^\frac12} \! \right)
   ( \| (w,r)\|_{H^{2k+2}} +
\| (g,k)\|_{H^{-1}}).
\end{equation}
\end{lemma}
The last part is applied with $k \leq \frac32$, which justifies the exponent
$\frac{1}{22}$ in the definition \eqref{omega} of $\omega(t,\alpha)$.

\begin{proof} Without any loss in generality we assume that $|t| \geq 1$.
It is convenient to work with frequency localized versions of \eqref{fixed-t-i}, at frequency 
$- 2^\ell$, with $\ell \in \Z$.
The localized dyadic portions $(w_{\ell},r_\ell)$ solve similar equations
with frequency localized right hand sides $(g_\ell,k_\ell)$.
Further, a straightforward commutator estimate shows that
\begin{equation}\label{fixed-diad}
\sum_{\ell \in \Z} \| (g_\ell,k_\ell)\|_{H^s}^2 \lesssim \|(w,r)\|_{H^s}^2 + \|(g,k)\|_{H^s}^2.
\end{equation}

To prove \eqref{point0} we observe that the system \eqref{fixed-t-i} is elliptic
away from frequency  $2^{\ell} \approx t^2
\alpha^{-2}$ and degenerate at frequency zero.  At frequencies less
than $\alpha^{-1}$ our source for the pointwise estimate is
Bernstein's inequality. Comparing the two frequencies yields the
threshold $\alpha = t^2$, $2^{\ell} = t^{-2}$.   Thus, we distinguish the
following regions:

{\bf Case A: $2^\ell \leq t^{-2}$}.  We group all such frequencies
together. We can harmlessly discard the $i t |D|^\frac12$ term from the
equations and compute
\[
 \frac{d}{dt} |w(\alpha)|^2 = 2\Re ( \bar w  w_\alpha),
\]
and similarly for $r$.
Depending on the sign of $\alpha$ we integrate from either $+\infty$ or $-\infty$ 
and apply the Cauchy-Schwartz inequality to obtain 
\[
|w_{<t^{-2}} |^2 \lesssim |\alpha|^{-1} \|w_{<t^{-2}}\|_{L^2} \| \alpha w_{<t^{-2},\alpha}\|_{L^2}.
\]
We remark that by Bernstein's inequality we also get 
\[
| w_{<t^{-2}}| \lesssim |t|^{-1} \|w_{<t^{-2}}\|_{L^2}.
\]
It follows that 
\[
| w_{<t^{-2}}| \lesssim (t+|\alpha|)^{-1}.
\]

{\bf Case B.} $t^{-2} \lesssim 2^{\ell}$.  Here we have three regions to consider:

\begin{enumerate}
\item[B1.] The outer region $|\alpha| \gg |t| 2^{-\frac{\ell}2}$ where the problem is elliptic, with 
$\alpha \partial_\alpha$ as the dominant term.

\item[B2.] The inner region $|\alpha| \ll |t| 2^{-\frac{\ell}2}$ where the problem is elliptic, with 
$i t |D|^\frac12 $ as the dominant term.

\item[B3.] The intermediate region $|\alpha| \approx |t|2^{-\frac{\ell}2}$
  where the problem is hyperbolic.
\end{enumerate}

We consider three overlapping smooth positive cutoff functions
$\chi^{\ell}_{out}$, $\chi^{\ell}_{med}$ and $\chi^{\ell}_{in}$
associated with the three regions. In order to keep the frequency
localization we assume that all three cutoffs are localized at
frequency $\ll 2^\ell$, at the expense of having tails which decay
rapidly on the $2^{-\ell}$ scale. We remark that the three cutoffs
begin to separate exactly at $2^\ell = t^{-2}$.

For the regions B1 and B3 we use elliptic estimates, while for B2 we
use a propagation bound.

Using the frequency localized form of \eqref{fixed-t-i} we can bound 
\[
\| \chi^{\ell}_{out} \alpha \partial_\alpha (w_{\ell},r_{\ell}) \|_{L^2} 
\lesssim      |t| \|\chi^{\ell}_{out} |D|^\frac12  (r_{\ell},w_{\ell})\|_{L^2} + \|(g_\ell,k_\ell)\|_{L^2}..
\]
After some commutations this gives 
\[
2^{\ell} \| \alpha  \chi^{\ell}_{out}  (w_{\ell},r_{\ell}) \|_{L^2} \lesssim 2^{\frac{\ell}2} |t| 
\|\chi^{\ell}_{out}   (r_{\ell},w_{\ell})\|_{L^2} +  \|(g_\ell,k_\ell)\|_{L^2} +  \|(w_{\ell},r_{\ell})\|_{L^2}.
\]
Taking into account the localization of $\chi^{\ell}_{out}$, this yields 
\begin{equation}
2^{\ell} \| \alpha  \chi^{\ell}_{out}  (w_{\ell},r_{\ell}) \|_{L^2} \lesssim   
\|(g_\ell,k_\ell)\|_{L^2} +  \|(w_{\ell},r_{\ell})\|_{L^2}.
\end{equation}
By Bernstein's inequality this gives the pointwise bound
\begin{equation}\label{point-out}
 \chi^{\ell}_{out}  |(w_{\ell},r_{\ell})| \lesssim 2^{-\frac{\ell}2} |\alpha|^{-1}\left(\|(g_\ell,k_\ell)\|_{L^2} +  \|(w_{\ell},r_{\ell})\|_{L^2}\right).
\end{equation}

A similar computation, but with the roles of the two terms on the left
in \eqref{fixed-t-i} reversed gives
\begin{equation}
  |t| 2^{\frac{\ell}2}  \|   \chi^{\ell}_{in}  (w_{\ell},r_{\ell}) \|_{L^2} \lesssim   \|(g_\ell,k_\ell)\|_{L^2} 
+  \|(w_{\ell},r_{\ell})\|_{L^2}.
\end{equation}
By Bernstein's inequality this gives the pointwise bound
\begin{equation}\label{point-in}
 \chi^{\ell}_{in}  |(w_{\ell},r_{\ell})| \lesssim |t|^{-1}\left(\|(g_\ell,k_\ell)\|_{L^2} +  \|(w_{\ell},r_{\ell})\|_{L^2}\right).
\end{equation}

It remains to consider the intermediate region, where we produce instead 
a propagation estimate. Precisely, for $\chi^{\ell}_{med} (w_{\ell},r_{\ell})$ we estimate
\[
\begin{split}
\| (4 \alpha^2 \partial_\alpha - i t^2)\chi^{\ell}_{med} w_{\ell}\|_{L^2}
\lesssim &\ \| 2 \alpha (2\alpha \partial_\alpha \chi^{\ell}_{med} w_{\ell} - i t |D|^\frac12 \chi^{\ell}_{med} r_{\ell}) \|_{L^2} \\
& \ + \|   t  (\alpha  |D|^\frac12 \chi^{\ell}_{med} r_{\ell} - it \chi^{\ell}_{med} w_{\ell})\|_{L^2}
\\ \lesssim &\  |t| 2^{-\frac{\ell}2} (\|(g_\ell,k_\ell)\|_{L^2} + \|(r_{\ell},w_{\ell})\|_{L^2}),
\end{split}
\] 
and similarly for $r_{\ell}$. Applying 
\[
\frac{d}{dt} |u|^2 = 2 \Re \left[( \partial_\alpha - i \frac{t^2}{\alpha^2}) u \cdot \bar u \right]
\]
for $u = \chi^{\ell}_{med} w_{\ell} $ and $u = \chi^{\ell}_{med} r_{\ell} $, integrating from infinity and using 
the Cauchy-Schwarz inequality yields
\[
 \chi^{\ell}_{med} |(w_{\ell}, r_{\ell})|^2 \lesssim \alpha^{-2}  |t|  2^{-\frac{\ell}2}
(\|(g_\ell,k_\ell)\|_{L^2} + \|(r_{\ell},w_{\ell})\|_{L^2})\|(r_{\ell},w_{\ell})\|_{L^2}.
\]
Using this in the interesting region $|\alpha| \approx |t| {\ell}^\frac12$
and the inner and outer estimates away from it we obtain
\begin{equation}\label{point-med}
 \chi^{\ell}_{med} |(w_{\ell}, r_{\ell})| \lesssim |\alpha|^{-\frac12}   
(\|(g_\ell,k_\ell)\|_{L^2}^\frac12  \|(r,w_{\ell})\|_{L^2}^\frac12 +\|(r_{\ell},w_{\ell})\|_{L^2}).
\end{equation}

Now, we prove the bounds in the Lemma by dyadic summation.
There are several cases to consider:

\bigskip

{\bf Case 1:  $t^{-2} < 2^{\ell} < 1$}, where all three bounds coincide, and it suffices 
to prove \eqref{point1}. Assume that the two norms in the right hand side of \eqref{point1} 
are $\leq 1$. For  $\alpha$ we have three cases:

(a) $|\alpha| > t^{2}$, where we are in case B1 for all ${\ell}$. 
There we need only \eqref{point-out} to conclude that
\[
|(w_{[t^{-2},1]} ,r_{[t^{-2},1]})(\alpha)| 
\lesssim \sum_{2^\ell = t^{-2}}^1 2^{-\frac{\ell}2} |\alpha|^{-1} \approx |t| |\alpha|^{-1}.
\]

(b) $|t| < |\alpha| <  t^2$, where we are successively in case B1, B2 and B3.
There we use \eqref{point-out} \eqref{point-med} and \eqref{point-in} to  
conclude that
\[
\begin{split}
|(w_{[t^{-2},1]},r_{[t^{-2},1]})(\alpha)| \lesssim & \ \sum_{2^\ell = t^{-2}}^{\alpha^{-2} t^2}   |t|^{-1}
+ |\alpha|^{-\frac12} + \sum_{2^\ell=\alpha^2 t^{-2}}^12^{-\frac{\ell}2} |\alpha|^{-1} 
\\ \approx & \ |t|^{-1} |\log (1+t^2 |\alpha|^{-1})| + |\alpha|^{-\frac12}+  |\alpha|^{-1} \lesssim |\alpha|^{-\frac12}.
\end{split}
\]

(c)  $ |\alpha| < |t|$, where we are in case B3 for all ${\ell}$. 
There we need only \eqref{point-in} to conclude that
\[
|(w_{[t^{-2},1]},r_{[t^{-2},1]})(\alpha)| \lesssim \sum_{2^\ell = t^{-2}}^1 |t|^{-1} \lesssim t^{-1} |\log (|t|+2)|.
\]

\bigskip

{\bf Case 2: $1 < 2^{\ell}$.} Here  we have two subcases:

(i) $|\alpha| >| t|$, where we are in case B1 for all ${\ell}$.  
There we need only \eqref{point-out} to conclude that
\[
|(w_{>1} ,r_{> 1})(\alpha)| \lesssim \sum_{2^\ell=1}^\infty  2^{-\frac{\ell}2} |\alpha|^{-1}  
(\|(w,r)\|_{L^2} + \|(g,k)\|_{L^2}) \approx  |\alpha|^{-1}(\|(w,r)\|_{L^2} + \|(g,k)\|_{L^2}) ,
\]
which suffices for \eqref{point0}. In order to also obtain 
\eqref{pointk} we also use the Bernstein bound
\begin{equation}\label{bernstein}
| (w_{\ell},r_{\ell})| \lesssim 2^{\frac{\ell}2} \|(w_\ell,r_\ell)\|_{L^2} .
\end{equation}
Then we obtain
\[
\begin{split}
|\partial^k (w_{>1} ,r_{> 1})(\alpha)| \lesssim & \ \sum_{2^\ell=1}^\infty  \min\{ 2^{(k+\frac12)\ell} |\alpha|^{-1}  
 (\|(g,k)\|_{H^{-1}}+\|(w,r)\|_{H^{-1}}),    2^{-(k-\frac32)\ell} \|(w,r)\|_{\dot H^{2k+2}}\}
\\
 \lesssim &  \  |\alpha|^{-\frac12}  
 (\|(g,k)\|_{H^{-1}}+\|(w,r)\|_{H^{-1}})^\frac12 \|(w,r)\|_{\dot H^{2k+2}}^\frac12.
\end{split}
\]

(ii) $ |\alpha| < |t|$, where we are successively in cases  B1, B2 and B3.
There we use \eqref{point-out} \eqref{point-med} and \eqref{point-in} to  
conclude that
\[
\begin{split}
|(w_{>1},r_{>1})(\alpha)| \lesssim &\  \left( \sum_{2^\ell = 1}^{\alpha^{-2} t^{2}}   |t|^{-1} 
+ |\alpha|^{-\frac12} + \sum_{2^\ell = \alpha^{-2} t^2}^\infty 2^{-\frac{\ell}2} |\alpha|^{-1}\right) 
(\|(w,r)\|_{L^2} + \|(g,k)\|_{L^2})
\\ 
= & \ \left( |t|^{-1} |\log (1+ |t|/|\alpha|) | + |\alpha|^{-\frac12}+ |t|^{-1}\right)
   |\alpha|^{-\frac12} (\|(w,r)\|_{L^2} + \|(g,k)\|_{L^2})
\\ 
\lesssim  & \ |\alpha|^{-\frac12} (\|(w,r)\|_{L^2} + \|(g,k)\|_{L^2}),
\end{split}
\]
which suffices for \eqref{point0}. 

Finally, the bound \eqref{pointk} is obtained exactly as in (i) by
combining the last computation with the trivial pointwise bound for
$\partial^k(w_\ell,r_\ell)$ obtained from Bernstein's inequality
\eqref{bernstein}. Separating the contributions from cases B1, B2 and
B3 we obtain
\[
|\partial^k (w_{>1},r_{>1})(\alpha)| \lesssim I +II+III,
\]
where by \eqref{point-out} and \eqref{bernstein} we have
\[
\begin{split}
I = & \   \sum_{2^\ell = 1}^{\alpha^{-2} t^{2}}   \min \{ 2^{(k+1) \ell} |t|^{-1} (\|(w,r)\|_{H^{-1}}+ \|(g,k)\|_{H^{-1}}),
2^{-(k+\frac32) \ell}  \|(w,r)\|_{\dot H^{2k+2}}\} \\ \lesssim & \ 
|t|^{-\frac12}  |t|^{-\frac{1}{8k+10}} (\|(w,r)\|_{H^{-1}}+ \|(g,k)\|_{H^{-1}}+
 \|(w,r)\|_{\dot H^{2k+2}})
\end{split}
\]
by \eqref{point-med} we have 
\[
II =\alpha^\frac12 |t|^{-1}  (\|(w,r)\|_{H^{-1}} + \|(g,k)\|_{H^{-1}})  \|(w,r)\|_{\dot H^{2k+2}}^\frac12,
\]
and by \eqref{point-out} and \eqref{bernstein} we have
\[
\begin{split}
III = & \ \sum_{2^\ell = \alpha^{-2} t^2}^\infty \min\{ 2^{(k+\frac{1}2)\ell} |\alpha|^{-1} (\|(w,r)\|_{H^{-1}} + \|(g,k)\|_{H^{-1}}),
2^{-(k+\frac32) \ell}  \|(w,r)\|_{\dot H^{2k+2}}\} 
\\ 
\lesssim & \ |\alpha|^\frac12 |t|^{-1}  (\|(w,r)\|_{H^{-1}} + \|(g,k)\|_{H^{-1}})^\frac12 \|(w,r)\|_{\dot H^{2k+2}}^\frac12
\end{split}
\]

\end{proof}

\appendix

\section{Holomorphic equations}
\label{holom-eq}

 In this section we give an alternative derivation  for the  evolution equations for
water waves in conformal coordinates. They were first obtained in \cite{ov}, and also later in \cite{zakharov, zakharov2} but using a different set up. We use a holomorphic form of
the equations, as in \cite{zakharov, zakharov2}, but we compactify the equations even more, as we will show below. We also express the normal
derivative of the pressure on the boundary in terms of our variables.

We consider two-dimensional, irrotational gravity water waves in an
inviscid, incompressible fluid of infinite depth.  First, we discuss
the localized case on $\mathbb{R}$ in which the waves decay at infinity. The
spatially periodic case is almost identical, and we describe the
appropriate modifications afterwards.

\subsection{Holomorphic coordinates}
Suppose that at time $t$ the fluid occupies a spatial region $\Omega(t)\subset
\mathbb{R}^2$ whose simple nondegenerate boundary
$\Gamma(t)= \partial\Omega(t)$ approaches $y=0$ at infinity. 
 Then there  is a unique conformal map 
$\con(t) : \Ha \to \Omega(t)$ from the lower half-plane 
\[\Ha =
\left\{\alpha+i\beta : \beta < 0\right\}
\]
 onto $\Omega(t)$, with $x =
x(t,\alpha,\beta)$ and $y = y(t,\alpha,\beta)$, such that $z = x + i
y$ satisfies
\[
 z -  (\alpha + i \beta) \to 0 \qquad \text{as} \ \ \alpha + i\beta \to \infty.
\]

Since $\con(t)$ is conformal, we have
$x_\alpha = y_\beta$, $x_\beta = - y_\alpha$.
If $f(t,\cdot) : \Omega(t) \to \Cx$ is a time-dependent spatial function and
$g(t,\cdot) = f(t,\cdot)\circ \con(t) : \Ha \to \Cx$ is the corresponding conformal function,
then $g_t = f_t + x_t f_x + y_t f_y$, so
\begin{equation}
f_t = g_t - \frac{1}{j}\left(x_\alpha x_t + y_\alpha y_t\right) g_\alpha
-\frac{1}{j}\left( x_\beta x_t + y_\beta y_t\right)g_\beta,
\quad
j =x_\alpha^2 + y_\alpha^2.
\label{tder}
\end{equation}
Also, if $f + ig :\Ha \to \Cx$ is a holomorphic function
with boundary value $F + iG$ on the real axis $\beta=0$ that vanishes at infinity, then
$F = \Hi G$
where the Hilbert transform $\Hi$ is defined by
\[
\Hi f(\alpha) = \frac{1}{\pi} \mathrm{p.v.} \int_{-\infty}^\infty \frac{f(\alpha')}{\alpha-\alpha'} \, d\alpha',
\qquad
\Hi e^{ik\alpha} = -i (\sgn k) e^{ik\alpha}.
\]
We denote by $P = \frac{1}{2}\left(\Id - i\Hi\right)$ the projection onto boundary values of functions that are holomorphic in the lower
half-plane and vanish at infinity. That is, $P$ projects functions onto their negative wavenumber components.

\subsection{Water waves in holomorphic coordinates}
Let $\phi : \Omega(t) \to \mathbb{R}$ be the spatial velocity potential of the fluid, chosen so that it vanishes at infinity,
and $\psi = \phi\circ \con : \Ha \to \mathbb{R}$ the
corresponding conformal velocity potential,
\[
\psi(t,\alpha,\beta) = \phi\left(t, x(t, \alpha,\beta), y(t, \alpha,\beta)\right).
\]
Then $\psi$ is harmonic since $\phi$ is harmonic; we denote the conjugate function of $\psi(t,\alpha,\beta)$
by $\theta(t,\alpha,\beta)$.
The velocity components of the fluid $(u,v) = (\phi_x,\phi_y)$
are given in terms of $\psi$ by
\begin{equation}
u = \frac{1}{j}\left(x_\alpha\psi_\alpha + x_\beta \psi_\beta\right),
\quad
v = \frac{1}{j}\left(y_\alpha\psi_\alpha + y_\beta \psi_\beta\right).
\label{uveq}
\end{equation}

The conformally parametrized
equation of the free surface $\Gamma(t)$ is $x = X(t,\alpha)$, $y = Y(t,\alpha)$, where
$X(t,\alpha) = x(t,\alpha,0)$, $Y(t,\alpha) = y(t,\alpha,0)$.

To avoid any confusions, we emphasize that the variable $Y$ used through this section has a different meaning than elsewhere in the paper; it is the vertical component of the parametrized  free surface $Z(t,\alpha)$.

Since $(x-\alpha)+i(y-\beta)$
is holomorphic in the lower half-plane and vanishes at infinity, we have
\begin{equation}
X = \alpha + \Hi Y,\qquad  Y = - \Hi (X-\alpha).
\label{XYhilb}
\end{equation}
Let $\pot(t,\alpha) = \psi(t,\alpha,0)$ denote
the boundary value of the conformal velocity potential
and $\hpot(t,\alpha) = \theta(t,\alpha,0)$ the conjugate function,
where $\hpot = -\Hi \pot$ and
\begin{equation}
\left.\psi_\beta\right|_{\beta=0} = \Hi \pot_\alpha = -\hpot_\alpha.
\label{psieq}
\end{equation}
After these preliminaries, we transform the spatial boundary conditions for water-waves into conformal
coordinates.

\emph{Kinematic BC.}
A spatial normal to the free surface $\Gamma$ is $(-Y_\alpha,X_\alpha)$. The kinematic BC, that the normal
component of the velocity of the free surface is equal to the normal component of the fluid velocity, is
\[
\left(X_t, Y_t\right)\cdot (-Y_\alpha,X_\alpha) = \left(u, v\right)\cdot (-Y_\alpha,X_\alpha)
\qquad\mbox{on $\Gamma(t)$}.
\]
Using (\ref{uveq}) and (\ref{psieq}) in this equation and simplifying the result, we get
\begin{equation}
X_\alpha Y_t - Y_\alpha X_t = -\hpot_\alpha.
\label{kinBC1}
\end{equation}
In addition, the function $z_t/z_\alpha$
is holomorphic in $\Ha$ and decays at infinity, so the real part of its boundary value on the real axis
is the Hilbert transform of its imaginary part. After the use of (\ref{kinBC1}), this gives
the equation
\begin{equation}
X_\alpha X_t + Y_\alpha Y_t = -J \Hi\left[\frac{\hpot_\alpha}{J}\right].
\label{kinBC2}
\end{equation}
Solving (\ref{kinBC1})--(\ref{kinBC2}) for $X_t$, $Y_t$, we get
an expression for the velocity of a conformal point on the free surface
\begin{equation}
X_t = -\Hi\left[\frac{\hpot_\alpha}{J}\right] X_\alpha + \frac{\hpot_\alpha}{J} Y_\alpha,
\quad
Y_t = -\frac{\hpot_\alpha}{J} X_\alpha - \Hi\left[\frac{\hpot_\alpha}{J}\right] Y_\alpha.
\label{kinBC}
\end{equation}

\emph{Dynamic BC.}
Bernoulli's equation for the pressure $p$ in the fluid, with gravitational acceleration $g=1$, is
\begin{equation}
\phi_t + \frac{1}{2}|\nabla\phi|^2 + y + p = 0.
\label{bernoulli_eq}
\end{equation}
The arbitrary function of $t$ that may appear in this equation is zero since we assume that
$\phi$ vanishes at infinity and $p=0$ on the free surface which approaches $y=0$. 
The spatial form of the dynamic BC, without surface tension, is
\[
\phi_t + \frac{1}{2}|\nabla \phi|^2 + y = 0
\qquad\mbox{on $\Gamma(t)$}.
\]
Using (\ref{tder}) to compute $\phi_t$, evaluating the result at $\beta=0$, and using (\ref{psieq})--(\ref{kinBC2}), we find that
\[
\left.\phi_t\right|_{\beta=0} =  \pot_t + \Hi\left[\frac{\hpot_\alpha}{J}\right] \pot_\alpha
-\frac{1}{J}\hpot_\alpha^2.
\]
We also have
\[
\frac{1}{2}\left.|\nabla \phi|^2\right|_{\beta=0} 
= \frac{1}{2J} \left( \pot_\alpha^2 + \hpot_\alpha^2\right).
\]
Hence, the dynamic BC in conformal variables is
\begin{equation}
\pot_t + \Hi\left[\frac{\hpot_\alpha}{J}\right] \pot_\alpha
+ \frac{1}{2J} \left(\pot_\alpha^2 - \hpot_\alpha^2\right) + Y = 0.
\label{dynBC}
\end{equation}

To put these equations in holomorphic form, we define
\begin{equation}
Z = X + i Y,\quad Q = \pot + i\hpot,\quad F = P\left[\frac{Q_\alpha - \bar{Q}_\alpha}{J}\right],
\quad J = |Z_\alpha|^2.
\label{cvar}
\end{equation}
Then $Z$, $Q$, $F$ are the boundary values of functions that are holomorphic in the lower half-plane, and
$P[Z-\alpha] = Z-\alpha$, $P Q = Q$. The kinematic BC (\ref{kinBC}) is equivalent to
\begin{equation}
Z_t + F Z_\alpha = 0.
\label{heq1}
\end{equation}
Applying the holomorphic projection $P$
to the dynamic BC (\ref{dynBC}), using
Hilbert transform identities, and simplifying the result, we get that
\begin{equation}
Q_t + F Q_\alpha +  P\left[\frac{|Q_\alpha|^2}{J}\right] = i\left( Z-\alpha\right),
\qquad J = |Z_\alpha|^2.
\label{heq2}
\end{equation}
Thus, the holomorphic equations are (\ref{heq1})--(\ref{heq2}).

\subsection{ The normal derivative of the pressure}

In this subsection, for comparison purposes, we compute the normal derivative of the pressure in terms of our variables. This played a role in the subject  as the Taylor
sign condition 
\[
\dfrac{\partial p}{\partial n}_{| \Gamma_t} < 0
\]
was identified as necessary for the well-posedness of the water wave equation, see \cite{taylor}.  In our context this is automatically satisfied, see  the discussion at the end of this section.  

 From Bernoulli's equation (\ref{bernoulli_eq})  we have that
\[
-\frac{\partial p}{\partial n} = - \frac{1}{J} \left.p_\beta\right|_{\beta=0} =
\frac{1}{J}\left.\partial_\beta \left(\phi_t + \frac{1}{2}|\nabla \phi|^2 + y\right)\right|_{\beta=0}.
\]
Converting $\beta$-derivatives to $\alpha$-derivatives and using the evolution equations, we find that
\begin{align*}
\left.\partial_\beta \phi_t\right|_{\beta=0} &=\partial_\alpha\left[-\Theta_t
+\frac{\Psi_\alpha}{J}
\left(X_\alpha Y_t - Y_\alpha X_t\right)
 + \frac{\Theta_\alpha}{J} \left(X_\alpha X_t + Y_\alpha Y_t\right)\right]
\\
&= - \partial_\alpha\left[\Theta_t + \Theta_\alpha\Hi\left(\frac{\Theta_\alpha}{J}\right) + \frac{\Psi_\alpha\Theta_\alpha}{J}\right]
\\
&= - \partial_\alpha\left[\Hi\left(\frac{\Psi_\alpha^2 + \Theta_\alpha^2}{2J}\right) + X - \alpha\right].
\end{align*}
Since $\left.y_\beta\right|_{\beta=0} = X_\alpha$ and
\[
\left.|\nabla\phi|^2 \right|_{\beta=0} = \frac{\Psi_\alpha^2 + \Theta_\alpha^2}{J},
\]
we find that
\[
-J\frac{\partial p}{\partial n} = 1 +\frac{1}{2} \left.\left(\partial_\beta - \Hi\partial_\alpha\right)|\nabla\phi|^2 \right|_{\beta=0}.
\]

To put this equation in holomorphic form, we introduce
\[
r = \frac{q_\alpha}{z_\alpha},\qquad \partial = \frac{1}{2}\left(\partial_\alpha - i\partial_\beta\right),
\quad \bar{\partial} = \frac{1}{2}\left(\partial_\alpha + i\partial_\beta\right),
\]
where $|\nabla\phi|^2 = r \bar{r}$ and $\left.r\right|_{\beta=0} = R$ is defined in (\ref{defR}). Then,
using the fact that $\partial r = r_\alpha$ and $\bar{\partial} r = 0$, we get that
\[
-J\frac{\partial p}{\partial n} = 1 + i\left[\bar{P}(\bar{R} R_\alpha) - P(R \bar{R}_\alpha)\right] = 1 + a,
\]
where $a$ is defined in (\ref{defa}). Comparing this result with Wu
\cite{wu}, we see that up to Jacobian factors, our $1+a$ is
proportional to her $\mathfrak{a}$. Moreover, as shown in \cite{wu2} under the assumption of non-self intersecting boundary,
we have $a \ge 0$.  A shorter alternate proof of this fact is provided 
in our Lemma~\ref{regularity for a}; further we impose no condition on the self intersections of the curve $Z(t, \alpha)$.

\subsection{The periodic case}
In the spatially periodic case, the map $\con(t)$ is uniquely
determined by the requirement that the holomorphic function
$z(t,\alpha,\beta) - (\alpha + i\beta)$ is a periodic function of
$\alpha$, whose real part approaches zero\footnote{The imaginary part
  need not have zero mean even if the average height stays equal to
  zero.} as $\beta \to -\infty$.  It follows that $\Re Z(t,\alpha) -
\alpha$ has zero mean with respect to $\alpha$, otherwise the
holomorphic function would have nonzero limit.

In the original coordinates, the velocity field $u+iv$ is holomorphic,
periodic and bounded. Thus it has a limit $u_0+ iv_0$ as $\beta \to
\infty$.  Further, this limit is independent of time. By extension, in the holomorphic coordinates $Q_\alpha$ also
has a limit $u_0+ iv_0$ as $b \to \infty$.  Thus, we can normalize $Q$ 
by setting 
\[
Q = Q_0 + (u_0+iv_0)(\alpha + i\beta)  + c(t),
\]
where $Q_0$ is periodic with average zero, and $c(t)$ is a real
normalization constant needed for Bernoulli's law. One could continue the computations using $u_0$ and $v_0$ as constants of motion, but 
this is not needed because we can factor them out using a Galilean
transformation. From here on we set them equal to zero.

The relation \eqref{kinBC1} rests unchanged,
\begin{equation}
X_\alpha Y_t - Y_\alpha X_t = -\hpot_\alpha.
\label{kinBC1p}
\end{equation}
In integrated form this expresses the conservation of mass.
Consider  now both terms divided by $J$, then this becomes
\[
\Im \left( \frac{z_t}{z_\alpha}\right) = -\frac{\hpot_\alpha}{J} \qquad \text{on} \ \ \beta = 0.
 \]
 The function on the left is holomorphic and its real part has limit zero as $\beta$
 goes to infinity, so we can get its real part using the Hilbert
 transform.  Hence \eqref{kinBC2} also holds, and \eqref{kinBC}
 follows.  Similarly, the derivation of \eqref{dynBC} remainss unchanged.

To put the equations \eqref{kinBC} and \eqref{dynBC} in holomorphic form we 
keep the definition of  the operator $P$ as 
\[
P = \frac12(I-iH)
\]
even though it is no longer a projector, as it selects exactly half of the zero mode.
With the same notations as in \eqref{cvar}, the equations \eqref{heq1} and \eqref{heq2}
remains unchanged. 

Concerning the balance of averages in these two equations, we remark
that in \eqref{heq1} the terms $\Re Z_t$ and $F$ have purely imaginary
averages while $Z_\alpha$ has average $1$. The nontrivial average here
is that of $F$, which contributes to the motion of nonzero
frequencies. In the second equation \eqref{heq2}, the real part of the average of $Q$
is nonzero due to the integrating constant in Bernoulli's law;
however this plays a trivial role, as it does not affect any of the remaining equations.
Further, $R$ has no zero modes.

All equations in the first section of the paper remain unchanged, most
importantly the expressions for the frequency shift $a$, the advection velocity 
$b$ and the auxiliary function $M$.  Further, all estimates in Lemmas~\ref{regularity for a},
~\ref{l:b},~\ref{l:M} remain unchanged; only the zero mode estimates need to be added,
and those are straightforward. The normal form transformation also remains valid.

We next consider the linearized equations. The derivation of \eqref{lin(wr)0} 
is purely algebraic, so it stays unchanged. We remark that 
the average of $w$ is purely imaginary, while the average of $r$ is the same as the 
average of $q$ and is purely real.

There is some choice to be made when writing the projected equations \eqref{lin(wr)}.
The operator $P$ defined as above is no longer a projector, so we can no longer 
use it directly. The new question that arises here is how we treat the zero modes. 
For that we introduce some variants of $P$ which differ in how the zero modes are treated.
Defining $P_0$ as the projection onto the zero modes, we define the projectors 
\[
P^\sharp = P - \frac12 P_0, \qquad P^r = P^\sharp + \Re P_0, \qquad P^i = P^\sharp + i \Im P_0,
\]
and similarly $\bar P^\sharp$, $\bar P^r$ and $\bar P^i$. We have the relations
\[
P = P^i + \bar P^r = P^r + \bar P^i = P^\sharp + \bar P^\sharp + P_0, \qquad 
P^i \bar P^r = P^r \bar P^i = 0, \qquad P^i = -i P^r i.
\]
With these notations, it is natural to project the first equation
using $P^i$, and the second using $P^r$.  Thus instead of \eqref{lin(wr)} we write
\begin{equation}\label{p-lin(wr)}
\left\{
\begin{aligned}
& (\partial_t + P^i b \partial_\alpha) w  + P^i \left[ \frac{1}{1+\bar \W} r_\alpha\right]
+  P^i \left[ \frac{R_{\alpha} }{1+\bar \W} w \right] = P^i \mathcal{G}(  w, r),
 \\
&(\partial_t + P^r \partial_\alpha)  r  - i P^i\left[ \frac{1+a}{1+\W} w\right]  =
 P^r \mathcal{K}( w,r).
\end{aligned}
\right.
\end{equation}
The quadratic part of $\mathcal G$ has real average, and  $\mathcal K$ has 
imaginary average, both of which get projected out. So 
\[
P^i \mathcal G^{(2)} = P^\sharp[ R \bar w_\alpha - \W \bar r_\alpha] , 
\qquad P^i \mathcal K^{(2)} = - P^\sharp[ R \bar r_\alpha] .
\]
After a similar modification in \eqref{lin(wr)inhom},
the statement and the proof of Proposition~\ref{plin-short} remain largely unchanged;
the difference is that $P$ gets replaced by $P^i$ in $err_1$.

Moving on to the cubic estimates, the equation \eqref{lin(wr)inhom3} is replaced by 
\begin{equation}\label{p-lin(wr)inhom3}
\left\{
\begin{aligned}
& (\partial_t + P^i b \partial_\alpha) w  + P^i \left[ \frac{1}{1+\bar \W} r_\alpha\right]
+  P^i \left[ \frac{R_{\alpha} }{1+\bar \W} w \right] =
P^\sharp[ R \bar w_\alpha - \W \bar r_\alpha] +G,
 \\
&(\partial_t + P^r \partial_\alpha)  r  - i P^i\left[ \frac{1+a}{1+\W} w\right]  =
- P^\sharp[ R \bar r_\alpha] +K.
\end{aligned}
\right.
\end{equation}
Also, the cubic energy needs to be modified. Precisely, the zero modes
of $w$ and $r$ do not affect the quadratic terms on the right hand
side above. Hence, the cubic energy correction should not involve these zero modes
either,
\begin{equation}\label{p-elin3}
\Elint(w,r) = \int_{\R} (1+a) |w|^2 + \Im (r \bar  r_\alpha)
+ 2 \Im (\bar R w^\sharp r_\alpha) -2\Re(\bar{\W} (w^\sharp)^2)\
d\alpha.
\end{equation}
With this modification, the result in Proposition~\ref{plin-long} remains valid, and 
the proof applies with minor modifications.

The higher energies in the periodic case are even more similar to the nonperiodic case,
since differentiation eliminates the zero modes, and the frequency localization 
is achieved using only  $P^\sharp$ and $\bar P^\sharp$.

An alternative to the above scheme is to select just the negative wave
numbers in the linearized equation. Then we lose the evolution of the
average of the imaginary part of $w$.  This is not so significant
since we have the conservation of mass relation
\[
 \int  Y X_\alpha d\alpha = const,
\]
where the (time independent) constant on the right can be set
arbitrarily (say to zero) by a vertical translation of the
coordinates,
\[
(Z,Q) \to (Z+ ic, Q-ct).
\]
This gives
\[
\int  Y d\alpha = i \int (\bar Z-\alpha)(Z_\alpha -1) \ d\alpha  + const,
\]
where the average of $Y$ plays no role on the right.  This shows that
also for the linearized equation, the average of $w$ is determined by
the initial data and the negative frequencies of $w$ and $W$,
\[
\int \Im w d\alpha = i  \int \bar w W_\alpha + \bar W w_\alpha \  d\alpha  + const.
\]
Again, the constant can be removed as the pair $(i,t-iR)$ solves the
linearized problem.  It follows that the contribution of the average
of $w$ to the linearized equations can be viewed as cubic and higher.


\section{ Norms and multilinear  estimates} \label{s:multilinear}

Here we prove some of the estimates used in
Section~\ref{s:linearized}, and Section~\ref{s:ee}.  We use a standard 
 Littlewood-Paley decomposition in frequency
\begin{equation*}
1=\sum_{k\in \mathbf{Z}}P_{k},
\end{equation*}
where the multipliers $P_k$ have smooth symbols localized at frequency $2^k$.

A good portion of our analysis happens at the level of homogeneous
Sobolev spaces $\dot{H}^{s}$, whose norm is given by
\begin{equation*}
\Vert f\Vert_{\dot{H}^{s}}\sim \Vert ( \sum_{k}\vert 2^{ks} P_{k}f\vert^2 )^{1/2}  \Vert _{L^2}=
\| 2^{ks} P_k f \|_{L^2_x \ell^2_k}.
\end{equation*}
We will also use  the Littlewood-Paley
square function and its restricted version,
\begin{equation*}
\displaystyle S(f)(x):=\bigg( \sum_{k\in {\mathbf Z}} |P_k(f)(x)|^2\bigg)^\frac{1}{2}, \qquad S_{>k}(u) = ( \sum_{j > k} |P_j u|^2 )^\frac12.
\end{equation*}
The Littlewood-Paley inequality  is recalled below
\begin{equation}\label{lp-square}
   \displaystyle \|S(f)\|_{L^p({\mathbf R})}\simeq_{p} \|f\|_{L^p({\mathbf R})}, \qquad  1<p<\infty.
   \end{equation}
By duality this also yields the estimate
\begin{equation}
\label{useful}
\Vert \sum_{k\in \mathbf{Z}}P_{k}f_{k}\Vert_{L^p}\lesssim
\Vert \sum_{k\in \mathbf{Z}}(\vert f_{k}\vert ^2)^{1/2}\Vert_{L^p}, \qquad 1 < p < \infty.
\end{equation}

The $p= 1$ version of the above estimate for the Hardy space $ H_1$ is
\begin{equation}\label{h1-square}
\Vert f\Vert_{H_1}\simeq \| S(f)  \|_{L^1_x \ell^2_k},
\end{equation}
which by duality implies the BMO bound
\begin{equation}
\label{useful-bmo}
\Vert \sum_{k\in \mathbf{Z}}P_{k}f_{k}\Vert_{BMO}\lesssim
\Vert S(f)\Vert_{L^\infty}.
\end{equation}
The square function characterization of BMO is slightly different,
\begin{equation}\label{bmo-square}
\| u\|_{BMO}^2 \approx \sup_k \sup_{|Q|=2^{-k}} 2^k \int_Q |S_{>k} (u)|^2 \, dx.
\end{equation}
We will also need the maximal function bound
\begin{equation}
\label{useful-max}
\Vert P_{< k}f \Vert_{L^2_x L^\infty_k}\lesssim  \|f \|_{L^2} , \qquad 1 < p < \infty.
\end{equation}

\subsection{Coifman-Meyer and and Moser type estimates.}

In the context of bilinear estimates a standard tool is to consider a Littlewood-Paley 
paraproduct type decomposition of the product of two functions,
\[
 f g = \sum_{k \in \Z} f_{<k-4} g_k +  \sum_{k \in \Z} f_{k} g_{<k-4} + \sum_{|k-l| \leq 4}
f_k g_l := T_f g + T_g f + \Pi(f,g).
\]
Here and below we use the notation $f_k = P_k f$, $f_{<k} = P_{<k} f$, etc.
By a slight abuse of notation, in the sequel we will omit the frequency separation 
from our notations in bilinear Littlewood-Paley decomposition; for instance instead of the 
above formula we will use the shorter expression
\[
  f g = \sum_{k \in \Z} f_{<k} g_k +  \sum_{k \in \Z} f_{k} g_{<k} + \sum_{k \in \Z} f_k g_k.
\]

Away from the exponents $1$ and $\infty$ one has a full set of estimates
\begin{equation}\label{CM}
\| T_f g\|_{L^r} + \| \Pi(f,g)\|_{L^r} \lesssim \|f\|_{L^p} \|g\|_{L^q}, \qquad \frac{1}r = \frac{1}{p} +
\frac{1}{q}, \qquad 1 < p,q,r < \infty.
\end{equation}
Corresponding to $q = \infty$ one also  has a BMO estimate 
\begin{equation}\label{CM-BMO}
\| T_f g\|_{L^p} + \| \Pi(f,g)\|_{L^p} \lesssim \|f\|_{L^p} \|g\|_{BMO},  \qquad 1 < p < \infty,
\end{equation}
which in turn leads to the commutator bound
\begin{equation}\label{CM-com}
\|[P, g] f \|_{L^p}  \lesssim \|f\|_{L^p} \|g\|_{BMO},  \qquad 1 < p < \infty.
\end{equation}
For $p = 2$ we also need an extension of this, namely

\begin{lemma}\label{l:com}
 The following commutator estimates hold:
\begin{equation}  \label{first-com}
\Vert |D|^s \left[ P,R\right] |D|^\sigma w \Vert _{L^2}\lesssim \Vert
  |D|^{\sigma+s} R\Vert_{BMO} \Vert w\Vert_{L^2}, \qquad \sigma \geq 0, \ \ s \geq 0,
\end{equation}
\begin{equation}\label{second-com}
\Vert |D|^s  \left[ P,R\right] |D|^\sigma w \Vert _{L^2}\lesssim \Vert
  |D|^{\sigma+s}  R\Vert_{L^2} \Vert w\Vert_{BMO}, \qquad  \sigma > 0, \ \ \, s \geq 0.
\end{equation}
\end{lemma}
We remark that later this is applied to functions which are holomorphic/antiholomorphic,
but that no such assumption is made above.
\begin{proof}
  If $\sigma = s = 0$ then \eqref{first-com} is the classical commutator
  estimate of Coifman and Meyer \eqref{CM-com}, so we take $\sigma + s
> 0$. We
  consider the usual paradifferential decomposition, and observe that
  the expression $\left[ P,R\right] |D|^\sigma w$ vanishes if the
  frequency of $w$ is much larger than the frequency of $R$. For the
  remaining frequency balances we discard $P$, and we are left with
  having to estimate the expressions
\[
f_{hh} = \sum_k 2^{(\sigma+s) k} (2^{-k}|D|)^s (R_k w_k), \qquad f_{hl} =   \sum_k 2^{(\sigma+s) k} R_k
2^{-\sigma k} |D|^\sigma  w_{<k}.
\]
In the term $f_{hh}$ the $\sigma$ derivatives are already moved to
$R$, so this is bounded using \eqref{CM-BMO} if $s = 0$, and directly
if $s > 0$. For the remaining part we only need the infinity Besov
norm of $R_k$, as
\[
\begin{split}
\| f_{hl}\|_{L^2}^2 \lesssim & \ \sum_k \| 2^{(\sigma+s) k} R_k \|_{L^\infty}^2
\| 2^{-\sigma k} |D|^\sigma  w_{<k}\|_{L^2}^2
\lesssim  \sup_k \| 2^{(\sigma+s) k} R_k \|_{L^\infty}^2
\sum_k \| 2^{-\sigma k} |D|^\sigma  w_{<k}\|_{L^2}^2
\\ \lesssim & \ \Vert
  |D|^{\sigma+s} R\Vert_{BMO}^2 \Vert w\Vert_{L^2}^2.
\end{split}
\]
The proof of \eqref{second-com} is similar.
\end{proof}

Next we consider some similar product type estimates involving $BMO$ and
$L^\infty$ norms. We define 
\[
\| u\|_{BMO^\frac12}  = \| |D|^\frac12 u\|_{BMO}.
\]
Then
\begin{proposition}\label{p:bmo}
a) The following estimates hold:
\begin{equation}\label{bmo-bmo}
\| \sum_k    u_k v_k \|_{BMO} \lesssim \| u\|_{BMO} \|v\|_{BMO},
\end{equation}
\begin{equation}\label{bmo=infty}
\| \sum_k  (2^{-k}|D|)^\sigma(  u_k v_k)\|_{BMO} \lesssim \| u\|_{BMO} \|v\|_{\dot B_{\infty,\infty}^0}, \qquad
\sigma >  0,
\end{equation}
\begin{equation}\label{bmo-infty}
\| \sum_k  u_{<k} v_k\|_{BMO} \lesssim \| u\|_{L^\infty} \|v\|_{BMO},
\end{equation}
\begin{equation}\label{bmo>infty}
\| \sum_k  (2^{-k}|D|)^\sigma (u_{<k} v_k)\|_{BMO} \lesssim \| u\|_{\dot B_{\infty,\infty}^0 } \|v\|_{BMO},
\qquad \sigma > 0.
\end{equation}

b) The space $L^\infty \cap BMO^{\frac12}$ is an algebra,
\begin{equation}\label{bmo-alg}
\| uv\|_{BMO^{\frac12}} \lesssim \| u\|_{L^\infty} \|v\|_{BMO^\frac12}+
 \| v\|_{L^\infty} \|u\|_{BMO^\frac12},
\end{equation}

c) The following Moser estimate holds for a smooth function $F$:
\begin{equation}\label{bmo-moser}
\| F(u)\|_{BMO^{\frac12}} \lesssim_{\|u\|_{L^\infty}} \|u\|_{BMO^\frac12},
\end{equation}
\end{proposition}

\begin{proof}
a) For \eqref{bmo-bmo} we fix a cube $Q$, which by scaling can be taken to have size $1$.
Suppose first that $\sigma = 0$.
For $k > 0$ we use the square function estimate,
\[
\begin{split}
\|  \sum_{k > 0}  u_k v_k\|_{L^1(Q)} & \ \lesssim \ \| u_k v_k\|_{\ell^1_k L^1(Q)} \lesssim \| u_k\|_{\ell^2_k L^2(Q)} \| v_k\|_{\ell^2_k L^2(Q)}\lesssim \|S_{>0}(u)\|_{L^2(Q)}  \|S_{>0}(v)\|_{L^2(Q)}
\\ & \ \lesssim \|u\|_{BMO} \|v\|_{BMO}.
\end{split}
\]
For $k > 0$ we subtract the average and estimate the output in $L^\infty$,
\[
\| \sum_{k \leq  0}  u_k v_k  - (u_k v_k)_{Q}   \|_{L^\infty(Q)}
\lesssim \sum_{k< 0} \| \partial_\alpha (u_k v_k)\|_{L^\infty}
\lesssim \sum_{k < 0} 2^{k} \|u\|_{BMO} \|v\|_{BMO}.
\]
Adding the two we get
\[
\| \sum_{k }  u_k v_k  - (u_k v_k)_{Q}   \|_{L^1(Q)} \lesssim   \|u\|_{BMO} \|v\|_{BMO},
\]
and \eqref{bmo-bmo} follows.

The case $k \leq 0$ is similar in the proof of \eqref{bmo=infty}. For $k > 0$
we first eliminate directly the low frequency output,
\[
\| P_{< 0} \sum_{k > 0}   (2^{-k}|D|)^\sigma(  u_k v_k)\|_{BMO}
\lesssim  \sum_{k > 0} 2^{-\sigma k} \|u_k\|_{L^\infty} \|v_k\|_{L^\infty}
\lesssim \| u\|_{BMO} \|v\|_{\dot B_{\infty,\infty}^0}.
\]
For the high frequency output we consider a bump function  $\chi_Q$
adapted to $Q$, which is localized at
frequency less than $1$, and thus does not change the frequency
localization of the factors it multiplies in the sequel.
Then, we have
\[
\begin{split}
\| P_{> 0} \sum_{k > 0}   (2^{-k}|D|)^\sigma(  u_k v_k)\|_{L^2(Q)}
& \ \lesssim  \| \chi_Q P_{> 0} \sum_{k > 0}   (2^{-k}|D|)^\sigma(  u_k v_k)\|_{L^2}
\\ & \hspace{-1in}
\lesssim   \| [\chi_Q, P_{> 0} \sum_{k > 0}   (2^{-k}|D|)^\sigma](  u_k v_k)\|_{L^2}
+ \| P_{> 0} \sum_{k > 0}   (2^{-k}|D|)^\sigma( \chi_Q u_k v_k)\|_{L^2}.
\end{split}
\]
For the commutator term we gain a small power of $2^k$ so $L^\infty$ bounds
suffice. The remaining term is bounded in $L^2$ in terms of the square function
using orthogonality,
\[
\begin{split}
\| P_{> 0} \sum_{k > 0}   (2^{-k}|D|)^\sigma( \chi_Q u_k v_k)\|_{L^2}^2
\lesssim & \
\sum_{k > 0} \|  \chi_Q u_k \|_{L^2}^2 \|v_k\|_{L^\infty}^2
\lesssim \| \chi_Q S_{>0} (u) \|_{L^2}^2  \|v\|_{\dot B_{\infty,\infty}^0}
\\ \lesssim & \ \| u\|_{BMO}^2  \|v\|_{\dot B_{\infty,\infty}^0}.
\end{split}
\]

The argument for \eqref{bmo-infty} is similar, with the following modification in the case
$k >0$, which leads to $L^2$ rather than $L^1$ bounds:
\[
\begin{split}
\|\sum_{k > 0}  u_{<k} v_k\|_{L^2(Q)} \lesssim & \ \|\sum_{k > 0}  \chi_{Q} u_{<k} v_k\|_{L^2(Q)}
\lesssim \|u\|_{L^\infty} ( \sum_k \| \chi_Q v_k\|^2_{L^2(Q)})^\frac12
\\ \lesssim & \ \| u\|_{L^\infty}  \| \chi_Q S_{>0}(v)\|_{L^2(Q)}.
\lesssim \|u\|_{L^\infty} \|v\|_{BMO}.
\end{split}
\]
Finally, the bound \eqref{bmo>infty} is similar since
\[
\| (2^{-k}|D|)^\sigma u_{<k}\|_{L^\infty} \lesssim \| u\|_{\dot B_{\infty,\infty}^0}.
\]

b) With the same paradifferential decomposition as before we need to estimate
the terms
\[
f_{hh} = \sum_k |D|^{\sigma}  u_k v_k, \qquad f_{hl} =   \sum_k 2^{\sigma k} u_k
v_{<k}.
\]
For $f_{hl}$ we use \eqref{bmo-infty}, while for $f_{hh}$ we use \eqref{bmo-bmo}.

c) We write
\[
\begin{split}
F(u) = & \ \int_{-\infty}^\infty u_k F'(u_{<k})\, dk\\  = & \
\int_{-\infty}^\infty u_k P_{<k} F'(u_{<k})\, dk
+ \int_{-\infty}^\infty u_k P_k F'(u_{<k}) dk + \int_{-\infty}^\infty \sum_{j > 0} u_k P_{k+j} F'(u_{<k})\, dk.
\end{split}
\]
For $F'(u_{<k})$ we can use the chain rule to obtain
the bound
\[
\| P_{k+j} F'(u_{<k})\|_{L^\infty} \lesssim 2^{-N j}, \qquad j \geq 0.
\]
With $\sigma = \frac12$ we estimate $|D|^\sigma F(u)$.
The first term in $|D|^{\sigma} F(u) $ is
\[
f_1 = \int_{-\infty}^\infty (2^{\sigma k} u_k) P_{<k} F'(u_{<k}) \, dk,
\]
and is estimated in BMO exactly as in the proof of \eqref{bmo-infty}.

The second term is
\[
f_2 = \int_{-\infty}^\infty (2^{-k}|D|)^\sigma ( u_k P_k F'(u_{<k}))\, dk,
\]
and is estimated as in the proof of \eqref{bmo=infty}.

The last term is $\displaystyle\sum_{j > 0} f_{3,j}$, where
\[
  f_{3,j} = \int_{-\infty}^\infty 2^{\sigma
  k } u_k 2^{\sigma j} P_j F'(u_{<k}) \, dk.
\]
 The $k \leq 0$ case is easy; it follows using pointwise estimates.
For fixed $j$ and $k > 0$ we bound $f_{3,j}$ by
\[
\begin{split}
\|f_{3,j,> 0}\|_{L^2(Q)}^2 \lesssim & \  \| \chi_Qf_{3,j,> 0}\|_{L^2}^2 \lesssim
\int_{-\infty}^\infty \| \chi_Q 2^{\sigma k } u_k 2^{\sigma j}  P_j F'(u_{<k})\|_{L^2}^2 \, dk
\\ \lesssim  & \ 2^{(\sigma-N)j} \| \chi_Q S_{>0}(|D|^\sigma u)\|_{L^2}^2 \lesssim 2^{(\sigma-N)j}
\|u\|_{BMO^\sigma}^2.
\end{split}
\]
\end{proof}

A more standard algebra estimate and the corresponding Moser bound is as follows:
 \begin{lemma}
Let $\sigma > 0$. Then $\dot H^\sigma \cap L^\infty$ is an algebra, and
\begin{equation}
\| fg\|_{\dot H^\sigma} \lesssim \| f\| _{\dot H^\sigma} \|g\|_{L^\infty} +
\|f\|_{L^\infty}  \| g\| _{\dot H^\sigma}.
\end{equation}
In addition, the following Moser estimate holds for a smooth function $F$:
\begin{equation}\label{bmo-hs}
\| F(u)\|_{\dot H^\sigma } \lesssim_{\|u\|_{L^\infty}} \|u\|_{\dot H^\sigma }.
\end{equation}
\end{lemma}

We also need to consider some multilinear estimates. Our starting point is the bound
\[
\| f_1 \cdots f_n\|_{L^r} \lesssim \prod_{j=\overline{1,n}} \|f_j\|_{L^{p_j}}, \qquad \frac{1}{r} = \sum \frac{1}{p_j},
\qquad 1 \leq r, p_j \leq \infty .
\]
Adding derivatives, we need the following generalization:

\begin{lemma}\label{l:multi}
The following estimate holds for $\sigma > 0$ and $ 1 <  r, p_j^{(k)}  \leq \infty $:
\begin{equation}
  \| |D|^\sigma( f_1 \cdots f_n)\|_{L^r} \lesssim \sum_{k = 1}^n 
  \| |D|^\sigma f_k\|_{L^{p_k^{(k)}}}  \prod_{j \neq k} \|f_j\|_{L^{p_j^{(k)}}}, \qquad \frac{1}{r} = \sum \frac{1}{p_j^{(k)}}.
\end{equation}
The same bound holds if for $L^{p_k^{(k)}}$ is replaced by $BMO$ whenever 
$p_k^{(k)}= \infty$.
\end{lemma}
\begin{proof}
  By induction it suffices to consider the case $n=2$.  After a
  Littlewood-Paley decomposition we place the derivatives on the
  highest frequency factor and apply either \eqref{CM} or
  \eqref{CM-BMO}, or \eqref{bmo-bmo},\eqref{bmo-infty}.
\end{proof}

\subsection{Water-wave related bounds.}

Here we consider estimates for objects related to the water wave equations,
primarily the real phase shift $a$ and advection velocity $b$. We recall that these 
are given by
\[
a = 2 \Im P[R \bar R_\alpha], \qquad b = 2 \Re P \left[ \frac{R}{1+\bar \W}\right]
= 2 \Re ( R - P[R \bar Y]).
\]
These are estimated in terms of the control parameters $A$ and $B$ defined
in \eqref{A-def}, \eqref{B-def}, and in terms of the $H^s$ Sobolev norms of $\W$ and $R$.
In all nonlinear bounds the implicit constant is allowed to depend on $A$.

We begin with the auxiliary variable $Y= \dfrac{\W}{1+\W}$, which
inherits its regularity from $\W$ due to \eqref{bmo-moser} and \eqref{bmo-hs}:

\begin{lemma}\label{l:Y}
The function $Y$ satisfies the $BMO$ bound 
\begin{equation}\label{est:Y}
\| |D|^\frac12 Y\|_{BMO} \lesssim_A B,
\end{equation}
and the $\dot H^\sigma$ bound 
\begin{equation}
\|  Y\|_{ \dot H^\sigma} \lesssim_A \| \W\|_{ \dot H^\sigma}.
\end{equation}
\end{lemma}

We continue with bounds for $a$. In particular the positivity of $a$
is established, providing a short alternate proof to Wu's result in
\cite{wu2}:

\begin{proposition}\label{regularity for a}
  Assume that $R \in \dot H^{\frac12} \cap \dot H^{\frac32}$.  Then the real frequency-shift $a$
 is nonnegative and satisfies the $BMO$ bound
\begin{equation}\label{a-bmo}
\| a\|_{BMO}    \lesssim \Vert R\Vert_{BMO^\frac12}^2,
\end{equation}
and the uniform bound
\begin{equation}\label{a-point}
\| a\|_{L^\infty}    \lesssim \Vert R\Vert_{\dot B^{\frac12}_{\infty,2}}^2.
\end{equation}
 Moreover,
\begin{equation}\label{a-bmo+}
\Vert |D|^{\frac{1}{2}} a\Vert_{BMO}\lesssim \Vert R_\alpha \Vert_{BMO}
\Vert |D|^\frac12 R\Vert_{L^\infty}, \qquad
\Vert  a\Vert_{B^{\frac12,\infty}_2}\lesssim \Vert R_\alpha \Vert_{B^{\frac12,\infty}_2}
\Vert |D|^\frac12 R\Vert_{L^\infty},
\end{equation}
\begin{equation}\label{aflow}
\Vert (\partial_t+b\partial_{\alpha})a \Vert_{L^{\infty}}\lesssim AB,
\end{equation}
and
\begin{equation}\label{a-Hs}
\Vert  a\Vert_{\dot H^s}\lesssim \Vert  R \Vert_{\dot H^{s+\frac12}}
\Vert R\Vert_{BMO^\frac12}, \qquad s > 0.
\end{equation}
\end{proposition}
\begin{proof}
We recall that
$a =  i\left(\bar{P} \left[\bar{R} R_\alpha\right]- P\left[R\bar{R}_\alpha\right]\right)$.
Switching to the Fourier space, this leads to the representation
\begin{equation}
\hat{a}(\zeta) = \int_{\xi-\eta = \zeta} \min\{ \xi,\eta\} 1_{\{\xi,\eta  > 0\}} \hat R(\xi) \bar {\hat R}(\eta) d\xi.
\end{equation}
Here $\xi$ and $\eta$ are restricted to the positive real axis due to the fact that $R$
is holomorphic.

To prove the positivity of $a$ we represent the above kernel as
\[
 \min\{ \xi,\eta\}1_{\{\xi,\eta  > 0\}}  = \int_{M > 0} 1_{\{\xi > M\}} 1_{\{\eta > M\}} \, dM.
\]
Inserting this in the previous representation of $\hat a$ and inverting the Fourier transform
we obtain
\[
a = \int | 1_{|D| > M} R|^2 dM,
\]
and the positivity follows.

To prove both the $BMO$ bound and the pointwise bound for $a$ we use a
bilinear Littlewood-Paley decomposition,
\begin{equation}\label{a=lp}
a = \sum_k  i\left( \bar{R}_k R_{\alpha,<k} -  R_k\bar{R}_{\alpha,<k}\right) +
 i\left(\bar{P} \left[\bar{R}_k R_{\alpha,k}\right]- P\left[R_k\bar{R}_{\alpha,k}\right]\right).
\end{equation}
To estimate the first term in BMO we use directly the bound \eqref{bmo>infty} with $\sigma = 1$.
To estimate it in $L^\infty$ we use the Cauchy-Schwarz inequality,
\[
\| \sum_k   \bar{R}_k R_{\alpha,<k} \|_{L^\infty}^2 \lesssim
( \sum_k 2^{k} \|R_k\|_{L^\infty}^2)( \sum_k 2^{-k} \|R_{<k}\|_{L^\infty}^2)
\lesssim \|R\|_{B^{\frac12}_{\infty,2}}^2.
\]
For the second term in $a$ we rewrite the symbol of the bilinear form as
\[
\min\{ \xi,\eta\} = \frac12(\xi+\eta) - \frac12|\xi-\eta|,
\]
which allow us to rewrite it in the form
\[
\frac12 P  \sum_k   i\left( \bar{R}_k R_{\alpha,k} -  R_k\bar{R}_{\alpha,k}\right)
- |D| (\bar R_k R_k).
\]
Now the two  terms  are  estimated in BMO using \eqref{bmo-bmo}, respectively
\eqref{bmo=infty}, and in $L^\infty$ by the Cauchy-Schwarz inequality as above.

The proof of \eqref{a-bmo+} is essentially identical to the proof of \eqref{a-bmo}.

We continue with the proof of \eqref{aflow}, where we begin with the
decomposition in \eqref{a=lp}. For the first term in \eqref{a=lp} we
apply the time derivative to obtain the expression
\[
A_1= [b \partial_\alpha,\bar P_k]  \bar R R_{\alpha,<k}   +
 \bar R_k [ b \partial_\alpha, P_{<k} \partial_\alpha] R + i P_k \left( \frac{\bar \W-a}{1+\bar \W}\right)   R_{\alpha,<k} + i \bar R_k P_{<k} \partial_\alpha \left( \frac{\bar \W-a}{1+\bar \W}\right).
\]
In the first term of $A_1$ we split the commutator according to the
usual Littlewood-Paley trichotomy.  We get several terms:
\[
b_{<k,\alpha}  \bar R_k R_{\alpha,<k} + b_k \bar R_{<k,\alpha} R_{\alpha,<k}  +
2^m P_k (b_m \bar R_m)  R_{\alpha,<k} + b_m\bar{R}_{k,\alpha} R_{\alpha,<k}.
\]
In all cases we use the $B$ norm to estimate the highest frequency term, and the $A$
norm for the other two. The second term is similar; for comparison purposes we list
the ensuing terms:
\[
\bar R_k b_{<k,\alpha} R_{<k,\alpha} + 2^m \bar R_k P_{<k} \partial_\alpha( b_{m} R_{m})
+\bar R_k  b_m R_{<k,\alpha \alpha}.
\]
The third and fourth terms in $A_1$ require the bound
\[
\|\left( \frac{\bar \W-a}{1+\bar \W}\right)\|_{B^{\frac12,\infty}_2} \lesssim B,
\]
which follows  by combining the bounds \eqref{a-point} and \eqref{a-bmo+}  for $a$
with the similar bounds for $\W$ and $Y$.

Now we consider the last term in \eqref{a=lp}. This has two components, one of the form
$2^k R_k \bar R_k$ and the other of the form $|D|(R_k \bar R_k) $. The first component
yields an output
\[
A_2= [b \partial_\alpha,\bar P_k]  \bar R R_{\alpha,k}   +
 \bar R_k [ b \partial_\alpha, P_{k} \partial_\alpha] R + i P_k \left( \frac{\bar \W-a}{1+\bar \W}\right)   R_{\alpha,k} + i \bar R_k P_{k} \partial_\alpha \left( \frac{\bar \W-a}{1+\bar \W}\right),
\]
which is treated in exactly the same way as $A_1$.

The second component yields the slightly more involved output
\[
\begin{split}
A_3= & \ b \partial_\alpha |D|(R_k \bar R_k) - |D| (P_k (b R_{\alpha}) \bar R_k) -
|D| (R_k \overline{P_k (b R_{\alpha})} ) \\ & \ + i |D| \left( P_k \left( \frac{\bar \W-a}{1+\bar \W}\right)   R_{k}\right)  +
 i |D| \left( \bar R_k P_{k} \partial_\alpha \left( \frac{\bar \W-a}{1+\bar \W}\right)\right) .
\end{split}
\]
The last two terms are no different from above, but in the first three there is a
more delicate commutator estimate. We split $b = b_{<k} + b_{\geq k}$ and estimate the output
of $b_{\geq k}$ directly for each term using the $s = \frac12$ case of Lemma~\ref{l:b}.
The output of $b_{<k}$, on the other hand, is expressed as a commutator
\[
 [b_{<k}, |D|_{\leq k}] \partial_\alpha (R_k \bar R_k) + |D| \left( [b_{<k}, P_k]  R_{\alpha} \bar R_k\right)  +
|D| \left( R_k \overline{[b,P_k] R_{\alpha}} \right). 
\]
The last two terms are like $2^k |D|( b_{<k,\alpha} R_k \bar R_k)$ and can be estimated directly.
For the first term we bound $\partial_\alpha (R_k \bar R_k)$ in $L^\infty$ by $2^{-\frac{k}2}$,
and then we need to show that
\[
\|  [b_{<k}, |D|_{\leq k}]\|_{L^\infty \to L^\infty} \lesssim 2^{\frac{k}2} \||D|^\frac12 b\|_{BMO}.
\]
Indeed the kernel of  $[b_{<k}, |D|_{\leq k}]$ is bound by
\[
|b_{<k} (\alpha) - b_{<k}(\beta)| \frac{2^{2k}}{(1+ 2^k |\alpha -\beta|)^2} \lesssim
 \frac{2^{\frac32 k}}{(1+ 2^k |\alpha -\beta|)^\frac32},
\]
 which integrates to $2^{\frac{k}2}$.

Finally, \eqref{a-Hs} is a direct consequence of the commutator estimates in Lemma~\ref{l:com}.
\end{proof}

Next we consider $b$, for which we have the follwoing result
\begin{lemma}\label{l:b}
Let $s > 0$. Then the transport coefficient $b$ satisfies
\begin{equation}\label{b-bounds}
\||D|^s  b\|_{BMO} \lesssim_A \| |D|^s R\|_{BMO}, \qquad
\||D|^s  b\|_{L^2} \lesssim_A \| |D|^s R\|_{L^2}.
\end{equation}
In particular we have
\begin{equation}\label{b-bounds1}
\||D|^\frac12  b\|_{BMO} \lesssim_A A, \qquad
\|b_\alpha\|_{BMO} \lesssim_A B.
\end{equation}
\end{lemma}
\begin{proof}
Recall that
\[
b = \Re P[R(1-\bar Y)] = R- P(R \bar Y).
\]
Hence, it remains to estimate $\partial_\alpha P( R \bar Y)$.
Consider first the $BMO$ bound.
As before, the role of $P$ is to restrict the bilinear
frequency interactions to high - low, in which case we can use
the bound \eqref{bmo-infty}, and the high-high case, where \eqref{bmo=infty}
applies.

A direct argument, taking into account the same two cases, yields the $L^2$ bound.
\end{proof}

Next we consider the auxiliary expression $M$:


\begin{lemma}\label{l:M}
The function $M$ satisfies the pointwise bound
\begin{equation}\label{M-infty}
\| M\|_{L^\infty} \lesssim_A AB,
\end{equation}
as well as the Sobolev bounds
\begin{equation}\label{M-L2}
\| M\|_{\dot H^{n-\frac32}} \lesssim_A A \norm_n.
\end{equation}
\end{lemma}
\begin{proof}
For the pointwise bound we claim that 
\begin{equation}\label{M-infty1}
\| M\|_{L^\infty} \lesssim \|R\|_{B^{\frac34,\infty}_2} \|Y\|_{B^{\frac14,\infty}_2}.
\end{equation}
This suffices since each of the the right hand side factors is bounded
by $\sqrt{AB}$ by interpolation.  To achieve this we write $M$ in two different ways,
\[
M = \bar P[\bar R  Y_\alpha - R_\alpha \bar Y]+ P[R \bar Y_\alpha - \bar R_\alpha Y] = 
\partial_\alpha P_{<k+4} (\bar P [ \bar R Y] + P [R \bar Y])
- (\bar R_\alpha Y + R_\alpha \bar Y) .
\]
We apply a bilinear Littlewood-Paley decomposition and use the first expression 
above for the high-low interactions, and the second for the high-high interactions,
to write $M = M_1+M_2$ where
\[
\begin{split}
M_1 = & \ \sum_k [\bar R_k  Y_{<k,\alpha} - R_{<k,\alpha} \bar Y_k ]+ 
[R_k \bar Y_{<k,\alpha} - \bar R_{<k,\alpha} Y_k], 
\\
M_2 = & \ \sum_k
\partial_\alpha(\bar P [ \bar R_k Y_k] + P [R_k \bar Y_k])
- (\bar R_{k,\alpha} Y_k + R_{k,\alpha} \bar Y_k).
\end{split}
\]
We estimate the terms in $M_1$ separately; we show the argument for the first:
\[
\|   \sum_k \bar R_k  Y_{<k,\alpha}\|_{L^\infty} \lesssim 
\sum_{j \leq k} 2^{\frac34(j-k)} \|R_k\|_{B^{\frac34,\infty}_2} \|Y_j\|_{B^{\frac14,\infty}_2}
\lesssim \|R\|_{B^{\frac34,\infty}_2} \|Y\|_{B^{\frac14,\infty}_2}.
\]
For the first term in $M_2$ we note that the multiplier $\partial_\alpha P_{<k+4}  P$ 
has an $O(2^k)$ $L^\infty$ bound. Hence, we can estimate
\[
\|M_2\|_{L^\infty} \lesssim \sum_k 2^k \|R_k\|_{L^\infty} \|Y_k\|_{L^\infty} 
\lesssim \|R\|_{B^{\frac34,\infty}_2} \|Y\|_{B^{\frac14,\infty}_2}.
\]

For the $L^2$ bound we consider again all terms in $M_1$ and $M_2$ separately.
For an $M_1$ term we compute
\[
\|   \sum_k \bar R_k  Y_{<k,\alpha}\|_{\dot H^{n-\frac32}}^2 
\lesssim \sup_{k} 2^{-2k} \|  Y_{<k,\alpha}\|_{L^\infty}^2 \cdot 
\sum_{k} 2^{(2n-1)k} \|R_k\|_{L^2}^2 \lesssim \|Y\|_{L^\infty}^2 \|R\|_{\dot H^{\frac{n-1}2}}^2.
\] 
For  $M_2$  we compute
\[
\| M_2\|_{\dot H^{n-\frac32}}^2 \lesssim \sum_k 2^{(2n-1)k} \| Y_k R_k\|_{L^2}^2
\lesssim  \|Y\|_{L^\infty}^2 \|R\|_{\dot H^{\frac{n-1}2}}^2.
\]
\end{proof}

Finally, we also need a quadrilinear bound related to the energy estimates:

\begin{lemma}
The following estimate holds for holomorphic functions $R$, $w$ and $r$ :
\begin{equation}\label{i1}
\left| \! \int \!\! \bar{R}r_{\alpha}\M_bw_{\alpha} - \bar{R}w_{\alpha}\M_br_{\alpha} \,  d\alpha \right| \!
\lesssim\! (\||D|^\frac12 R\|_{BMO} \|b_\alpha\|_{BMO} + \|R_\alpha\|_{BMO}
\||D|^\frac12 b\|_{BMO}) \|w\|_{L^2} \|r\|_{\dot H^\frac12} \!
\end{equation}
\end{lemma}
\begin{proof}

We denote by $I_1$ the integral on the left.
In a first step we replace the holomorphic multiplication operator
$\M_{\bar P b}$ by the corresponding paraproduct operator
\[
T_{\bar P b} f = \sum_k \bar P b_{<k}  f_k.
\]
Thus, $I_1$ is replaced by
\[
I'_1 = \int -\bar{R}r_{\alpha}T_{\bar P b} w_{\alpha}
+\bar{R}w_{\alpha}T_{\bar P b} r_{\alpha} \, d\alpha .
\]
To estimate the difference $I'_1-I_1$ we observe that
 for holomorphic $f$ we have
\[
\M_{\bar P b} f = T_{\bar P b} f +  P(\sum_k \bar P b_k  f_k ).
\]
We use this for $f = w_\alpha$, respectively $f = r_\alpha$. Then
\[
I_1-I'_1 = \int  -\bar P[\bar{R}r_{\alpha}]  P\left[\sum_k \bar P b_k w_{k,\alpha}\right] +\bar P[\bar{R}w_{\alpha}] P\left[\sum_k\bar P b_k r_{k,\alpha}\right] \, d\alpha.
\]
Applying the bounds in Lemma~\ref{l:com} to estimate each of the four factors in $L^2$
we obtain
\[
|I_1-I'_1| \lesssim (\||D|^\frac12 R\|_{BMO} \|b_\alpha\|_{BMO} + \|R_\alpha\|_{BMO}
\||D|^\frac12 b\|_{BMO}) \|w\|_{L^2} \|r\|_{\dot H^\frac12}
\]
as needed.

It remains to estimate the integral $I'_1$. We take a Littlewood-Paley
decomposition and denote by $k,j,l$ the frequencies of $w$, $r$,
respectively $b$.  After  canceling the common terms we are left with
\[
I'_1 = \int \sum_{k \leq l < j} \bar{R}_j  r_{j,\alpha} b_{l} w_{k,\alpha} \, d\alpha
- \int \sum_{j \leq l < k} \bar{R}_k  r_{j,\alpha} b_{l} w_{k,\alpha} \, d\alpha := I_2 - I_3.
\]
The first sum $I_2$ can be estimated using only the infinity Besov norms for
$R_\alpha$ and $|D|^\frac12 b$,
\[
\begin{split}
|I_2| \lesssim & \ \| R_\alpha\|_{BMO} \||D|^\frac12 b\|_{BMO} \sum_{k  < j}
(j-k)  2^{\frac{k-j}2} \| r_j\|_{\dot H^\frac12} \|w_k\|_{L^2} \\ \lesssim & \ 
 \| R_\alpha\|_{BMO} \||D|^\frac12 b\|_{BMO}
\| r\|_{\dot H^\frac12} \|w\|_{L^2}.
\end{split}
\]
The argument for $I_3$ is slightly more involved since we cannot gain
rapid decay in $k-j$. Instead, we rewrite it as
\[
I_3 = \int \sum_{ k} \bar{R}_k   w_{k,\alpha}  \sum_{j \leq  l}
b_{l}  r_{j,\alpha} \, d\alpha -
\int \sum_{j,k \leq l} \bar{R}_k   w_{k,\alpha}
b_{l}  r_{j,\alpha} \, d\alpha
:= I'_3 - I''_3.
\]
The first term has a product structure, and we can bound each factor
in $L^2$ using Lemma~\ref{l:com} to obtain
\[
|I'_3| \lesssim \| R_\alpha\|_{BMO} \||D|^\frac12 b\|_{BMO}
\| r\|_{\dot H^\frac12} \|w\|_{L^2}.
\]
The second term is bounded in the same manner as $I_2$,
\[
\begin{split}
|I''_3| \lesssim & \  \| |D|^\frac12 R\|_{BMO} \|b_\alpha\|_{BMO} \sum_{j,k \leq l}
 2^{\frac{j-l}2+\frac{k-l}2 } \| r_j\|_{\dot H^\frac12} \|w_k\|_{L^2} \\ \lesssim & \
  \| |D|^\frac12 R\|_{BMO} \|b_\alpha\|_{BMO}
\| r\|_{\dot H^\frac12} \|w\|_{L^2}.
\end{split}
\]
Thus, the proof of \eqref{i1} is concluded.

\end{proof}

\end{document}